\documentclass{birkjour}
\PassOptionsToPackage{linktocpage}{hyperref}

\usepackage{amssymb}
\usepackage{cite}
\usepackage{hyperref}
\usepackage{mathtools}
\newtheorem{thm}{Theorem}[section]
\newtheorem{cor}[thm]{Corollary}
\newtheorem{lem}[thm]{Lemma}
\newtheorem{prop}[thm]{Proposition}
\theoremstyle{definition}
\newtheorem{defn}[thm]{Definition}
\theoremstyle{remark}
\newtheorem{rem}[thm]{Remark}
\newtheorem*{ex}{Example}
\numberwithin{equation}{section}

\newenvironment{enui}{
\begin{enumerate}
\renewcommand{\theenumi}{\alph{enumi}}
\renewcommand{\labelenumi}{\textnormal{(\theenumi)}}
}{\end{enumerate}}
\newenvironment{enuit}{
\begin{enumerate}
\renewcommand{\theenumi}{\(\tilde{\alph{enumi}}\)}
\renewcommand{\labelenumi}{\textnormal{(\theenumi)}}
}{\end{enumerate}}

\newcommand{\beql}[1]{\begin{equation}\label{#1}}
\newcommand{\eeq}{\end{equation}}
\newcommand{\benui}{\begin{enui}}
\newcommand{\eenui}{\end{enui}}
\newcommand{\blem}{\begin{lem}}

\newcommand{\blemnl}[2]{\blem[#1]\label{#2}}
\newcommand{\elem}{\end{lem}}

\newcommand{\rcor}[1]{Corollary~\ref{#1}}
\newcommand{\coref}[1]{\rcor{#1}}
\newcommand{\rthm}[1]{Theorem~\ref{#1}}
\newcommand{\thref}[1]{\rthm{#1}}
\newcommand{\rmref}[1]{Remark~\ref{#1}}
\newcommand{\lmref}[1]{Lemma~\ref{#1}}
\newcommand{\rsec}[1]{Section~\ref{#1}}
\newcommand{\csec}[1]{Section~#1}
\newcommand{\clem}[1]{Lemma~#1}
\newcommand{\cthm}[1]{Theorem~#1}
\newcommand{\ccor}[1]{Corollary~#1}
\newcommand{\cdef}[1]{Definition~#1}
\newcommand{\crem}[1]{Remark~#1}

\newcommand{\tie}{i.\,e.}
\newcommand{\tresp}{resp.}
\newcommand{\norm}[1]{\lVert #1\rVert}
\newcommand{\defeq}{:=}
\newcommand{\Mat}[1]{\begin{pmatrix}#1\end{pmatrix}}
\newcommand{\ip}[2]{\langle #1,#2\rangle}
\newcommand{\il}[1]{\item\label{#1}}
\DeclareMathOperator{\Rstr}{Rstr}
\newcommand{\rstr}[2]{\Rstr_{#2}#1}
\newcommand{\set}[1]{\{#1\}}
\newcommand{\tcL}{{\widetilde{\mathfrak{L}}}}
\newcommand{\abs}[1]{\lvert #1\rvert}
\newcommand{\zitaa}[2]{\cite[#2]{#1}}
\newcommand{\ko}[1]{\overline{#1}}
\DeclareMathOperator*{\prodr}{\overset{\curvearrowright}{\prod}}

\newcommand{\N}{{\mathbb{N}}}
\newcommand{\Z}{\ensuremath{\mathbb Z}}
\newcommand{\cB}{{\mathcal{B}}}
\newcommand{\cC}{{\mathfrak{C}}}
\newcommand{\cS}{{\mathcal{S}}}
\newcommand{\cL}{{\mathfrak{L}}}
\newcommand{\fL}{{\mathfrak{L}}}
\newcommand{\cD}{{\mathfrak{D}}}
\newcommand{\cQ}{{\mathcal{Q}}}
\newcommand{\cR}{{\mathcal{R}}}
\newcommand{\cN}{{\mathfrak{N}}}
\newcommand{\fM}{{\mathfrak{M}}}
\newcommand{\cF}{{\mathfrak{F}}}
\newcommand{\cM}{{\mathfrak{M}}}
\newcommand{\cU}{{\mathcal{U}}}
\newcommand{\cH}{{\mathfrak{H}}}
\newcommand{\gH}{{\mathfrak{H}}}
\newcommand{\cG}{{\mathfrak{G}}}
\newcommand{\cK}{{\mathfrak{K}}}
\newcommand{\cE}{{\mathfrak{E}}}
\newcommand{\cP}{{\mathcal{P}}}
\newcommand{\cA}{{\mathcal{A}}}
\newcommand{\C}{\mathbb{C}}
\newcommand{\D}{\mathbb{D}}
\newcommand{\T}{\mathbb{T}}
\newcommand{\aagf}{\begin{eqnarray}}
\newcommand{\zzgf}{\end{eqnarray}}
\newcommand{\aai}{\begin{itemize}}
\newcommand{\zzi}{\end{itemize}}
\newcommand{\matr}[4]{\Mat{
#1 & #2 \\
#3 & #4}}

\newcommand{\cl}[2]{\Mat{#1 \\ #2}}

\newcommand{\dcol}[3]{\Mat{
            #1 \\
            #2  \\
            #3}}

\newcommand{\bcase}[4]{\begin{cases}
#1,&\text{ if }#2\\
#3,&\text{ if }#4
\end{cases}}
\newcommand{\te}{\theta}
\newcommand{\la}{\lambda}
\newcommand{\La}{\Lambda}
\newcommand{\Ga}{\Gamma}
\newcommand{\ga}{\gamma}
\newcommand{\dl}{\delta}
\newcommand{\ome}{\omega}
\newcommand{\al}{\alpha}
\newcommand{\be}{\beta}
\newcommand{\si}{\sigma}
\newcommand{\Dl}{\Delta}
\newcommand{\Om}{\Omega}
\newcommand{\wt}[1]{\widetilde{#1}}
\newcommand{\wh}[1]{\widehat{#1}}
\DeclareMathOperator{\rank}{rank}
\newcommand{\nnu}{\nonumber}
\newcommand{\ol}[1]{\overline{#1}}
\DeclareMathOperator{\Lin}{Lin}
\DeclareMathOperator{\col}{col}

\newcommand{\gB}{{\mathfrak B}}
\newcommand{\gF}{{\mathfrak F}}
\newcommand{\gG}{{\mathfrak G}}

\newcommand{\gL}{{\mathfrak L}}
\newcommand{\gM}{{\mathfrak M}}
\newcommand{\R}{{\mathbb R}}
\newcommand{\dEins}{{\mathbf{1}}}
\newcommand{\dT}{{\mathbb T}}
\newcommand{\dC}{{\mathbb C}}
\newcommand{\Pol}{{\mathcal Pol}}
\DeclareMathOperator{\sgn}{sgn}
\DeclareMathOperator{\Ran}{Ran}
\newcommand{\ul}{\underline}

\newcommand{\eklam}[1]{\left[#1\right]}
\newcommand{\gklam}[1]{\left\{#1\right\}}
\newcommand{\fref}[1]{\eqref{#1}}
\newcommand{\FKlam}[5]{ \rklam{ \fEl{#1}{#2} }_{#3 = #4}^{#5}}
\newcommand{\fKlam}[4]{ \FKlam{#1}{#2}{#2}{#3}{#4}}
\newcommand{\intervalO}[2]{ \rklam{ #1,#2 } }
\newcommand{\orth}[1]{ {#1}^{\bot} }
\newcommand{\inorm}[2]{ {\left\lVert  #1  \right\rVert}_{#2} }
\newcommand{\rklamPaar}[2]{ \rklam{ #1,#2 } }
\newcommand{\Inv}[1]{ { #1 }^{-1} }

\newcommand{\fEl}[2]{#1_{#2}}
\newcommand{\mkEl}[3]{#1_{ #2,#3 }}
\newcommand{\rklam}[1]{\left(#1\right)}
\newcommand{\rklamFunk}[2]{#1 \rklam{#2}}
\newcommand{\drklam}[1]{\left( #1 \right)}
\newcommand{\iEl}[2]{ #1^{\drklam{#2}} }
\newcommand{\qKomma}{\	,}
\newcommand{\Adj}[1]{#1^\ast}
\newcommand{\qPunkt}{\	\ldotp}

\begin{document}

%
%
%
%
%
%
%
%
%

\title[Characteristic Function, Schur Parameters and Pseudocontinuation]{Characteristic Function, Schur Parameters and Pseudocontinuation of Schur
functions}


\author[Dubovoy]{Vladimir K. Dubovoy}
\address{Department of Mathematics and Computer Science\\
    Karazin National University\\
    Kharkov\\
    Ukraine\\
    and\\
    Max Planck Institute for Human Cognitive and Brain Sciences\\
    Stephanstrasse~1A\\
    04103~Leipzig\\
    Germany}
\email{dubovoy.v.k@gmail.com}

\author[Fritzsche]{Bernd Fritzsche}
\address{Universit\"at Leipzig\\
    Fakult\"at f\"ur Mathematik und Informatik\\
    PF~10~09~20\\
    D-04009~Leipzig\\
    Germany}
\email{fritzsche@math.uni-leipzig.de}

\author[Kirstein]{Bernd Kirstein}
\address{Universit\"at Leipzig\\
    Fakult\"at f\"ur Mathematik und Informatik\\
    PF~10~09~20\\
    D-04009~Leipzig\\
    Germany}
\email{kirstein@math.uni-leipzig.de}

\author[M\"adler]{Conrad M\"adler}
\address{Universit\"at Leipzig\\
    Zentrum f\"ur Lehrer:innenbildung und Schulforschung\\
    Prager Stra\ss{}e 38--40\\
    D-04317~Leipzig\\
    Germany}
\email{maedler@math.uni-leipzig.de} 

\author[M\"uller]{Karsten M\"uller}
\address{Max Planck Institute for Human Cognitive and Brain Sciences\\
    Stephanstrasse~1A\\
    04103~Leipzig\\
    Germany}
\email{karstenm@cbs.mpg.de}



\date{January 1, 2004}
\dedicatory{Dedicated to Damir Zjamovich Arov}

\begin{abstract}
In \cite{D06} there is an approach to the investigation of the pseudo-continuability of Schur functions in terms of Schur parameters.
    In particular, there was obtained a criterion for the pseudocontinuability of Schur functions and the Schur parameters of rational Schur functions were described.
    This approach is based on the description in terms of the Schur parameters of the relative position of the largest shift and the largest coshift in a completely nonunitary contraction.
    It should be mentioned that these results received a further development in \cite{BDFK,DFK,MR4049686,MR2931932,MR2805420}.

    This paper is aimed to give a survey about essential results on this direction.
    The main object in the approach is based on considering a Schur function as characteristic function of a contraction (see \rsec{Sec1.2}).
    This enables us outgoing from Schur parameters to construct a model of the corresponding contraction (see \rsec{S2}).
    In this model, the relative position of the largest shift and the largest coshift in a completely nonunitary contraction is described in \rsec{sec3-20221123} and then, based on this model, to find characteristics which are responsible for the pseudocontinuability of Schur functions (see Sections~\ref{sec4-20221123} and~\ref{sec5}).
    The further parts of this paper (see Sections~\ref{sec6}--\ref{sec8}) admit applications of the above results to the study of properties of Schur functions and questions related with them.
\end{abstract}

\maketitle
\tableofcontents

\section{Introduction}

\subsection{Pseudocontinuability of Schur functions}

Let $\D:=\{\zeta\in\C:|\zeta|<1\}$ be the open unit disk of the complex plane $\C$. The symbol $\cS$ denotes the set of all Schur functions in $\D$, i.e. the set of all functions $\theta:\D\rightarrow\C$ which are holomorphic in $\D$ and satisfy the condition $|\theta(\zeta)|\leq 1$ for all $\zeta\in\D$. The symbol $\cR\cS$ (resp. $J$) stands for the subset of $\cS$ which consists of all rational (resp. inner) functions belonging to $\cS$. Thus, the intersection $\cR\cS\cap J$ consists of all finite Blaschke products.

Let $f$ be a function which is meromorphic in $\D$ and which has nontangential
boundary limit values (Lebesgue) a.e. on $\T:=\{\zeta\in\C:|\zeta|=1\}$.
Denote $\D_e:=\{\zeta:|\zeta|>1\}$ the exterior of the unit disk including
the point infinity. The function $f$ is said to admit \emph{ a pseudocontinuation
of bounded type into} $\D_e$ if there exist functions $\alpha,\beta\not\equiv 0$
which are bounded and holomorphic in $\D_e$ such that the boundary values of
$f$ and $ \hat{f}:=\frac{\alpha}{\beta}$ coincide a.e. on $\T$. From the Theorem of
Luzin--Privalov (see, e.g., \cite{K}) it follows that there is at most one pseudocontinuation.

The study of the phenomenon of pseudocontinuability is important in many
questions of analysis. We draw our attention to two of them. For more detailed
information we refer the reader to \cite{MR0270196}, \cite{RS},
\cite{A2}, \cite{CR}, \cite{N}.

In the Hardy space $H^2(\D)$ we consider the unilateral shift $U^\times$ which is
generated by multiplication with the independent variable $\zeta\in\D$, i.e.
$$(U^\times f)(\zeta)=\zeta f(\zeta),\ f\in H^2(\D).$$
The operator which is adjoint to $U$ is given by
$$(Vf)(\zeta)=\frac{f(\zeta)-f(0)}{\zeta},\ f\in H^2(\D).$$
If we represent a function $f\in H^2(\D)$ as Taylor series via
$$f(\zeta)=a_0+a_1\zeta+a_2\zeta^2+\dotsb+a_n\zeta^n+\dotsb,\ \zeta\in\D,$$
and identify $f$ with the sequence $(a_k)_{k=0}^\infty\in l^2$
then the actions of the operators $U^\times$ and $V$ (by preserving the notations) are
given by
$$U^\times:(a_0,a_1,a_2,a_3,\dotsc)\mapsto(0,a_0,a_1,a_2,\dotsc)$$
and
$$W:(a_0,a_1,a_2,a_3,\dotsc)\mapsto(a_1,a_2,a_3,a_4,\dotsc).$$
In view of a theorem due to Beurling (see, e.g., \cite{K}) the invariant subspaces
of the shift $U$ in $H^2(\D)$ are described by inner functions whereas a function
$f\in H^2(\D)$ is cyclic for $U^\times$ if and only if $f$ is outer. In this connection we note
that in view of a theorem due to Douglas, Shapiro and Shields \cite{MR0270196} a function
$f\in H^2(\D)$ is not cyclic for the backward shift $V$ if and only if it admits
a pseudocontinuation of bounded type in $\D_e$.

Following D.Z. Arov \cite{A3} we denote by $\cS\Pi$ the subset of all functions
belonging to $\cS$ which admit a pseudocontinuation of bounded type in $\D_e$.
We note that the set $J$ of all inner functions in $\D$ is a subset of $\cS\Pi$.
Indeed, if $\te\in J$, then the function $\hat{\te}(\zeta)=\ol{\te^{-1}(\frac{1}{\ol\zeta})},\ \zeta\in\D_e $
is the pseudocontinuation of $\te$.

It is known (see Adamjan/Arov\cite{A-A}, Arov\cite{A3}) that each function of the Schur class $\cS$
is realized as scattering suboperator (Heisenberg scattering function) of a
corresponding unitary coupling, i.e., as characteristic function of unitary colligation.
D.Z. Arov indicated the important role of the
class $\cS\Pi$ in the theory of scattering with loss (see Arov\cite{A2}, \cite{A3}, \cite{A4}).
In this connection the following result is essential for our subsequent considerations.

\begin{thm}[Arov\cite{A2}, De~Wilde\cite{DW}, Douglas/Helton\cite{DH}]\label{thm4.1}
A function $\te$ belongs to the class $\cS\Pi$ if and only if there exists a
$2\times 2$ inner \textnormal{(}in $\D$\textnormal{)} matrix function $\Om$ which satisfies
\aagf\Om(\zeta)=\matr{\chi(\zeta)}{\phi(\zeta)}{\psi(\zeta)}{\te(\zeta)},\ \zeta\in\D.\label{4.1}\zzgf
\end{thm}

\begin{defn}\label{de4.2}
Let
$$w_a(\zeta):=\bcase{\frac{|a|}{a}\frac{a-\zeta}{1-\ol a\zeta}}{a\in\D\setminus\{0\},}{\zeta}{a=0}$$
denote the elementary Blaschke factor associated with $a$. By an
elementary $2\times 2$--Blaschke--Potapov factor we mean a
$2\times 2$--inner (in $\D$) matrix function of the form \aagf
b(\zeta):=I_2+(w_a(\zeta)-1)P\label{4.2A}\zzgf where $w_a(\zeta)$
is an elementary Blaschke factor whereas $P$ is an orthoprojection
in $\C^2$ of rank one, i.e., $P^2=P,\ P^*=P$ and $\rank P=1$. A
$2\times 2$--matrix function $B(\zeta)$ which is inner in $\D$ is
called a finite Blaschke--Potapov product if $B$ admits a
representation of the form \aagf
B(\zeta)=ub_1(\zeta)b_2(\zeta)\cdot\ldots\cdot
b_n(\zeta)\label{4.2B}\zzgf where $u$ is a constant unitary matrix
and $(b_k(\zeta))_{k=1}^n$ is a sequence of elementary $2\times
2$--Blaschke--Potapov factors.
\end{defn}

It follows easily from a result due to D.Z. Arov \cite{A2} that a
function $\te\in\cS$ is rational if and only if there exists a
finite Blaschke--Potapov product $\Om(\zeta)$ of the form
\eqref{4.1}. Thus, a function $\te\in\cS$ is rational if and only
if it can be represented as a block of a finite product of
elementary $2\times 2$--Blaschke--Potapov factors.

The above formulated criteria establish interesting connections between the pseudocontinuability of Schur functions and properties of other objects.
There it remains as open question which properties of a Schur function are responsible for its pseudocontinuability.
Against this background we mention that in \cite{D06} there was obtained a criterion for the pseudocontinuability of Schur functions which is formulated in terms of its Schur parameters (see \zitaa{D06}{\csec{5}}).
In particular, the Schur parameters of rational Schur functions were described (see \zitaa{D06}{\cthm{5.9}}).
Note that this approach is based on the description in terms of the Schur parameters of the relative position of the largest shift and the largest coshift in a completely nonunitary contraction (see \zitaa{D06}{\csec{3}}).
It should be mentioned that these results received a further development in \cite{BDFK,DFK,MR4049686,MR2931932,MR2805420}.

This paper is aimed to give a survey about essential results on this direction.
The main object in the approach 
is based on considering a Schur function as characteristic function of a contraction (see \rsec{Sec1.2}).
This enables us outgoing from Schur parameters to construct a model of the corresponding contraction (see \rsec{S2}).
In this model, the relative position of the largest shift and the largest coshift in a completely nonunitary contraction is described in \rsec{sec3-20221123} and then, based on this model, to find characteristics which are responsible for the pseudocontinuability of Schur functions (see Sections~\ref{sec4-20221123} and~\ref{sec5}).
The further parts of this paper (see Sections~\ref{sec6}--\ref{sec8}) admit applications of the above results to the study of properties of Schur functions and questions related with them.

\subsection{Schur function as a characteristic function of a unitary colligation}\label{Sec1.2}
Let \(T\) be a contraction acting in some Hilbert space \(\cH\), \tie{},\  \(\norm{T}\le 1\).
The operators 
\begin{align*}
D_T&\defeq \sqrt{I_\cH-T^*T}&
&\text{and}&
D_{T^*}&\defeq \sqrt{I_\cH-TT^*}
\end{align*}
are called \emph{the defect operators} of \(T\).
The closures of their ranges
\begin{align*}
\cD_T&\defeq \ol{D_T(\cH)}&
&\text{and}&
\cD_{T^*}&\defeq \ol{D_{T^*}(\cH)}
\end{align*}
are called \emph{the defect spaces} of \(T\).
The dimensions of these spaces
\begin{align*}
\dl_T&\defeq \dim{\cD_T}&
&\text{and}&
\dl_{T^*}&\defeq \dim{\cD_{T^*}}
\end{align*}
are called \emph{the defect numbers} of the contraction \(T\).
In this way, the condition \(\dl_T=0\) (\tresp{}\ \(\dl_{T^*}=0\)) characterizes isometric (\tresp{}\ coisometric) operators, whereas the conditions \(\dl_T=\dl_{T^*}=0\) characterize unitary operators.
Recall that an operator is called coisometric if its adjoint is isometric.

Starting from the contraction \(T\), we can always find Hilbert spaces \(\cF\) and \(\cG\) and operators \(F\colon\cF\to\cH\), \(G\colon\cH\to\cG\), and \(S\colon\cF\to\cG\) such that the operator matrix
\beql{1.2}
Y
=\Mat{T&F\\ G&S} \colon\cH\oplus\cF\to\cH\oplus\cG
\eeq
is unitary, \tie{},\ the conditions \(Y^*Y=I_{\cH\oplus\cF}\), \(YY^*=I_{\cH\oplus\cG}\) are satisfied.
Obviously, these equalities can be rewritten in the form
\beql{1.3}
\begin{aligned}
G^*G+T^*T&=I_{\cH},&S^*S+F^*F&=I_{\cF},&G^*S+T^*F&=0\\
TT^*+FF^*&=I_{\cH},&GG^*+SS^*&=I_{\cG},&TG^*+FS^*&=0.
\end{aligned}
\eeq
Note that in the general situation from \eqref{1.3} it follows \(G^*G=D_T^2\), \(FF^*=D_{T^*}^2\).
Hence,
\begin{align}\label{1.4}
\ol{F(\cF)}&=\cD_{T^*},&
\ol{G^*(\cG)}&=\cD_T.
\end{align}

\begin{defn}[{\cite{B}}]\label{de1.1}
The 7-tuple
\beql{E2.5}
\Dl=(\cH, \cF, \cG; T, F, G, S)
\eeq
consisting of three Hilbert spaces \(\cH,\cF,\cG\) and four operators \(T,F,G,S\) where
\begin{align*}
T&\colon\cH\to\cH,&
F&\colon\cF\to\cH,&
G&\colon\cH\to\cG,&
S&\colon\cF\to\cG
\end{align*}
is called a \emph{unitary colligation} (or more short \emph{colligation}) if the operator matrix \(Y\) given via \eqref{1.2} is unitary. 
\end{defn}

The operator \(T\) is called \emph{the fundamental operator} of the colligation \(\Dl\).
Clearly, the fundamental operator of a colligation is a contraction.
The operation of representing a contraction \(T\) as fundamental operator of a unitary colligation is called \emph{embedding} \(T\) in a colligation.
This embedding permits to use the spectral theory of unitary operators for the study of contractions (see \cite{SN})

The spaces
\begin{align}\label{eq: 1.7}
\cH_{\cF}&\defeq \bigvee_{n=0}^{\infty}T^nF(\cF),&
\cH_{\cG}&\defeq \bigvee_{n=0}^{\infty}T^{*n}G^*(\cG)
\end{align}
and their orthogonal complements \(\cH_{\cF}^\bot\defeq \cH\ominus\cH_{\cF}\), \(\cH_{\cG}^\bot\defeq \cH\ominus\cH_{\cG}\) play an important role in the theory of colligations. 
Therefore, we will use the decompositions
\begin{align}\label{1.7}
\cH&=\cH_{\cF}\oplus\cH_{\cF}^\bot
\intertext{and}
\label{1.201}
\cH&=\cH_{\cG}^\bot\oplus\cH_{\cG}.
\end{align}
in what follows.
From \eqref{1.4} it follows that the spaces \(\cH_\cF\) and \(\cH_\cG\) can also be defined on an alternate way, namely
\begin{align}\label{1.8}
\cH_{\cF}&\defeq \bigvee_{n=0}^{\infty}T^n\cD_{T^*},&
\cH_{\cG}&\defeq \bigvee_{n=0}^{\infty}T^{*n}\cD_T.
\end{align}
Consequently, the spaces \(\cH_\cF\) and \(\cH_\cG\) do not depend on the concrete way of embedding \(T\) in a colligation.
Note that \(\cH_\cF\) is invariant with respect to \(T\) whereas \(\cH_\cG\) is invariant with respect to \(T^*\).
This means that \(\cH_\cF^\bot\) and \(\cH_\cG^\bot\) are invariant with respect to \(T^*\) and \(T\), respectively.
Passing on to the kernels of the adjoint operators in equalities \eqref{1.8}, we obtain
\begin{align}\label{1.9}
\cH_\cF^\bot&=\bigcap_{n=0}^\infty\ker(D_{T^*}T^{*n}),&
\cH_\cG^\bot&=\bigcap_{n=0}^\infty\ker(D_TT^n).
\end{align}

\begin{thm}
The equalities
\[
\cH_\cF^\bot
=\{h\in\cH : \norm{T^{*n}h}=\norm{h}, \ n=1,2,3,\dotsc \}
\]
and
\[
\cH_\cG^\bot
=\{h\in\cH : \norm{T^nh}=\norm{h}, \ n=1,2,3,\dotsc \}
\]
hold true.
\end{thm}
\begin{proof}
For \(n=1,2,3,\dotsc\), clearly

\[
\norm{T^nh}^2
=\ip{T^{*n}T^nh}{h}
=\ip{T^{*n-1}T^{n-1}h}{h}-\ip{T^{*n-1}D_T^2T^{n-1}h}{h}.
\]
Now the first assertion follows from \eqref{1.9} and the equality
\begin{align*}
\norm{T^{n-1}h}^2-\norm{T^nh}^2
&=\norm{D_TT^{n-1}h}^2,&n&=1,2,3,\dotsc.
\end{align*}
Analogously, the second assertion can be proved.
\end{proof}

\begin{cor}\label{cor1.1} 
The space $\cH_\cF^\bot\;(\text{resp.\ }\cH_\cG^\bot)$ is characterized by the following properties:
\benui
\il{cor1.1.a} \(\cH_\cF^\bot\;(\text{resp.\ }\cH_\cG^\bot)\) is invariant with respect to \(T^*\) $(\text{resp.\ }T)$.
\il{cor1.1.b} $ \rstr{T^*}{\cH_\cF^\bot} \ (\tresp{}\ \rstr{T}{\cH_\cG^\bot})$
is an isometric operator.
\il{cor1.1.c} \(\cH_\cF^\bot\;(\text{resp.\ }\cH_\cG^\bot)\) is the largest subspace of \(\cH\) having the properties \eqref{cor1.1.a}, \eqref{cor1.1.b}.
\eenui
\end{cor}

From the foregoing consideration we immediately obtain the following result.

\begin{thm}\label{th1.2}
The equality
\[
\cH_\cF^\bot\cap\cH_\cG^\bot
=\{h\in\cH : \norm{T^{*n}h}=\norm{h}=\norm{T^nh}, \ n=1,2,3,\dotsc \}
\]
holds true.
\end{thm}
\begin{cor}\label{cor1.2}
The subspace \(\cH_\cF^\bot\cap\cH_\cG^\bot\) is the largest amongst all subspaces \(\cH'\) of \(\cH\) having the following properties:
\begin{enuit}
\il{cor1.2.a} \(\cH'\) reduces \(T\).
\il{cor1.2.b} $ \rstr{T}{\cH'} $ is a unitary operator.
\end{enuit}
\end{cor} 
A contraction \(T\) on \(\cH\) is called \emph{completely nonunitary} if there is no nontrivial reducing subspace \(\cL\) of \(\cH\) for which the operator $ \rstr{T}{\cL}$ is unitary.
Consequently, a contraction is completely nonunitary if and only if \(\cH_\cF^\bot\cap\cH_\cG^\bot=\set{0}\).
The colligation \(\Dl\) given in \eqref{E2.5} is called \emph{simple} if \(\cH=\cH_\cF\vee\cH_\cG\) 
and \emph{abundant} otherwise.
Hence, the colligation \(\Dl\) is simple if and only if its fundamental operator \(T\) is a completely nonunitary contraction.

Taking into account the Wold decomposition for isometric operators  (see, e.g., \cite[Ch.~I]{SN}), from \rcor{cor1.1} we infer the following result: 

\begin{thm}\label{th1.3}
Let \(T \in [\cH]\) be a completely nonunitary contraction. 
Then the subspace \(\cH_\cF^\bot\;(\text{resp.\ }\cH_\cG^\bot)\) is characterized by the following properties:
\benui
\il{th1.3.a} The subspace \(\cH_\cF^\bot\;(\text{resp.\ }\cH_\cG^\bot)\) is invariant with respect to \(T^*\) $(\text{resp.\ }T)$.
\il{th1.3.b} The operator $ \rstr{T^*}{\cH_\cF^\bot}\;(\text{resp.\ }\rstr{T}{\cH_\cG^\bot})$   is a unilateral shift.
\il{th1.3.c} \(\cH_\cF^\bot\;(\text{resp.\ }\cH_\cG^\bot)\) is the largest subspace of \(\cH\) having the properties \eqref{th1.3.a}, \eqref{th1.3.b}.
\eenui
\end{thm} 

We say that a unilateral shift \(V\colon\cL\to\cL\) \emph{is contained} in the contraction \(T\) if \(\cL\) is a subspace of \(\cH\) which is invariant with respect to \(T\) and \(\rstr{T}{\cL}=V\) is satisfied.

\begin{defn}\label{D1.9-1207}
Let \(T \in [\cH]\) be a completely nonunitary contraction. 
Then the shift \(V_T\defeq\ \rstr{T}{\cH_\cG^\bot}\)   is called \emph{the largest shift} contained in \(T\).
\end{defn}

By a \emph{coshift} we mean an operator the adjoint of which is a unilateral shift.
We say that a coshift \(\wt V\colon\tcL\to\wt \cL\) \emph{is contained} in \(T\) if the unilateral shift \(\wt V^*\) is contained in \(T^*\).
Then from \rthm{th1.3} it follows that the operator $ V_{T^*}= \rstr{T^*}{\cH_\cF^\bot}  $ is the largest shift contained in \(T^*\).
If \(\cH_\cG^\bot=\set{0}\) (\tresp{}\ \(\cH_\cF^\bot=\set{0}\)) we will say that the shift \(V_T\) (\tresp{}\ \(V_{T^*}\)) has multiplicity zero.

\begin{defn}
Let \(T \in [\cH]\) be a completely nonunitary contraction. 
Then the coshift \(\wt V_T\defeq (V_{T^*})^*\) is called \emph{the largest coshift} contained in \(T\).
\end{defn}

Let $T$ be a completely nonunitary contraction in $\mathfrak H$. From the above considerations it
follows that decomposition \eqref{1.7} corresponds to the block representation
\begin{align}\label{1.10}
T = \begin{pmatrix}{T_{\mathfrak F}}&*\\0&{\Tilde{V}_T}\end{pmatrix}.
\end{align}

Analogously, decomposition  \eqref{1.201} corresponds to the block representation
\begin{align}\label{1.11}
T=\begin{pmatrix}{V_T}&*\\0&{T_{\mathfrak G}}.\end{pmatrix}
\end{align}

The properties of a completely nonunitary contraction $T$ are determined in many aspects by the mutual position of the subspaces $\mathfrak H_{\mathfrak{F}}^\perp$ and $\mathfrak H_\mathfrak{G}^\perp$. According to this we note that in \cite[Subsection 3.2]{D06} (see also Sections~\ref{sec3-20221123} and~\ref{sec4-20221123} of this paper) a description of the mutual position of these subspaces was given.

\begin{thm}[{\cite[Theorem~1.9]{D06}}]\label{th1.4}
Let \(T \in [\cH]\) be a completely nonunitary contraction. 
Then the multiplicities of the largest shifts \(V_T\) and \(V_{T^*}\) are not greater than \(\dl_{T^*}\) and \(\dl_T\),  respectively.
\end{thm}

\begin{cor}[{\cite[Corollary~1.10]{D06}}]\label{cor1.9}
Let \(\Dl\) be a simple unitary colligation of type \eqref{E2.5}.
Denote \(\cL_0\) and \(\tcL_0\) the generating wandering subspaces for the largest shifts \(V_T\) and \(V_{T^*}\),  respectively.
Then \(\ol{P_{\cL_0}F(\cF)}=\cL_0\), \(\ol{P_{\tcL_0}G^*(\cG)}=\tcL_0\), where \(P_{\cL_0}\) and \(P_{\tcL_0}\) are the orthogonal projections from \(\cH\) onto \(\cL_0\) and \(\tcL_0\),  respectively.
\end{cor}

\begin{rem}\label{rm1.10}
In \cite[part III]{D82} it was shown that the multiplicity of the shift $V_T$ 
coincides
with $\dl_{T^*}$ if and only if the multiplicity of the shift $V_{T^*}$ 
coincides
with $\dl_T$. Moreover, all remaining cases connected with the inequalities
$$0\le\dim\cL_0<\dl_{T^*},\ 0\le\dim\wt{\cL}_0<\dl_T$$ are possible.
\end{rem}

\begin{defn}\label{de1.10}
Let $\Delta$ be the unitary colligation of the form
\eqref{E2.5}. The operator function
\begin{equation}\label{Nr.1.13}
\theta_{\Dl} (\zeta) := S + \zeta G (I_\gH - \zeta T)^{-1} F,\quad
\zeta\in\mathbb D,
\end{equation}
is called the characteristic operator function $($c.o.f.$)$ of the
colligation $\Delta$.
\end{defn}

By $\cS[\cF,\cG]$ it will be denoted the set of
operator-valued functions $\theta$ which are defined and
holomorphic on $\mathbb D$ and the values of which are contractive
operators from $[\cF,\cG]$ (the set of operator-valued Schur
functions).   

The next statements are very important (see, e.g., Brodskii
\cite{B}):

\begin{thm}\label{th1.12}
$($a$)$ The c.o.f. $\theta_\Delta$ of the unitary colligation
$\Delta$ belongs to the class $\cS [\cF,\cG]$.

$($b$)$
Let $\theta\in\cS[\cF,\cG]$. %
Then there exists a simple unitary colligation $\Delta$ of the form
\eqref{E2.5} such that $\theta$ is the c.o.f. of
$\Delta$, i.e., $\theta=\theta_\Delta$.
\end{thm}
\begin{defn}\label{de1.13}
Let $\Delta_k = (\cH_k, \cF, \cG; T_k, F_k, G_k, S_k), \; k=1,2,$ be
unitary colligations. The colligations $\Delta_1$ and $\Delta_2$ are
called unitarily equivalent if $S_1 =S_2$ and there exists a unitary
operator $Z : \cH_1\to\cH_2$ which satisfies $ZT_1 = T_2Z,\;
ZF_1=F_2,\, G_2 Z = G_1$.
\end{defn}

It can be easily seen that the c.o.f. of unitarily equivalent
colligations coincide. In this connection it turns out to be
important that the converse statement is also true (see, e.g.,
Brodskii \cite{B}).

\begin{thm}\label{th1.14}
If the c.o.f. of two simple colligations coincide, then the
colligations are unitarily equivalent.
\end{thm}

Theorem \ref{th1.12} and Theorem \ref{th1.14} tell us now that,
outgoing from the c.o.f., a simple unitary colligation can be
recovered up to unitary equivalence. Thus, the fundamental operator
$T$ of this colligation is also determined up to unitary
equivalence. For this reason the c.o.f. of a simple unitary
colligation will be also called the c.o.f. of its corresponding
fundamental operator $T$.

\subsection{Naimark dilations}

In this subsection we follow {\cite[Subsection 1.3]{D06}}. Let us consider interrelations between unitary colligations and Naimark
dilations of Borel measures on the unit circle $\T:=\{t\in\C:|t|=1\}$.

Let $\cE$ be a separable complex Hilbert space and denote $[\cE]$ the set of
bounded linear operators in $\cE$. We consider a $[\cE]-$valued Borel 
measure $\mu$ on $\T$. More precisely, $\mu$ is defined on the $\si-$algebra
$\cB(\T)$ of Borelian subsets of $\T$ and has the following properties:
\aai\item[(a)] For $\Dl\in\cB(\T)\ ,\ \mu(\Dl)\in[\cE]$.
\item[(b)] $\mu(\emptyset)=0$.
\item[(c)] For $\Dl\in\cB(\T)\ ,\ \mu(\Dl)\ge0$.
\item[(d)] $\mu$ is $\si-$additive with respect to the strong operator 
convergence.\zzi
The set of all $[\cE]-$valued Borel measures on $\T$ is denoted by  
$\cM(\T,\cE)$.

\begin{defn}
Let $\mu\in\cM(\T,\cE)$ and assume $\mu(\T)=I$. Denote $(s_n)_{n \in \Z}$
the sequence of Fourier coefficients of $\mu\ ,$ i.e.
\aagf s_n:=\int\limits_\T t^{-n}\mu(dt)\ ,\ n\in\Z\ .\label{1.19}\zzgf
By a Naimark dilation of the measure $\mu$ we mean an ordered triple
$(\cK,\cU,\tau)$
having the following properties.
\aai\item[1.] $\cK$ is a separable Hilbert space over $\C$.
\item[2.] $\cU$ is a unitary operator in $\cK$.
\item[3.] $\tau$ is an isometric operator from $\cE$ into $\cK\ ,$ the 
so-called embedding
operator, i.e. $\tau:\cE\to\cK$ and $\tau^*\tau=I_\cE$.
\item[4.] For $n\in\Z$
\aagf s_n=\tau^*{\cU}^n\tau\ .\label{1.20}\zzgf
\zzi
\end{defn}

A Naimark dilation is called \textit{minimal} if
\aagf \cK=\bigvee\limits_{n\in\Z}{\cU}^n\tau(\cE)\ .\label{1.21}\zzgf

\begin{defn}
Two Naimark dilations $(\cK_j,{\cU}_j,\tau_j)\ ,j=1,2\ ,$ of a measure
$\mu\in\cM(\T,\cE)$ are called unitarily equivalent if there exists a unitary
operator $Z:\cK_1\to\cK_2$ which satisfies
\aagf {\cU}_2Z=Z{\cU}_1\ ,\ Z\tau_1=\tau_2\ .\label{1.22}\zzgf
\end{defn}

Analogously with Theorem \ref{th1.14} one can prove

\begin{lem}
Any two minimal Naimark dilations $(\cK_j,{\cU}_j,\tau_j)\ ,j=1,2\ ,$
of a measure $\mu\in\cM(\T,\cE)$ are unitarily equivalent 
(i.e. a minimal Naimark dilation is essentially unique).
\end{lem}

According to the construction of a Naimark dilation of the measure $\mu$ we
consider two functions. The first of them has the form
\aagf\Phi(\zeta)=\int\limits_{\T}\frac{t+\zeta}{t-\zeta}\mu(dt)\ ,\ 
\zeta\in\D\ 
.\label{1.23}\zzgf
Obviously, $\Phi(\zeta)$ is holomorphic in $\D$. Moreover,
$$\Re[\Phi(\zeta)]=\frac{1}{2}[\Phi(\zeta)+\Phi^*(\zeta)]\ge 0\ ,\ \zeta\in\D\ 
,$$
and, as it follows from \eqref{1.19}, $\Phi(\zeta)$ has the Taylor series
representation
\aagf\Phi(\zeta)=I+2s_1\zeta+2s_2\zeta^2+\dots\ ,\zeta\in\D\ .\label{1.24}\zzgf
Thus, $\Phi(\zeta)$ belongs to the Caratheodory class $\cC[\cE]$ of all
$[\cE]-$valued functions which are holomorphic in $\D$ and have nonnegative
real part in $\D$.

The second of these functions $\te(\zeta)$ is related to $\Phi(\zeta)$ via the
Cayley transform:
\aagf\zeta\te(\zeta)=(\Phi(\zeta)-I)(\Phi(\zeta)+I)^{-1}\ .\label{1.25}\zzgf
From the properties of $\te$ and the wellknown lemma of H.A. Schwarz it follows
that $\te(\zeta)\in\cS[\cE]$.
where $\cS[\cE]:=\cS[\cE,\cE]$.
The functions $\Phi(\zeta)$ and $\te(\zeta)$ are called the functions of 
classes
$\cC[\cE]$ and $\cS[\cE],$ respectively, which are \textit{associated} with 
the
measure $\mu$.

If
\aagf\te(\zeta)=c_0+c_1\zeta+c_2\zeta^2+\dots\label{1.26}\zzgf
then from \eqref{1.24} and \eqref{1.25} we obtain
$$(c_0+c_1\zeta+c_2\zeta^2+\dots)(I+s_1\zeta+s_2\zeta^2+\dots)=s_1+s_2\zeta+s_3\zeta^2+\dots\ 
.$$
Thus
\aagf s_1=c_0\ ,\ 
s_n=c_0s_{n-1}+c_1s_{n-2}+\dots+c_{n-2}s_1+c_{n-1}\ ,\ n\in\N\setminus\{1\}\ 
.\label{1.27}\zzgf

\begin{thm}\label{th1.8}
Let $\mu\in\cM(\T,\cE)$ and denote $\te(\zeta)$ the function from the class
$\cS[\cE]$ associated with $\mu$. 
\aai\item[(a)] Let $\Dl=(\cH,\cE,\cE;T,F,G,S)$ be a 
simple unitary colligation  which satisfies $\te_\Dl(\zeta)=\te(\zeta)$. 
Then the triple
\aagf (\cK,\cU,\tau)\label{1.28}\zzgf
where $\cK=\cH\oplus\cE\ ,\ \cU =\matr T F G S :\cH\oplus\cE\to\cH\oplus\cE$
and $\tau$ is the operator of embedding $\cE$ into $\cH\oplus\cE\ ,$ i.e. for
each $e\in\cE$
$$\tau e=(0,e)\in\cH\oplus\cE$$
is a minimal Naimark dilation of the measure $\mu$.
\item[(b)] Let $(\cK,\cU,\tau)$ be a minimal Naimark dilation of the measure 
$\mu$ and
$\tau(\cE)=\wt\cE$. Let $\cH=\cK\ominus\wt\cE$ and suppose that according to
the decomposition $\cK=\cH\oplus\wt\cE$ the unitary operator $W$ has the 
matrix 
representation $$\cU=\matr T F G S:\cK\oplus\wt\cE\to\cK\oplus\wt\cE\ .$$
Then the tuple
$(\cH,\cE,\cE;T,F\tau,\tau^*G,\tau^*S\tau)$
is a minimal unitary colligation which satisfies $\te_\Dl(\zeta)=\te(\zeta)\ .$
\zzi
\end{thm}

\begin{rem}
The spaces $\cF$ and $\cG$ in the unitary colligation \eqref{E2.5} are different,
whereas they are identified in the consideration of Naimark dilation. For this
reason one has to distinguish between the unitary operator $U$ from \eqref{1.2}
and the unitary operator $\cU$ from \eqref{1.28}. The first of them acts between
different spaces, whereas the second one acts in the space $\cK\ .$
\end{rem}

\begin{proof}
Suppose that $\te$ has the Taylor series representation \eqref{1.26}. In view
of 
$$\te(\zeta)=S+\zeta G(I-\zeta 
T)^{-1}F=S+\sum\limits_{n=1}^\infty\zeta^nGT^{n-1}F\ ,\ \zeta\in\D\ ,$$
we obtain
\aagf c_0=S\ ,\ c_n=GT^{n-1}F\ ,n\in\N\ .\label{1.29}\zzgf

Observe that concerning the proof of the identities \eqref{1.20} it is enough
to prove them for $n\in\N$. For this, it is sufficient to prove that 
\aagf {\cU}^n=\matr * {x_n} * {y_n} \ ,\ n\in\N\label{1.30}\zzgf
where 
\aagf x_n=T^{n-1}Fs_0+T^{n-2}Fs_1+\dots+TFs_{n-2}+Fs_{n-1}\label{1.31}\zzgf
and
\aagf y_n=s_n\ .\label{1.32}\zzgf
Indeed, if the identity \eqref{1.30} is verified, for $n\in\N$ we get
$$\tau^*{\cU}^n\tau=(0,I)\matr * {x_n} * {y_n} \cl 0 I =y_n=s_n\ .$$
From \eqref{1.27} and \eqref{1.29} we infer that \eqref{1.30} is satisfied for
$n=1$. Applying the method of mathematical induction we assume that 
\eqref{1.30}
is satisfied for $n$. Then
\aagf\matr * {x_{n+1}} * {y_{n+1}} &=&{\cU}^{n+1}=\cU{\cU}^n\nnu\\
&=&\matr T F G S \cdot\matr * {x_n} * {y_n}
=\matr * {Tx_n+Fy_n} * {Gx_n+Sy_n}\ .\nnu\zzgf
This yields, together with \eqref{1.27}, \eqref{1.29}, \eqref{1.31} and 
\eqref{1.32}
the validity of \eqref{1.30} for $n+1\ ,$ too. Hence the triple \eqref{1.28}
is a Naimark dilation of the measure $\mu$. According to the proof of 
minimality we note that from \eqref{1.30} it follows 
$$\bigvee\limits_{n=0}^m 
{\cU}^n\cE=(\bigvee\limits_{n=0}^{m-1}T^nF(\cE))\oplus\cE\ ,\ m\in\N\ .$$
Thus,
$\bigvee\limits_{n=0}^\infty {\cU}^n\cE=(\bigvee\limits_{n=0}^\infty
T^nF(\cE))\oplus\cE\ .$
Analogously, we get
$$\bigvee\limits_{n=-\infty}^0
{\cU}^n\cE=(\bigvee\limits_{n=0}^\infty T^{*n}G^*(\cE))\oplus\cE\ .$$
Now the minimality condition \eqref{1.21} follows from the simplicity of the
colligation $\Dl$.

The assertion of (b) follows from the fact that the above considerations
can be done in the reverse way.
\end{proof}

\begin{rem}\label{rem1.22}
Thus, the model of unitary colligations is also a model for the Naimark
dilation
of Borel measures on the unit circle.
\end{rem}

Obviously, $\cS = \cS[\C]$.  In what follows, we consider only
this case.

\subsection{Schur algorithm, Schur parameters}

Let $\te\in\cS$. Following I. Schur \cite{Schur} we set
$\te_0:=\te$
and $\ga_0:=\te_0(0)$. Obviously, $|\ga_0|\le 1$. If $|\ga_0|<1\ ,$ we
consider the function
$$\te_1(\zeta):=\frac 1 \zeta
\frac{\te_0(\zeta)-\ga_0}{1-\ol\ga_0\te_0(\zeta)},
\qquad\zeta\in\D.
$$
In view of the Lemma of H.A. Schwarz $\te_1\in\cS$. As above we set
$\ga_1:=\te_1(0)$ and if $|\ga_1|<1$ we consider the function
$\te_2(\zeta):=\frac 1 \zeta
\frac{\te_1(\zeta)-\ga_1}{1-\ol\ga_1\te_1(\zeta)}$, $\zeta\in\D$.
Further, we continue this procedure inductively. Namely, if in the $j-$th step
a function $\te_j$ occurs for which $|\ga_j|<1$ where $\ga_j:=\te_j(0)$
we set
$$\te_{j+1}(\zeta):=\frac 1 \zeta 
\frac{\te_j(\zeta)-\ga_j}{1-\ol\ga_j\te_j(\zeta)},
\qquad\zeta\in\D,
$$
and continue this procedure.

Then two cases are possible:
\aai\item[(1)] The procedure can be carried out without end, i.e. $|\ga_j|<1\ 
, 
j=0,1,2,\dots$
\item[(2)] There exists an $n\in\{0,1,2,\dots\}$ such that $|\ga_n|=1$ and, if 
$n>0\ 
,$
then $|\ga_j|<1\ \forall j\in\{0,\dots,n-1\}\ .$\zzi
Thus, a sequence $(\ga_j)_{j=0}^\ome$ is associated with each function 
$\te\in\cS\ .$
Hereby we have $\ome=\infty$ in the first case and $\ome=n$ in the second. From
I. Schur's paper \cite{Schur} it is known that the second case appears if and only
if $\te$ is a finite Blaschke product of degree $n\ .$

\begin{defn}
The sequence $(\ga_j)_{j=0}^\ome$ obtained by the above procedure is called the
sequence of Schur parameters associated with the function $\te\ .$
\end{defn}

The following two properties established by I. Schur in \cite{Schur} determine the
particular role which the Schur parameters play in the study of functions of
class $\cS\ .$
\aai\item[(a)] There is a one-to-one correspondence between the set of 
functions 
$\te\in\cS$
and the set of corresponding sequences $(\ga_j)_{j=0}^\ome\ .$
\item[(b)] For each sequence $(\ga_j)_{j=0}^\ome$ which satisfies 
$$\left\{\begin{array}{lll}|\ga_j|<1\ ,\ j\in\{0,1,2,\dots\}\ , {\mbox\ \ if\ 
\ } \ome=\infty\\
|\ga_j|<1\ ,\ j\in\{0,\dots,\ome-1\}\ ,\ |\ga_\ome|=1\ ,{\mbox\ \ if\ \ } 0<
\ome<\infty\\
|\ga_0|=1\ ,{\mbox\ \ if\ \ } \ome=0\end{array}\right.$$
there exists a function $\te\in\cS$ such that the sequence 
$(\ga_j)_{j=0}^{\ome}$
is the Schur parameter sequence of $\te\ .$\zzi

Thus, the Schur parameters are independent parameters which determine the 
functions of class $\cS\ .$

\section{Construction of a model of a unitary colligation via the Schur parameters of its characteristic function in the scalar case}\label{S2}

In this Section, we follow \cite[Section{2}]{D06}.

\subsection{General form of the model}\label{sec2.1}

Let $\te\in\cS$. Assume that 
\aagf\Dl=(\cH,\C,\C;T,F,G,S)\label{2.1}\zzgf
is a simple unitary colligation satisfying $\te=\te_\Dl\ .$
In the considered case we have $\cF=\cG=\C\ .$
Note that in \eqref{2.1} the complex plane $\C$ is considered as the one-dimensional
complex Hilbert space with the usual inner product
\[ \bigl(z,w\bigr)_{\C} = z^{\ast}w, \quad z,w\in\C.   \] 

We take $1$ as basis vector of the one-dimensional vector space $\C\ .$
Set
\aagf \phi_1':=F(1)\ ,\ \wt\phi_1':=G^*(1)\ .\label{2.2}\zzgf
Then
$\cH_\cF=\bigvee\limits_{n=0}^\infty T^n\phi_1'
\ ,\ \cH_\cG=\bigvee\limits_{n=0}^\infty T^{*n}\wt\phi_1'\ .$
If $(f_\al)_{\al\in{\mathcal{A}}}$ is some family of vectors from 
${\mathcal{H}}$ the symbol
$\bigvee\limits_{\al\in{\mathcal{A}}}f_\al$ denotes the smallest (closed) 
subspace of
${\mathcal{H}}$ which contains all vectors of this family. The 
orthogonalization of the 
sequence $(T^n\phi_1')_{n=0}^\infty$ occupies an important place in the 
construction of the model.
First we assume that $(T^n\phi_1')_{n=0}^\infty$ is a sequence of linearly 
independent vectors. Then the Gram-Schmidt orthogonalization procedure
uniquely determines an orthonormal basis $(\phi_k)_{k=1}^\infty$ of the 
subspace $\cH_\cF$ such that, for $n\in\N$ the conditions
\aagf\bigvee\limits_{k=1}^n\phi_k=\bigvee\limits_{k=0}^{n-1}T^k\phi_1'\ ,
\ \ (T^{n-1}\phi_1',\phi_n)>0\label{2.3}\zzgf
are satisfied. Observe that for $n\in\{2,3,\dots\}$ the second condition is
equivalent to
$$(T\phi_{n-1},\phi_n)>0\ .$$
It is well known that the sequence $(\phi_k)_{k=1}^\infty$ is constructed in
the following way. We set
\aagf\phi_k':=T^{k-1}\phi_1'\ ,\ k\in\{1,2,3,\dots\}\label{2.4}\zzgf
and define inductively the sequence of vectors $(\wh\phi_k)_{k=1}^\infty$ via 
\aagf\wh\phi_1:=\phi_1'\ ,\ 
\wh\phi_k:=\phi_k'-\sum\limits_{s=1}^{k-1}\la_{ks}\phi_s'\ ,\ 
k\in\{2,3,\dots\}\label{2.5}\zzgf
where the coefficients $\la_{ks}$ are determined by the conditions
$\wh\phi_k\perp\phi_j'\ ,\ j\in\{1,\dots,k-1\}$. This means that the sequence
$(\la_{ks})_{s=1}^{k-1}$ yields a solution of the system of linear equations
$$\sum\limits_{s=1}^{k-1}\la_{ks}(\phi_s',\phi_j')=(\phi_k',\phi_j')\ ,\ 
j\in\{1,\dots,k-1\}\ .$$
Now we set
$\phi_k:=\frac 1 {\|\wh\phi_k\|}\wh\phi_k\ ,\ k\in\{1,2,\dots\}\ .$
Thus,
$$\phi_1=\frac 1 {\|\phi_1'\|}\phi_1'\ ,
\phi_k=\frac 1 {\|\wh\phi_k\|}T^{k-1}\phi_1'+u_{k-1}\ ,\ k\in\{2,3,\dots\}$$
where $u_k\in\bigvee\limits_{j=1}^k\phi_j$. These relations are equivalent
to
$$T^{k-1}\phi_1'=\|\wh\phi_k\|(\phi_k-u_{k-1})\ ,u_0=0\ ,\ k\in\{1,2,\dots\}\ 
.$$
From these identities we obtain for $k\in\N$
\aagf T\phi_k&=&\frac 1 {\|\wh\phi_k\|}T^k\phi_1'+Tu_{k-1}
=\frac 1 {\|\wh\phi_k\|}\|\wh\phi_{k+1}\|(\phi_{k+1}-u_k)+Tu_{k-1}\nnu\\
&=&\frac {\|\wh\phi_{k+1}\|} {\|\wh\phi_k\|}\phi_{k+1}+v_k\nnu\zzgf
where $v_k\in\bigvee\limits_{j=1}^k\phi_j\ .$
Thus,
\aagf\phi_1'=\|\phi_1'\|\phi_1\label{2.6}\zzgf
and $T\phi_k=\sum\limits_{j=1}^{k+1}t_{jk}\phi_j$ where
\aagf t_{k+1,k}=\frac {\|\wh\phi_{k+1}\|}{\|\wh\phi_k\|}\ ,\ 
k\in\{1,2,\dots\}\ .\label{2.7}\zzgf

Denote by $T_\cF$ the restriction of $T$ onto the invariant subspace $\cH_\cF$. 
From the above consideration it follows that the matrix of the operator 
$T_\cF$ 
with respect to the basis $(\phi_k)_{k=1}^\infty$ of $\cH_\cF$ has the shape
\aagf\left(\begin{array}{cccccc}
t_{11} & t_{12} & \hdots & t_{1n} & \hdots\\
t_{21} & t_{22} & \hdots & t_{2n} & \hdots\\
0 & t_{32} & \hdots & t_{3n} & \hdots\\
\vdots & \vdots & \ddots & \vdots & \\
0 & 0 & \hdots & t_{n+1,n} & \hdots\\
\vdots & \vdots &  & \vdots & \ddots
\end{array}\right)\ .\label{2.8}\zzgf

We assume that $\cH_\cF^\perp=\cH\ominus\cH_\cF\neq\{0\}$. Remember that the
largest shift $V_{T^*}=\Rstr_{\cH_\cF^\perp}T^*$ acts in $\cH_\cF^\perp$. 
Denote $\wt\cL_0$ the generating wandering subspace for the shift $V_{T^*}\ .$
Then 
\aagf\cH_\cF^\perp=\bigoplus\limits_{n=0}^\infty T^{*n}\wt\cL_0\label{2.9}\zzgf
where in view of \thref{th1.4} we have $\dim\wt\cL_0=1$. In view of
\coref{cor1.9} there exists a unique unit vector 
$\psi_1\in\wt\cL_0$
such that
\aagf (\wt\phi_1',\psi_1)>0\label{2.10}\zzgf
where $\wt\phi_1'$ is defined in \eqref{2.2}. In view of \eqref{2.9} and
\eqref{2.10}
the sequence $(\psi_k)_{k\in\N}$ where
$\psi_k:=T^{*k-1}\psi_1\ ,\ k\in\{1,2,\dots\}$
is the unique orthonormal basis in $\cH_\cF^\perp$ satisfying the conditions
\aagf(\wt\phi_1',\psi_1)>0\ ,\ \psi_{k+1}=T^*\psi_k\ ,\ k\in\{1,2,\dots\}\
.\label{2.11}\zzgf

\begin{defn}\label{D2.1-20230123}
The constructed orthonormal basis
\aagf \phi_1,\phi_2,\dots;\psi_1,\psi_2,\dots\label{2.12}\zzgf
of $\cH$ which satisfies the conditions \eqref{2.3} and \eqref{2.11} is called
\emph{canonical}.
\end{defn}

From the form of the construction it is clear that the canonical basis is
uniquely defined by the conditions \eqref{2.3} and \eqref{2.11}. This allows us
to identify in the following considerations operators and their matrix
representations with respect to this basis. We note that we suppose in the 
sequel that the vectors of the canonical basis are ordered as in \eqref{2.12}.
From the above considerations it follows that the matrix of the operator $T$ 
with respect to the canonical basis of $\cH$ has the block shape
\aagf T=\matr {T_\cF} {\wt R} 0 {\wt V_T} \ ,\ \wt V_T=(V_{T^*})^*\ 
,\label{2.13}\zzgf
where the matrix of $T_\cF$ is given by \eqref{2.8},
\aagf\wt R=
\left(\begin{array}{ccc}
r_1' & 0 & \hdots \\
r_2' & 0 & \hdots \\
\vdots & \vdots &  \\
r_n' & 0 & \hdots \\
\vdots & \vdots &  
\end{array}\right)\ ,\ \wt V_T=
\left(\begin{array}{cccc}
0 & 1 & &  \\
& 0 & 1 & \\
&   & 0 & \ddots\\
&   &   &  \ddots
\end{array}\right)\ .\label{2.14}\zzgf
Hereby missing matrix elements are assumed to be zero. Because of \eqref{2.2}
and \eqref{2.6} the matrix of the operator $F$ with respect to the canonical
basis has the shape
\aagf F=\col(\|\phi_1'\|,0,0,\dots;0,0,0,\dots)\ .\label{2.15}\zzgf
For the remaining elements of the unitary colligation $\Dl$ we obtain the 
following matrix representations
\aagf G=(g_1,g_2,g_3,\dots;g_\infty,0,0,\dots)\ ,\ S=\te(0)=\ga_0\ 
,\label{2.16}\zzgf
where in accordance with the above notations we get
\aagf g_k&=&G\phi_k=(G\phi_k,1)=(\phi_k,G^*(1))=(\phi_k,\wt\phi_1')\ ,\ 
k\in\{1,2,\dots\}\nnu\\
g_\infty&=&G\psi_1=(G\psi_1,1)=(\psi_1,G^*(1))=(\psi_1,\wt\phi_1')\ .\nnu\zzgf
The remaining entries in formula \eqref{2.16} are zero since from the 
colligation condition
$TG^*+FS^*=0$ we have for $k\in\{2,3,\dots\}$
\[\begin{split}
    G\psi_k
    &=(G\psi_k,1)=(T^*\psi_{k-1},G^*(1))
    =(\psi_{k-1},TG^*(1))
    =-(\psi_{k-1},FS^*(1))\\
    &=-\ga_0(\psi_{k-1},F(1))
    =-\ga_0\|\phi_1'\|(\psi_{k-1},\phi_1)
    =0\ .
\end{split}\]

From the colligation condition $F^*F+S^*S=I$ we get $\|F\|^2=1-|\ga_0|^2\ .$
Therefore, from \eqref{2.15} we infer
\aagf\|\phi_1'\|=\|F\|=\sqrt{1-|\ga_0|^2}\ .\label{2.17}\zzgf

\begin{lem}[{\cite[Lemma~2.4]{D06}}]\label{lm2.4}
For $k\in\{1,2,\dots\}$ the identities
$\|\wh\phi_k\|=\prod\limits_{j=0}^{k-1}\sqrt{1-|\ga_j|^2}$ hold true.
\end{lem}

From \eqref{2.7} and \lmref{lm2.4} we get the following result.

\begin{cor}\label{cor2.5}
The identities \aagf t_{k+1,k}=\sqrt{1-|\ga_k|^2}\ ,\
k\in\{1,2,\dots\}\label{2.25}\zzgf
hold true.
\end{cor}

\begin{cor}\label{cor2.6}
The sequence $(T^n\phi_1')_{n=0}^\infty$ consists of linearly independent
vectors
if and only if $|\ga_k|<1\ ,\ k\in\{0,1,\dots\}\ .$
\end{cor}

The proof follows from \lmref{lm2.4} and the observation that the sequence
$(T^n\phi_1')_{n=0}^\infty$ consists of linearly independent elements if and
only if $\|\wh\phi_k\|>0\ , \ k\in\{1,2,\dots\}\ .$

Let us set
\beql{E2.19}
D_{\gamma_k}\defeq\sqrt{1-\abs{\gamma_k}^2},
\qquad k=0,1,2,\dotsc
\eeq

\blemnl{\zitaa{D06}{\csec{2}}}{L2.3}
The following equations
\begin{align}
t_{kk}&=-\ko{\gamma_{k-1}}\gamma_k,&k&=1,2,\dotsc,\label{E2.20}\\
t_{nk}&=-\ko{\gamma_{n-1}}\gamma_k\prod_{j=n}^{k-1}D_{\gamma_j},&k&\geq n+1,\label{E2.21}\\
g_k&=\gamma_k\prod_{j=0}^{k-1}D_{\gamma_j},&k&=1,2,\dotsc\label{E2.22}
\end{align} 
hold.
\elem

\blemnl{\zitaa{D06}{Corollary~2.10}}{L2.4}
The vector system \eqref{2.4} is not closed in $\cH$ if and only if the product
\aagf\prod\limits_{j=0}^\infty(1-|\ga_j|^2)\label{2.53}\zzgf
converges. If this condition is satisfied then
\aagf g_\infty=\prod\limits_{j=0}^\infty\sqrt{1-|\ga_j|^2}\ ,\label{2.54}\zzgf
\aagf r_k'=-\ga_{k-1}\prod\limits_{j=k}^\infty\sqrt{1-|\ga_k|^2}\ \ k&=1,2,\dotsc,  .\label{2.254}\zzgf
\elem

The nonclosedness of the vector system \eqref{2.4} in $\cH$ means that 
\[
\cH_\cF^\perp=\cH\ominus\cH_\cF\neq\{0\}.
\]
In view of \thref{th1.3} this implies that $T^*$ contains a nonzero largest 
shift $V_{T^*}$. Because of $\dl_T=1\ ,$ from \thref{th1.4} we obtain that 
the multiplicity of
$V_{T^*}$ is equal to $1\ ,$ too. In view of \rmref{rm1.10} this
is equivalent to the fact that $T$ contains a nonzero largest shift $V_T$ of
multiplicity $1$. Thus, we have obtained the following result.

\blemnl{\zitaa{D06}{\clem{2.11}}}{L2.7}
Let $\te\in\cS$ and let $\Dl$ be a simple unitary colligation of type
\eqref{2.1} which satisfies $\te_\Dl=\te$.
Then the contraction $T$ (resp. $T^*\ )$ contains a nonzero largest shift if 
and
only if the infinite product \eqref{2.53} converges. If this condition is 
satisfied the multiplicities of the largest shifts $V_T$ and $V_{T^*}$ are 
both equal to $1\ .$
\elem

\begin{rem}\label{re2.8}
It is known (see e.g. Bertin et al. \cite[Ch.3]{Ber}) that 
\aagf\prod\limits_{j=0}^\infty(1-|\ga_j|^2)=\exp\{\frac 1 
{2\pi}\int\limits_{-\pi}^\pi\ln(1-|\te(e^{i\alpha})|^2) d\alpha\}\ ,\nnu\zzgf
where $\te(e^{i\alpha})$ denotes the nontangential boundary values of 
$\te$
which exist and are finite almost everywhere in view of a theorem due to Fatou.
Hence, the convergence of the product \eqref{2.54} means that
$\ln(1-|\te(e^{i\alpha})|^2)\in L^1[-\pi,\pi]\ .$
\end{rem}

\begin{defn}\label{de3.5}
Denote by $\Ga$ the set of all sequences $\ga=(\ga_j)_{j=0}^\ome$
which occur as Schur parameters of Schur functions. Furthermore,
denote $\Ga l_2$ the subset of all sequences belonging to $\Ga$
for which the product \eqref{2.53} converges. Thus,
$$\Ga l_2:=\{\ga=(\ga_j)_{j=0}^\infty:\ga_j\in\C,|\ga_j|<1,j\in\{0,1,2,\dotsc\}{\mbox\ \ and\ \ }\sum\limits_{j=0}^\infty|\ga_j|^2<\infty\}.$$
\end{defn}

The following statement shows the principal difference between the
properties of Schur parameters of inner functions and the
properties of Schur parameters of pseudocontinuable Schur
functions which are not inner.

\begin{thm}[\cite{D98}]\label{thm4.2}
Let $\te\in\cS\Pi$ and denote $(\ga_j)_{j=0}^\ome$ the sequence of
Schur parameters of $\te$. If $\te$ is not inner then
$\ome=\infty$ and the product \eqref{2.53} converges. If $\te$ is
inner then the product \eqref{2.53} diverges.
\end{thm}

\begin{proof}
If $\te\in\cS\Pi\setminus J$ then the function $\phi$ in the
representation \eqref{4.1} does not identically vanish. Hence,
$\ln(1-|\te(e^{i\al})|^2)=2\ln|\phi(e^{i\al})|\in L^1[-\pi,\pi]$
and in view of \rmref{re2.8} the product \eqref{2.53} converges. If
$\te\in J$ then from \rmref{re2.8} we infer that the product
\eqref{2.53} diverges.
\end{proof}

\begin{cor}\label{cor4.3}
Let $\te\in\cS\Pi\setminus J$. Then the sequence of Schur
parameters of $\te$ belongs to $\Ga l_2$.
\end{cor}

\subsection{Description of the model of a unitary colligation if
$\prod\limits_{j=0}^\infty(1-|\ga_j|^2)$ converges}

In this case we have in particular $|\ga_k|<1$ for all $k\in\{0,1,2,\dots\}\
.$
Therefore in view of \coref{cor2.6} the sequence \eqref{2.4}
does not contain linearly dependent elements whereas \lmref{L2.4}
tells us that this vector system is not closed in $\cH$. This means that
the canonical basis in $\cH$ has the shape \eqref{2.12}. The operators
$T,F,G$ and $S$ have with respect to this basis the matrix representations
\eqref{2.13}-\eqref{2.16}. Thus, from \eqref{2.17}- \eqref{E2.22} and 
\eqref{2.54}- \eqref{2.254} we obtain

\begin{thm}[\zitaa{D06}{Theorem~2.13}]\label{th2.1}
Let $\te\in\cS$ and let $\Dl$ be a simple unitary colligation of type \eqref{2.1} which satisfies $\te_\Dl=\te$. Assume that for the Schur parameter sequence of the function $\te$ the product \eqref{2.53} converges.
Then the canonical basis of the space $\cH$ has the shape \eqref{2.12}.
The operators $T, F, G$ and $S$ have with respect to this basis the following matrix representations:
\aagf T=\matr {T_\cF} {\wt R} 0 {\wt V_T}\ ,\label{2.62}\zzgf
where the operators in \eqref{2.62} are given by
\aagf T_\cF=\left(\begin{array}{ccccc}
-\ol\ga_0\ga_1	& -\ol\ga_0D_{\ga_1}\ga_2 	& \hdots 	& 
-\ol\ga_0\prod\limits_{j=1}^{n-1}D_{\ga_j}\ga_n & \hdots\\
D_{\ga_1} 	& -\ol\ga_1\ga_2  		& \hdots 	& 
-\ol\ga_1\prod\limits_{j=2}^{n-1}D_{\ga_j}\ga_n & \hdots\\
0 		& D_{\ga_2} 			& \hdots 	& -\ol\ga_2\prod\limits_{j=3}^{n-1}D_{\ga_j}\ga_n 
& \hdots\\
\vdots 		& \vdots 			& 		& \vdots 					  & 	 \\
0 		& 0 				& \hdots 	& -\ol\ga_{n-1}\ga_n 				  & \hdots\\
0 		& 0 				& \hdots 	& D_{\ga_n} 					  & \hdots\\
\vdots 		& \vdots 			&  		& \vdots 					  & 
\end{array}\right)\ ,\label{2.63}\zzgf
\aagf \wt R=\left(\begin{array}{cccc}
-\ol\ga_0\prod\limits_{j=1}^\infty D_{\ga_j} & 0 & 0 & \hdots\\
-\ol\ga_1\prod\limits_{j=2}^\infty D_{\ga_j} & 0 & 0 & \hdots\\
\vdots & \vdots & \vdots & \\
-\ol\ga_n\prod\limits_{j=n+1}^\infty D_{\ga_j} & 0 & 0 & \hdots\\
\vdots & \vdots & \vdots & 
\end{array}\right)\ ,
\ \ \wt V_T=\left(\begin{array}{cccc}
0 & 1 & 0 & \hdots\\
0 & 0 & 1 & \hdots\\
0 & 0 & 0 & \hdots\\
\vdots & \vdots & \vdots & 
\end{array}\right)\ ,\nnu\zzgf
\aagf F=\col(D_{\ga_0},0,0,\dots;0,0,0,\dots)\ ,\label{2.64}\zzgf
\aagf G=(\ga_1D_{\ga_0},\ga_2\prod\limits_{j=0}^1 D_{\ga_j},\dots, 
\ga_n\prod\limits_{j=0}^{n-1} D_{\ga_j},\dots; 
\prod\limits_{j=0}^\infty D_{\ga_j}, 0, 0, \dots)\label{2.65}\zzgf
and \aagf S=\ga_0\ .\label{2.66}\zzgf
\end{thm}

We consider the model space
\begin{equation}
    \wt\cH
    =l_2\oplus l_2
    =\left\{[(x_k)_{k=1}^\infty,(y_k)_{k=1}^\infty]\colon x_k,y_k\in\C;\;\sum_{k=1}^\infty
    |x_k|^2<\infty,\,\sum_{k=1}^\infty |y_k|^2<\infty\right\}\label{2.67}
\end{equation}
For $h_j=[(x_{jk})_{k=1}^\infty,(y_{jk})_{k=1}^\infty]\in\wt\cH\ ,j=1,2$ we
define
$$h_1+h_2:=[(x_{1k}+x_{2k})_{k=1}^\infty,(y_{1k}+y_{2k})_{k=1}^\infty]\ ,$$
$$\la h_1:=[(\la x_{1k})_{k=1}^\infty,(\la y_{1k})_{k=1}^\infty]\ ,\la\in\C\
,$$
$$(h_1,h_2):=\sum\limits_{k=1}^\infty x_{1k}\ol 
x_{2k}+\sum\limits_{k=1}^\infty y_{1k}\ol y_{2k}\ .$$
Equipped with these operations $\wt\cH$ becomes a Hilbert space. By the 
\emph{canonical} basis in $\wt\cH$ we mean the orthonormal basis 
\aagf e_1,e_2,\dots,e_n,\dots;e_1',e_2',\dots,e_n',\dots\label{2.68}\zzgf
where for $j\in\{1,2,\dots\}$
$$e_j=[(\dl_{jk})_{k=1}^\infty,(\dl_{0k})_{k=1}^\infty]\ ,$$
$$e_j'=[(\dl_{0k})_{k=1}^\infty,(\dl_{jk})_{k=1}^\infty]$$
and, as usual, for $j,k\in\{0,1,2,\dots\}$
$$\dl_{jk}=\left\{\begin{array}{ll}
1, & k=j\\
0, & k\neq j\ .
\end{array}\right.$$
We suppose that the elements of the canonical basis are ordered as in 
\eqref{2.68}.

\begin{cor}[\zitaa{D06}{Corollary~2.14} (Description of the model)]\label{cor2.14}
Let $\te\in\cS$ and let $\Dl$ be a simple unitary colligation of type
\eqref{2.1} which satisfies $\te_\Dl=\te$. Assume that the 
product \eqref{2.53} formed from the Schur parameter sequence of the function
$\te$ converges. Let us consider the model space \eqref{2.67} and let
$\wt T$ be the operator in $\wt\cH$ which has the matrix representation 
\eqref{2.62}
with respect to the canonical basis \eqref{2.68}. Moreover, let
$$\wt F:\C\to\wt\cH\ , \wt G:\wt\cH\to\C$$
be those operators which have the matrix representations \eqref{2.64} and 
\eqref{2.65}
with respect to the canonical basis in $\wt\cH\ ,$ respectively. Furthermore
let $\wt S:=\ga_0$. Then the tuple  $\wt\Dl=(\wt\cH,\wt\cF,\wt\cG;\wt T,\wt 
F,\wt G,\wt S)$
where $\wt\cF=\wt\cG=\C\ ,$ is a simple unitary colligation which is unitarily 
equivalent to $\Dl$ and thus $\te_{\wt\Dl}=\te\ .$
\end{cor}

\begin{proof}
From \thref{th2.1} it is obvious that the unitary operator $Z:\cH\to\wt\cH$
which maps the canonical basis \eqref{2.12} of $\cH$ to the canonical basis
\eqref{2.68} of $\wt\cH$ via
$$Z\phi_k=e_k\ ,Z\psi_k=\wt e_k\ ,k=1,2\ ,$$
satisfies the conditions
\aagf ZT=\wt TZ\ , ZF=\wt F\ , \wt GZ=G\ .\label{2.69}\zzgf
Thus, the tuple $\wt\Dl$ is a simple unitary colligation which is unitarily
equivalent to $\Dl\ .$
\end{proof}

\subsection{Description of the model of a unitary colligation in the case
of divergence of the series $\sum\limits_{j=0}^\infty |\ga_j|^2$}

In the case considered now the sequence of Schur parameters does not terminate.
Thus, $|\ga_j|<1$ for all $j\in\{0,1,2,\dots\}$. From \coref{cor2.6} 
we obtain that in this case the sequence \eqref{2.4} does not
contain linearly dependent vectors. On the other hand, the infinite product
\eqref{2.53} diverges in this case. Thus, in view of \lmref{L2.4}
we have $\cH_\cF=\cH$. This means that in this case the canonical basis of
the space $\cH$ consists of the sequence
\aagf\phi_1,\phi_1,\dots,\phi_n,\dots\ .\label{2.70}\zzgf
Hence, in the case considered now we have $T=T_\cF$. So we obtain the 
following statement.

\begin{thm}[\zitaa{D06}{\cthm{2.15}}]\label{th2.2}
Let $\te\in\cS$ and let $\Dl$ be a simple unitary colligation of type
\eqref{2.1} which satisfies $\te_\Dl=\te$. Assume that the
Schur parameter sequence $(\ga_j)_{j=0}^\infty$ of the function $\te$ 
satisfies
\aagf\sum\limits_{j=0}^\infty |\ga_j|^2=+\infty\ .\label{2.71}\zzgf
Then the canonical basis of $\cH$ has the shape \eqref{2.70}. The operator $T$ 
has
the matrix representation \eqref{2.63} with respect to this basis, whereas the 
matrix representation of the operators $F, G$ and $S$ with respect to this
basis are given by
\aagf F=\col(D_{\ga_0},0,0,\dots)\ ,\label{2.72}\zzgf
\aagf G=(\ga_1D_{\ga_0},\ga_2\prod\limits_{j=0}^1 
D_{\ga_j},\dots,\ga_n\prod\limits_{j=0}^{n-1} D_{\ga_j},\dots)\label{2.73}\zzgf
and $S=\ga_0\ ,$ respectively.
\end{thm}

In the case considered now as model space $\wt\cH$ we choose the space $l_2$
equipped with the above defined operations. By the canonical basis in $\wt\cH$
we mean the orthonormal basis
$e_j=(\dl_{jk})_{j,k=1}^\infty\ ,j\in\{1,2,3,\dots\}\ .$
Hereby, the elements of this basis are supposed to be naturally ordered via
\aagf e_1,e_2,\dots,e_n,\dots\ .\label{2.75}\zzgf

\begin{cor}[\zitaa{D06}{\ccor{2.16}} Description of the model]
Let $\te\in\cS$ and let $\Dl$ be a simple unitary colligation of type
\eqref{2.1} which satisfies $\te_\Dl=\te$. Assume that the
Schur parameter sequence $(\ga_j)_{j=0}^\infty$ of the function
$\te$ satisfies the divergence condition \eqref{2.71}.
Let us consider the model space $\wt\cH=l_2$ and let $\wt T$
be the operator in $\wt\cH$ which has the matrix representation \eqref{2.63}
with respect to the canonical basis \eqref{2.75}. Moreover, let
$\wt F:\C\to\wt\cH\ , \wt G:\wt\cH\to\C$
be those operators which have the matrix representations \eqref{2.72} and 
\eqref{2.73}
with respect to the canonical basis in $\wt\cH\ ,$ respectively. Furthermore
let $\wt S:=\ga_0$. Then the tuple  $\wt\Dl=(\wt\cH,\wt\cF,\wt\cG;\wt T,\wt 
F,\wt G,\wt S)$
where $\wt\cF=\wt\cG=\C\ ,$ is a simple unitary colligation which is unitarily 
equivalent to $\Dl$ and thus $\te_{\wt\Dl}=\te\ .$
\end{cor}

\begin{proof}
It suffices to mention that the unitary operator $Z:\cH\to\wt\cH$
which maps the canonical basis \eqref{2.70} of $\cH$ to the canonical basis
\eqref{2.75} of $\wt\cH$ satisfies the conditions \eqref{2.69}.
\end{proof}

\subsection{Description of the model in the case
that the function $\te$ is a finite Blaschke product}

Now we consider the case that the product \eqref{2.53} diverges whereas the 
series \eqref{2.71} converges. Obviously, this can only occur, if there exists
a number $n$ such that
$$|\ga_k|<1\ ,k=0,1,\dots,n-1\ ; |\ga_n|=1\ .$$
As already mentioned this means that the function $\te$ is a finite
Blaschke product of degree $n$. From \eqref{2.25} it follows that in this
case
$$t_{k+1,k}>0\ ,k=1,2,\dots,n-1; t_{n+1,n}=0\ .$$
From \eqref{2.7} we see then that this is equivalent to the fact that the
vectors $(\phi_k')_{k=1}^n$ are linearly dependent whereas the vector 
$\phi_{n+1}'$ 
is a linear combination of them. This means, in the case considered now we
have 
$$\cH=\cH_\cF=\bigvee\limits_{k=1}^\infty\phi_k'=\bigvee\limits_{k=1}^n\phi_k'\ 
.$$
Hence, $\dim\cH=n$ and the canonical basis in $\cH$ has the form
\aagf\phi_1,\phi_2,\dots,\phi_n\ .\label{2.76}\zzgf
As above we obtain the following result.

\begin{thm}[\zitaa{D06}{\cthm{2.17}}]\label{th2.3}
Let $\te$ be a finite Blaschke product of degree $n$ and let $\Dl$ be
a simple unitary colligation of type \eqref{2.1} which satisfies 
$\te_\Dl=\te\ .$
Then the canonical basis of the space $\cH$ has the form \eqref{2.76}. The
operators $T,F,G$ and $S$ have the following matrix representations with
respect to this basis:
\aagf T=\left(\begin{array}{cccc}
-\ol\ga_0\ga_1  & -\ol\ga_0D_{\ga_1}\ga_2 & \hdots & 
-\ol\ga_0\prod\limits_{j=1}^{n-1}D_{\ga_j}\ga_n \\
D_{\ga_1} & -\ol\ga_1\ga_2  & \hdots & 
-\ol\ga_1\prod\limits_{j=2}^{n-1}D_{\ga_j}\ga_n \\
0 & D_{\ga_2} & \hdots & -\ol\ga_2\prod\limits_{j=3}^{n-1}D_{\ga_j}\ga_n \\
\vdots & \vdots & & \vdots \\
0 & 0 & \hdots & -\ol\ga_{n-1}\ga_n  
\end{array}\right)\ ,\label{2.77}\zzgf
\aagf F=\col(D_{\ga_0},0,\dots,0)\ ,\label{2.78}\zzgf
\aagf G=(\ga_1D_{\ga_0},\ga_2\prod\limits_{j=0}^1 D_{\ga_j},\dots, 
\ga_n\prod\limits_{j=0}^{n-1} D_{\ga_j})\label{2.79}\zzgf
\aagf S=\ga_0\ .\label{2.80}\zzgf
\end{thm}

In the case considered now we choose the $n-$dimensional Hilbert space $\C^n$
as model space $\wt\cH$. By the canonical basis in $\wt\cH$ we mean the 
orthonormal basis
\aagf  e_j=(\dl_{jk})_{j,k=1}^n\ ,j\in\{1,\dots,n\}\ .\nnu\zzgf
Hereby we assume that the elements of this basis are naturally ordered via
\aagf e_1,e_2,\dots,e_n\ .\label{2.81}\zzgf
As above the following result can be verified:

\begin{cor}[\zitaa{D06}{\ccor{2.18}} Description of the model]
Let $\te$ be a finite Blaschke product of degree $n$ 
and let $\Dl$ be a simple unitary colligation of type
\eqref{2.1} which satisfies $\te_\Dl=\te$. 
Let $(\ga_j)_{j=0}^n$ be the 
Schur parameter sequence of the function
$\te$. 
Let us consider the model space $\wt\cH=\C^n$ and let $\wt T$
be the operator in $\wt\cH$ which has the matrix representation \eqref{2.77}
with respect to the canonical basis \eqref{2.81}. Moreover, let
$$\wt F:\C\to\wt\cH\ , \wt G:\wt\cH\to\C$$
be those operators which have the matrix representations \eqref{2.78} and 
\eqref{2.79}
with respect to the canonical basis in $\wt\cH\ ,$ respectively. 
Furthermore let $\wt S:=\ga_0$. 
Then the tuple  $\wt\Dl=(\wt\cH,\wt\cF,\wt\cG;\wt T,\wt F,\wt G,\wt S)$
where $\wt\cF=\wt\cG=\C\ ,$ is a simple unitary colligation which is unitarily 
equivalent to $\Dl$ and thus $\te_{\wt\Dl}=\te\ .$
\end{cor}

\subsection{Comments}
\paragraph*{A}
Let $\mu$ be a scalar, normalized Borel measure on $\T $. Let
us define the usual Hilbert space of square integrable complex
valued functions on $\T $ with respect to $\mu $ by
$$L^{2}(\mu)=L^{2}(\mu,\T)=\{f: f \ is\ \mu -measurable\ and
\int\limits_{\T}|f(t)|^2 \mu(dt)<\infty\}.
$$On the space  $L^{2}(\mu)$ we consider the unitary operator $U^\times$ which
is generated by multiplication by $\ol t$ where $t \in\T $ \ is
the independent variable :
$$(U^\times f)(t)=\ol tf(t),\ f\in L^{2}(\mu) .$$

Let $\tau$ be the embedding operator of $\C$ into $L^{2}(\mu)$,
i.e., $\tau:\C\to L^{2}(\mu)$ and for each $c\in\C$ the value
$\tau c$ is the constant function with the value $c $. It is
obvious that the triple $(L^{2}(\mu),U^\times,\tau)$ is the
minimal Naimark dilation of the measure $\mu .$

Consider the subspace $\cH_\mu:=L^{2}(\mu)\ominus\tau(\C) .$
According to the decomposition $L^{2}(\mu)=\cH_\mu\oplus\tau(\C)$
the operator $U^\times$ is given by the block matrix
$$U^\times=\matr {T^\times}{F^\times}{G^\times}{S^\times} .$$ Then
from \thref{th1.8}, statement (b), it follows that the set
$$\Dl_\mu:=(\cH_\mu, \C, \C; T^\times, F^\times\tau, \tau^{*} G^\times, \tau^{*} S^\times\tau)$$
is a simple unitary colligation.Moreover, the characteristic
function $\te_{\Dl_\mu}(\zeta)$ is associated with the measure
$\mu $. Thus, if the function $\Phi(\zeta)$ has the form
\eqref{1.23} then from \eqref{1.25} it follows that
$\zeta\te_{\Dl_\mu}(\zeta)=(\Phi(\zeta)-I)(\Phi(\zeta)+I)^{-1}\ .$

It  is important that the canonical basis \eqref{2.12} for the
colligation $\Dl_\mu$ \ is \ generated by the system of orthogonal
polynomials in  $L^{2}(\mu)$. Hence (see \thref{th1.8} and
\rmref{rem1.22}), \thref{th2.1}, \thref{th2.2} and \thref{th2.3}
give the matrix representation of the operator $U^\times$ in this
basis. The first appearance (1944) of this matrix is in Geronimus
\cite{G1}. Ya.L.Geronimus considered the case when the sequence of
the orthogonal polynomials is basis in $L^{2}(\mu),$ i.e., when
the series $\sum\limits_{j=0}^\infty |\ga_j|^2$ diverges (see
\thref{th2.2}). W.B.Gragg \cite{GR} in 1982 rediscovered this
matrix representation and used it for calculations. A.V.Teplyaev
\cite{TEP} (1991) seems to be first to use it for spectral
purposes. What concerns the role of this matrix representation in
theory of orthogonal polynomials on the unit circle we refer the
reader to B.Simon \cite[Chapter 4]{S}.

\paragraph*{B}
The full matrix representation (see \thref{th2.1}) appeared in
Constantinescu \cite{C} in 1984 (see also Bakonyi/Constantinescu
\cite{BC}). He finds it as the Naimark dilation. Let us establish
some connections with results in \cite{C}. We note that from
\rmref{rem1.22} it follows that the above constructed models of
unitary colligations are also models of Naimark dilations of
corresponding Borel measures on $\T$. Under this aspect we
consider in more detail the model described in \coref{cor2.14}. In
this case the model space $\cK$ for the Naimark dilation
\eqref{1.28} has the form \aagf\cK=\wt\cH\oplus\C=(l_2\oplus
l_2)\oplus\C\ .\label{2.82}\zzgf Moreover, $ \cU=\matr{\wt T}{\wt
F}{\wt G}{\wt S}:(l_2\oplus l_2)\oplus\C\to(l_2\oplus
l_2)\oplus\C$. In accordance to \eqref{2.82}, the vectors
$k\in\cK$ have the form \aagf
k=(x_1,x_2,\dots,x_n,\dots;y_1,y_2,\dots,y_n,\dots;c)\
,\label{2.83}\zzgf where $(x_k)_{k\in\N}\in l_2,(y_k)_{k\in\N}\in
l_2\ ,c\in\C$. The operator $\tau$ embeds $\C$ into $\cK$ in the
following way: $\tau c=(0,0,\dots,0,\dots;0,0,\dots,0,\dots;c)\
,c\in\C$. We change the order in the considered orthonormal base
of $\cK$ in such way that the vector $k$ which has the form
\eqref{2.83} is given in the following way \aagf
k=(\dots,y_n,\dots,y_2,y_1,c,x_1,x_2,\dots,x_n,\dots)\ ,\nnu\zzgf
In this case, it is convenient to set $ c=x_0\ ,y_k=x_{-k}\
,k\in\{1,2,\dots\}\ ,$ i.e., \aagf
k=(\dots,x_{-n},\dots,x_{-2},x_{-1},\mbox{\fbox{\rule[-0.1cm]{0.cm}{0.3cm}\,{$x_0$}\,}},x_1,x_2,\dots,x_n,\dots)\label{2.84}\zzgf
where we have drawn a square around the central entry with index
$0$. Now $$\tau
x_0=(\dots,0,\dots,0,\mbox{\fbox{\rule[-0.1cm]{0.cm}{0.3cm}\,{$x_0$}\,}},0,\dots,0,\dots)\
.$$ We associate with the representation \eqref{2.84} the following
orthogonal decomposition \aagf\cK=l_2^-\oplus
l_2^+\label{2.85}\zzgf where $l_2^-=\{k\in\cK:x_k=0\ ,k\ge 0\}
{\mbox\ \ and\ \ }l_2^+=\{k\in\cK:x_k=0\ ,k< 0\}$. From the form
of the operators $\wt T,\wt F,\wt G$ and $\wt S$ it follows that
the operator $\cU$ has the following matrix representation with
respect to the new basis and with respect to the orthogonal
decomposition \eqref{2.85}:
$\matr{\cU_{11}}{\cU_{12}}{\cU_{21}}{\cU_{22}}\ ,$ where
$\cU_{12}=0\ ,$ \aagf \cU_{11}=\left(\begin{array}{cccc}
\ddots &   &   &  \\
\ddots & 0 &  & \\
& 1 & 0 & \\
&  & 1 & 0\\
\end{array}\right)\ ,\ \ \cU_{21}=\left(\begin{array}{cccc}
\dots & 0 & 0 & \prod\limits_{j=0}^\infty D_{\ga_j}\\
\dots & 0 & 0 & -\ol\ga_1\prod\limits_{j=1}^\infty D_{\ga_j}\\
& \vdots & \vdots & \vdots \\
\dots & 0 & 0 & -\ol\ga_n\prod\limits_{j=n+1}^\infty D_{\ga_j}
\end{array}\right)\nnu\zzgf
and \aagf \cU_{22}=\left(\begin{array}{ccccc} \ga_0  &
D_{\ga_0}\ga_1 & \hdots & \prod\limits_{j=0}^{n-1}D_{\ga_j}\ga_n &\hdots\\
D_{\ga_0} & -\ol\ga_0\ga_1  & \hdots &
-\ol\ga_0\prod\limits_{j=1}^{n-1}D_{\ga_j}\ga_n & \hdots\\
0 & D_{\ga_1} & \hdots & -\ol\ga_1\prod\limits_{j=2}^{n-1}D_{\ga_j}\ga_n &
\dots\\
\vdots & \vdots & & \vdots & \\
0 & 0 & \hdots & -\ol\ga_{n-1}\ga_n & \hdots \\
0 & 0 & \hdots & D_{\ga_n} & \hdots \\
\vdots & \vdots &  & \vdots &
\end{array}\right)\ .\label{2.86}\zzgf

In this form but using different methods a Naimark dilation is
constructed in Constantinescu \cite{C} and Bakonyi/Constantinescu
\cite[Chapter 2]{BC}.

\paragraph*{C}
We consider a simple unitary colligation $\Dl$ of the form
\eqref{2.1}. If we choose in $\cH$ the canonical basis in
accordance with \eqref{2.15} we obtain for the contraction $(S,G)$
the matrix representation $ (S,G)=(S
,g_1,g_2,\dots,g_\infty,0,0,\dots)$. The above results show that
if we parametrize the contractive block row
$(S,g_1,g_2,\dots,g_n,\dots)$ by the method proposed in
Constantinescu \cite[Chapter 1]{C} we will obtain all blocks
described in \thref{th2.1} with exception of the coshift $\wt V_T$.

\paragraph*{D}
Because of $\cU_{12}=0$ the operator $\cU_{22}$ is an isometry
acting in $l_2^+$. We mention that the representation of an
isometry in the form \eqref{2.86} plays an important role in
Foias/Frazho \cite[Chapter 13]{FF} in connection with the
construction of Schur representations for the commutant lifting
theorem.

\paragraph*{E}
If $|\ga_k|<1\ ,k\in\{0,1,2,\dots\}$ and the product \eqref{2.53}
diverges then $\cK=l_2^+\ ,$ i.e. $\cU=\cU_{22}$. In this case the
layered form of the model is particularly clear. For example, if we
pass in the Schur algorithm from the Schur function
$\te_0(\zeta)=\te(\zeta)$ to the function
$\te_1(\zeta)=\frac{\te_0(\zeta)-\ga_0}{\zeta(1-\ol\ga_0\te_0(\zeta))}$
the Schur parameter sequence changes from $(\ga_k)_{k=0}^\infty$ to
$(\ga_k)_{k=1}^\infty$. This is expressed in the model
representation \eqref{2.86} in the following way. One has to cancel
the first column and the first row. After that one has to divide the
second row by $-\ol\ga_0$. This "layered form" finds its
expression in the following multiplicative representation of
$\cU_{22}$ which can be immediately checked (see also Foias/Frazho
\cite{FF}, Constantinescu \cite{C}): $\cU_{22}=V_0VV_2\dots
V_n\dots =s-\lim\limits_{n\to\infty}V_0V\dots V_n$ where
$V_0=R_{\ga_0}\oplus 1\oplus 1\oplus\dots ,$ $V=1\oplus
R_{\ga_1}\oplus 1\oplus\dots ,$ $V_2=1\oplus 1\oplus R_{\ga_2}\oplus
\dots$ and $R_{\ga_j}$ is the elementary rotation matrix associated
with $\ga_j\ ,$ i.e.,
$R_{\ga_j}=\matr{\ga_j}{D_{\ga_j}}{D_{\ga_j}}{-\ol\ga_j}\
,j\in\{0,1,2,\dots\}\ .$

\section{A model representation of the largest shift $V_T$ contained in a
    contraction $T$}\label{sec3-20221123}

In this section we follow {\cite[Section 3]{D06}}. 

\subsection{The conjugate canonical basis}\label{sec3.1}

Let $\te(\zeta)\in\cS$. Assume that \aagf\Dl=(\cH, \cF, \cG; T, F,
G, S)\label{3.1}\zzgf is a simple unitary colligation satisfying
$\te(\zeta)=\te_\Dl(\zeta)$. As above we consider the case
$\cF=\cG=\C$. Moreover, we choose the complex number $1$ as basis
vector of the onedimensional complex vector space $\C$. We assume
that the sequence $\ga=(\ga_j)_{j=0}^\ome$ of Schur parameters of
the function $\te(\zeta)$ is infinite (i.e., $\ome=\infty$) and
that the infinite product \eqref{2.53}
converges. In this case, as it follows from \thref{th2.1}, the
canonical basis of the space $\cH$ has the form (2.12). Hereby,
the matrix representation of the operators of the colligation
$\Dl$ with respect to this basis are given by formulas
(2.62)--(2.65).

We consider the function $\wt\te(\zeta)$ which is associate to \
$\te(\zeta)$,\ i.e., \ $\wt\te(\zeta)=\ol{\te(\ol\zeta)},\
\zeta\in\D$. Clearly, that $\wt\te(\zeta)\in\cS$ and
\aagf\wt\Dl:=(\cH, \cG, \cF; T^*, G^*, F^*, S^*)\label{3.3}\zzgf
is a simple unitary colligation satisfying
$\wt\te(\zeta)=\te_{\wt\Dl}(\zeta)$ (see Brodskii\cite{B}). The
unitary colligation \eqref{3.3} is called adjoint to the
colligation \eqref{3.1}. Hence, the function $\wt\te(\zeta)$ is the
c.o.f. of the contraction $T^*$. It can be easily seen that the
Schur parameter sequence $(\wt\ga_j)_{j=0}^\infty$ of the function
$\wt\te(\zeta)$ is given by $\wt\ga_j=\ol\ga_j,
j\in\{0,1,2,\dotsc\}$ and, consequently, the product (2.53)
converges for $(\wt\ga_j)_{j=0}^\infty$ , too. This means that the
canonical basis of the space $\cH$ which is constructed for the
colligation $\wt\Dl$ will also consist of two sequences of vectors
\aagf \wt\phi_1,\wt\phi_2,\dots;\wt\psi_1,\wt\psi_2,\dots\
.\label{E3.3-20221201}\zzgf From the considerations in Section 2.2 it
follows that this basis can be uniquely characterized by the
following conditions:

(1) The sequence $(\wt\phi_k)_{k=1}^\infty$ arises in the result
of the Gram-Schmidt orthogonalization procedure of the sequence
$(T^{*k-1}G^*(1))_{k=1}^\infty$ taking into account the
normalization conditions $(T^{*k-1}G^*(1),\wt\phi_k)>0,\
k\in\{1,2,3,\dotsc\}.$

(2) The vector $\wt\psi_1$ is that basis vector of the
one-dimensional generating wandering subspace of the maximal
unilateral shift $V_T$ acting in $\cH_\cG^\perp=\cH\ominus\cH_\cG$
which satisfies the inequality $(\phi_1,\wt\psi_1)>0$ and,
moreover, \aagf \wt\psi_{k+1}=T\wt\psi_k,\
k\in\{1,2,3,\dotsc\}.\label{3.7}\zzgf

\begin{defn}\label{de3.1}
The canonical basis \eqref{E3.3-20221201} which is constructed for the
adjoint colligation \eqref{3.3} is called conjugated to the
canonical basis (2.12) constructed for the colligation \eqref{3.1}.
\end{defn}

\begin{rem}\label{rm3.2}
In view of $\te(\zeta)=(\wt{\wt\te}(\zeta))$ and
$\Dl=(\wt{\wt\Dl})$ the canonical basis (2.12) is conjugated to
the canonical basis \eqref{E3.3-20221201}.
\end{rem}

Our approach is based on the study of interrelations between the
canonical basis (2.12) and the basis \eqref{E3.3-20221201} which is
conjugated to it. For this reason, we introduce the unitary
operator $\cU(\ga):\cH\to\cH$ which maps the first basis onto the
second one: \aagf \cU(\ga)\phi_k=\wt\phi_k,\ \
\cU(\ga)\psi_k=\wt\psi_k,\  \ k\in\{1,2,3,\dotsc\}
.\label{3.8}\zzgf The orthonormal systems $(\phi_k)_{k=1}^\infty$
and $(\psi_k)_{k=1}^\infty$ are bases of the subspaces $\cH_\cF$
and $\cH_\cF^\perp$, respectively, whereas the orthonormal systems
$(\wt\phi_k)_{k=1}^\infty$ and $(\wt\psi_k)_{k=1}^\infty$ are
bases of the subspaces $\cH_\cG$ and $\cH_\cG^\perp$,
respectively. Therefore, the operator $U(\ga)$ transfers the
decomposition $\cH=\cH_\cF\oplus\cH_\cF^\perp$ into the
decomposition $\cH=\cH_\cG\oplus\cH_\cG^\perp$ taking into account
the structures of the canonical bases. Consequently, the knowledge
of the operator $\cU(\ga)$ enables us to describe the position of
each of the subspaces $\cH_\cG$ and $\cH_\cG^\perp$ in relation to
$\cH_\cF$ and $\cH_\cF^\perp$. We emphasize that many properties
of the function $\te(\zeta)$ and the corresponding contraction $T$
depend on the mutual position of these subspaces.

In view of $\wt\ga_j=\ol\ga_j, j\in\{0,1,2,\dotsc\}$, the
replacement of the canonical basis (2.12) by its conjugated basis
\eqref{E3.3-20221201} requires that in corresponding matrix representations
we have to replace $\ga_j$ by $\ol\ga_j$. In particular, the
following result holds:

\begin{thm}\label{thm3.3}
The matrix representation of the operator $T^*$ with respect to
the canonical basis \eqref{E3.3-20221201} is obtained from the matrix
representation of the operator $T$ with respect to the canonical
basis (2.12) by replacing $\ga_j$ by $\ol\ga_j,\
j\in\{0,1,2,\dotsc\}$.
\end{thm}

\subsection{A model representation of the maximal unilateral shift $V_T$ contained in a contraction $T$}\label{sec3.2}

Let $\te(\zeta)\in\cS$ and assume that $\Dl$ is a simple unitary
colligation of the form \eqref{3.1} which satisfies
$\te(\zeta)=\te_\Dl(\zeta)$. We assume that the sequence of Schur
parameters of the function $\te(\zeta)$ is infinite and that the
infinite product \eqref{2.53} converges. Then it follows from \lmref{L2.7}
that in this and only this case the contraction $T$ (resp.
$T^*$) contains a nontrivial largest shift $V_T$ (resp.
$V_{T^*}$). Hereby, the multiplicity of the shift $V_T$ (resp.
$V_{T^*}$) equals $1$. The shift $V_{T^*}$ in the model
representation associated with the canonical basis (2.12) is
immediately determined by the sequence of basis vectors
$(\psi_k)_{k=1}^\infty$ since $\psi_1$ is a basis vector of the
one-dimensional generating wandering subspace of $V_{T^*}$ and
$\psi_k=V_{T^*}^{k-1}\psi_1,\ k\in\{2,3,4,\dotsc\}$. Analogously
(see property (2) of the conjugate canonical basis \eqref{E3.3-20221201}) the
sequence $(\wt\psi_k)_{k=1}^\infty$ of the basis \eqref{E3.3-20221201}
determines the largest shift $V_T$. Thus, representing the vectors
$(\wt\psi_k)_{k=1}^\infty$ in terms of the vectors of the basis
(2.12) we obtain a model representation of the largest shift $V_T$
with the aid of the canonical basis (2.12). The main goal of this
paragraph is the detailed description of this model. In the
following we use the same symbol for an operator and its matrix
with respect to the canonical basis (2.12).

The unitary operator \eqref{3.8} has the matrix representation
\aagf
\cU(\ga)=\matr{\cR(\ga)}{\cL(\ga)}{\cP(\ga)}{\cQ(\ga)}\label{3.9}\zzgf
where $\cR, \cP, \cL$ and $\cQ$ are the matrices of the operators
$$P_{\cH_\cF}\Rstr_{\cH_\cF}\cU:\cH_\cF\to\cH_\cF\ ,
\ P_{\cH_\cF^\perp}\Rstr_{\cH_\cF}\cU:\cH_\cF\to\cH_\cF^\perp,$$
$$P_{\cH_\cF}\Rstr_{\cH_\cF^\perp}\cU:\cH_\cF^\perp\to\cH_\cF {\mbox\ \  and\ \ }
P_{\cH_\cF^\perp}\Rstr_{\cH_\cF^\perp}\cU:\cH_\cF^\perp\to\cH_\cF^\perp,$$
respectively. Hereby, if $\cK$ is a closed subspace of $\cH$, the
operator $P_\cK$ denotes the orthoprojection from $\cH$ onto
$\cK$.

From \eqref{3.8} we see that the columns of the matrix \aagf
\cl{\cL(\ga)}{\cQ(\ga)}\label{3.10}\zzgf provide the coefficients
in the representation of the vectors $(\wt\psi_k)_{k=1}^\infty$
with respect to the canonical basis (2.12). Thus, the model
description of the shift $V_T$ leads to the determination of the
matrix \eqref{3.10}. We note that the matrix \eqref{3.10} shows how
the subspace $\cH_\cG^\perp$ is located relatively to the
subspaces $\cH_\cF$ and $\cH_\cF^\perp$.

\begin{thm}[\zitaa{D06}{\cthm{3.4}}]\label{thm3.4}
The identities
\aagf(\wt\psi_1,\phi_1)=\prod\limits_{j=1}^\infty(1-|\ga_j|^2)^{\frac{1}{2}}\label{3.11}\zzgf
and \aagf(\wt\psi_j,\phi_1)=0,\
j\in\{2,3,\dotsc\}\label{3.12}\zzgf hold true.
\end{thm}

\begin{proof}
In view of $\ \wt\phi_1=\frac{1}{\sqrt{1-|\ga_0|^2}}G^*(1)$ from
the matrix representation \eqref{2.65} of the operator $G$ it follows
\aagf(\psi_1,\wt\phi_1)=\frac{1}{\sqrt{1-|\ga_0|^2}}(\psi_1,G^*(1))=
\prod\limits_{j=1}^\infty(1-|\ga_j|^2)^{\frac{1}{2}}.\label{3.13}\zzgf
Since changing from $\ (\psi_1,\wt\phi_1)$ to $\
(\wt\psi_1,\phi_1)$ is realized by replacing $\ga_j$ by
$\wt\ga_j=\ol\ga_j,$ $j\in\{0,1,2,\dotsc,\}$, formula \eqref{3.11}
follows from \eqref{3.13}.

For $j\in\{2,3,\dotsc,\}$ we obtain
$$(\wt\psi_j,\phi_1)=(T\wt\psi_{j-1},\frac{1}{\sqrt{1-|\ga_0|^2}}F(1))
=\frac{1}{\sqrt{1-|\ga_0|^2}}(\wt\psi_{j-1},T^*F(1)).$$ From the
colligation condition (1.3) we infer
$$T^*F(1)=-G^*S(1)=-\ga_0G^*(1)=-\ga_0\sqrt{1-|\ga_0|^2}\ \wt\phi_1.$$
Thus, $(\wt\psi_j,\phi_1)=-\ol\ga_0(\wt\psi_{j-1},\wt\phi_1)=0.$
\end{proof}

We define the coshift $W$ via %
\aagf(\ga_0,\ga_1,\ga_2,\dotsc)\mapsto(\ga_1,\ga_2,\ga_3,\dotsc).\ \label{3.14}\zzgf

In the sequel, the system of functions $(L_n(\ga))_{n=0}^\infty$,
which was introduced for $\ga\in\Ga l_2$ in \cite{D98} will play an
important role. For $\ga\in\Ga l_2$ we set
\begin{flalign}
    &L_0(\ga):=1,\ \  L_n(\ga)=L_n(\ga_0,\ga_1,\ga_2,\dotsc):=\label{3.15}\\
    &\sum_{r=1}^n(-1)^r\!\!\!\!\!\!\!\!\!\!\sum_{s_1+\dotsb+s_r=n}\sum_{j_1=n-s_1}^\infty\sum_{j_2=j_1-s_2}^\infty    \!\!\!\dotso\!\!\!\!\!\sum_{j_r=j_{r-1}-s_r}^\infty\!\!\!\!\!\ga_{j_1}\ol\ga_{j_1+s_1}\ga_{j_2}\ol\ga_{j_2+s_2}\dotsm\ga_{j_r}\ol\ga_{j_r+s_r}.\nnu
\end{flalign}
Here the summation runs over all ordered $r$--tuples
$(s_1,\dotsc,s_r)$ of positive integers which satisfy
$s_1+s_2+\dotsb+s_r=n$.
For example,
\[
    L_1(\ga)
    =-\sum_{j=0}^\infty\ga_j\ol\ga_{j+1}, \ \
    L_2(\ga)
    =-\sum_{j=0}^\infty\ga_j\ol\ga_{j+2}    +\sum_{j_1=1}^\infty\sum_{j_2=j_1-1}^\infty    \ga_{j_1}\ol\ga_{j_1+1}\ga_{j_2}\ol\ga_{j_2+1}.
\]
In
view of $\ga\in\Ga l_2$ the series in \eqref{3.15} converges
absolutely.

\begin{thm}[\zitaa{D06}{\cthm{3.6}}]\label{thm3.5}
(Model representation of the largest shift $V_T$
with respect to the canonical basis (2.12)) Let $\te(\zeta)$ be a
function of class $\cS$ the Schur parameter sequence
$\ga=(\ga_j)_{j=0}^\infty$ of which belongs to $\Ga l_2$. Further,
let $\Dl$ be a simple unitary colligation of the form \eqref{3.1}
which satisfies $\te(\zeta)=\te_\Dl(\zeta)$. Then the vectors
$(\wt\psi_j)_{j=1}^\infty$ of the conjugate canonical basis
\eqref{E3.3-20221201} admit the following representations in terms of the
vectors of the canonical basis (2.12):
\aagf\wt\psi_j=\sum\limits_{k=j}^\infty\Pi_k L_{k-j}(W^j\ga)\phi_k
+\sum\limits_{k=1}^\infty Q(W^{k+j-1}\ga)\psi_k\label{3.16}\zzgf
where \aagf\Pi_k=\prod\limits_{j=k}^\infty\sqrt{1-|\ga_j|^2},\
k\in\{0,1,2,\dotsc\} \label{3.17}\zzgf and \aagf
Q(\ga)=-\sum\limits_{j=0}^\infty\ga_j L_{j}(\ga).
\label{3.18}\zzgf Hereby, the sequence $(L_n(\ga))_{n=0}^\infty$
is defined by \eqref{3.15} whereas the coshift $W$ is given via
\eqref{3.14}. The series in \eqref{3.18} converges absolutely.
\end{thm}

\begin{cor}[\zitaa{D06}{\ccor{3.7}}]\label{cor3.6}
The matrices $\cL(\ga)$ and $\cQ(\ga)$ introduced via \eqref{3.9}
can be expressed as \aagf \cL(\ga)=\left(\begin{array}{cccc}
\Pi_1 & 0 & 0 & \hdots \\
\Pi_2 L_1(W\ga) & \Pi_2 & 0 & \hdots\\
\Pi_3 L_2(W\ga) & \Pi_3 L_1(W^2\ga) & \Pi_3 & \hdots\\
\vdots & \vdots & \vdots & \ddots \\
\Pi_n L_{n-1}(W\ga) & \Pi_n L_{n-2}(W^2\ga) & \Pi_n L_{n-3}(W^3\ga) & \hdots\\
\vdots & \vdots &  \vdots &
\end{array}\right)\label{3.42}\zzgf
and \aagf\cQ(\ga)=\left(\begin{array}{cccc}
Q(W\ga) & Q(W^2\ga) & Q(W^3\ga) & \hdots\\
Q(W^2\ga) & Q(W^3\ga) & Q(W^4\ga) & \hdots\\
Q(W^3\ga) & Q(W^4\ga) & Q(W^5\ga) & \hdots\\
\vdots & \vdots & \vdots &  \\
\end{array}\right)\label{3.43}\zzgf
where $(L_n(\ga))_{n=0}^\infty$, $(\Pi_n)_{n=0}^\infty$ and $Q(\ga)$ are defined via formulas \eqref{3.15}, \eqref{3.17} and
\eqref{3.18}, respectively, whereas $W$ is the coshift introduced in \eqref{3.14}.
\end{cor}

\begin{proof}
The representation formulae \eqref{3.42} and \eqref{3.43} follow from \eqref{3.16}.
\end{proof}

\begin{cor}[\cite{D98}] \label{cor3.7} Each sequence $(\ga_j)_{j=0}^\infty\in\Ga l_2$
satisfies the following orthogonality relations:
\begin{multline*}
    \sum_{n=0}^\infty\Pi_{n+k}^2L_{n+k}(\ga)\ol{L_n(W^k\ga)}\\
    +\sum_{n=0}^\infty Q(W^n\ga)\ol{Q(W^{n+k}\ga)}
    =\bcase{1}{k=0,}{0}{k\in\{1,2,3,\dotsc\}.}
\end{multline*}
\end{cor}

\begin{proof}
It suffices to consider for the sequence
$\ga=(0,\ga_0,\ga_1,\dotsc)$ the representations \eqref{3.16} and
to substitute them into the orthogonality relations
$$(\wt\psi_1,\wt\psi_{k+1})=\bcase{1}{k=0,}{0}{k\in\{1,2,3,\dotsc\}.}$$
\end{proof}

\begin{thm}[\zitaa{D06}{\ccor{3.9}}]\label{cor3.8}
The recurrent formulas \aagf L_0(\ga)=L_0(W\ga)\label{3.44}\zzgf
and \aagf
L_n(\ga)=L_n(W\ga)-\ol\ga_n\sum\limits_{j=0}^{n-1}\ga_jL_j(\ga),\
n\in\N,\label{3.45}\zzgf hold true.
\end{thm}

The Hankel matrix  \eqref{3.43} is the matrix of the Hankel operator
which describes the mutual position of the subspaces
$\cH_\cG^\perp$ and $\cH_\cF^\perp$ in which the largest shifts
$V_T$ and $V_{T^*}$ are acting, respectively. As it was already
mentioned (see Introduction) the subspaces $\cH_\cG^\perp$ and
$\cH_\cF^\perp$ are interpreted as inner channels of scattering in
the scattering system associated with the contraction $T$. In this
connection we introduce the following notion.

\begin{defn}\label{de3.9}
The Hankel matrix \eqref{3.43} will be called the Hankel matrix of
the largest shifts $V_T$ and $V_{T^*}$ or the Hankel matrix of the
inner channels of scattering associated with $T$.
\end{defn}

We note that the unitarity of the operator matrix given via
\eqref{3.9} implies
$$I-\cQ^*(\ga)\cQ(\ga)=\cL^*(\ga)\cL(\ga).$$
This means the matrix $\cL(\ga)$ plays the role of a defect
operator for $\cQ(\ga)$. Taking into account \eqref{3.18} and
\eqref{3.15} from \eqref{3.43} we infer $\cQ^*(\ga)=\cQ(\ol\ga)$.

From the form \eqref{3.42} we get immediately the following
observation.

\begin{lem}[\zitaa{D06}{\clem{3.11}}]\label{lm3.10}
The block representation \aagf
\cL(\ga)=\matr{\Pi_1}{0}{B(\ga)}{\cL(W\ga)}\label{3.46}\zzgf with
$B(\ga)=\col(\Pi_2L_1(W\ga),\Pi_3L_2(W\ga),\dotsc,\Pi_nL_{n-1}(W\ga),\dotsc)$
holds true.
\end{lem}

The next result, which describes the multiplicative structure of $\cL(\gamma)$ and indicates connections of the operator $\cL(\gamma)$ with the backward shift $W$, plays a key role in our approach.

\begin{thm}[\zitaa{D06}{\cthm{3.12}}]\label{thm3.11}
It holds \aagf \cL(\ga)=\cM(\ga)\cL(W\ga)\label{3.47}\zzgf where
\aagf \cM(\ga)=\left(\begin{array}{cccc}
D_{\ga_1}       & 0                             & 0             & \hdots\\
-\ga_1\ol\ga_2  & D_{\ga_2}                     & 0             & \hdots\\
-\ga_1D_{\ga_2}\ol\ga_3 & -\ga_2\ol\ga_3           & D_{\ga_3}     & \hdots\\
\vdots          & \vdots                        & \vdots        & \ddots\\
-\ga_1\prod\limits_{j=2}^{n-1}D_{\ga_j}\ol\ga_n & -\ga_2\prod\limits_{j=3}^{n-1}D_{\ga_j}\ol\ga_n & -\ga_3\prod\limits_{j=4}^{n-1}D_{\ga_j}\ol\ga_n & \hdots\\
\vdots          & \vdots                        & \vdots        &
\end{array}\right)\label{3.48}\zzgf
and $D_{\ga_j}=\sqrt{1-|\ga_j|^2},\ j\in\{0,1,2,\dotsc\}$.
\end{thm}

\begin{proof}
From \eqref{3.7} we infer $T^*\wt\psi_{k+1}=\wt\psi_k,\
k\in\{1,2,3,\dotsc\}$. Thus, $T^*$ maps the sequence
$(\wt\psi_2,\wt\psi_3,\wt\psi_4,\dotsc)$ to the sequence
$(\wt\psi_1,\wt\psi_2,\wt\psi_3,\dotsc)$, i.e.,
\aagf(\wt\psi_1,\wt\psi_2,\wt\psi_3,\dotsc)=(T^*\wt\psi_2,T^*\wt\psi_3,T^*\wt\psi_4,\dotsc).\label{3.49}\zzgf
From \eqref{2.62} it follows that the matrix representation of the
operator $T^*$ with respect to the canonical basis \eqref{2.12} has the
shape \aagf T^*=\matr{T_\cF^*}{0}{\wt R^*}{\wt
V_T^*}.\label{3.50}\zzgf Hereby, as it can be seen from \eqref{2.63} and
\eqref{3.48}, we have \aagf T_\cF^*=(-\ga_0\eta(\ga)\ ,\ \cM(\ga)\
)\label{3.51}\zzgf where \aagf\eta(\ga):=\col(\ol\ga_1,\ol\ga_2
D_{\ga_1},\dotsc,\ol\ga_n\prod\limits_{j=1}^{n-1}D_{\ga_j},\dotsc).\label{3.52}\zzgf
Taking into account \eqref{3.10}, \eqref{3.43} and \eqref{3.46} we
get the representations
\aagf(\wt\psi_1,\wt\psi_2,\wt\psi_3,\dotsc)=\cl{\cL(\ga)}{\cQ(\ga)}\label{3.53}\zzgf
and
\aagf(\wt\psi_2,\wt\psi_3,\wt\psi_4,\dotsc)=\cl{\cl{0}{\cL(W\ga)}}{\cQ(W\ga)}\label{3.54}\zzgf
with respect to the canonical basis (2.12). Inserting the matrix
representations \eqref{3.50}, \eqref{3.53} and \eqref{3.54} in
formula \eqref{3.49} we find in particular
$$\cL(\ga)=T_\cF^*\cl{0}{\cL(W\ga)}.$$
Combining this with \eqref{3.51} we obtain \eqref{3.47}.
\end{proof}

\begin{cor}[\zitaa{D06}{\ccor{3.13}}]\label{cor3.12}
It holds \aagf
I-\cM(\ga)\cM^*(\ga)=\eta(\ga)\eta^*(\ga)\label{3.55}\zzgf where
$\eta(\ga)$ is given via \eqref{3.52}.
\end{cor}

\begin{proof}
From \eqref{2.65} and \eqref{3.52} we obtain \aagf G=D_{\ga_0}(\
\eta^*(\ga)\ ;\ \prod\limits_{j=1}^\infty D_{\ga_j}\
,0,0,\dotsc).\label{3.56}\zzgf Substituting now the matrix
representations \eqref{2.62} and \eqref{3.56} in the colligation
condition $I-T^*T=G^*G$ we infer in particular
$I_{\cH_\cF}-T_\cF^*T_\cF=(1-|\ga_0|^2)\eta(\ga)\eta^*(\ga).$
Substituting the block representation \eqref{3.51} in this
representation we get \eqref{3.55}.
\end{proof}

From \eqref{3.42} it follows
\[
    \lim_{n\to\infty}\cL(W^n\gamma)
    =I.
\]
Thus, from \eqref{3.47} we obtain

\begin{cor}\label{C3.13-1207}
 It holds
\[
    \cL(\gamma)
    =\prodr_{j=0}^\infty\cM(W^j\gamma).
\]
\end{cor}

\begin{lem}[\zitaa{D06}{\clem{3.14}}]\label{lm3.13}
The matrices $\cP(\ga)$ and $\cL(\ga)$ introduced via \eqref{3.9}
are linked by the formula $\cP(\ga)=\cL(\ol\ga)^*.$
\end{lem}

\begin{proof}
Let $\cP(\ga)=(p_{kj}(\ga))_{k,j=1}^\infty$ and
$\cL(\ga)=(l_{kj}(\ga))_{k,j=1}^\infty$. Since the change from the
canonical basis \eqref{2.12} to the conjugate canonical basis \eqref{E3.3-20221201}
is connected via the replacement of $\ga_j$ by $\ol\ga_j,\
j\in\{0,1,2,\dotsc\}$ and taking into account matrix
representation \eqref{3.42} we get
$p_{kj}(\ga)=(\wt\phi_j,\psi_k)=\ol{(\phi_j,\wt\psi_k)}=(\wt\psi_k,\phi_j)=l_{jk}(\ga)=\ol{l_{jk}(\ol\ga)},
j,k\in\N.$
\end{proof}

\section{Some criteria for the pseudocontinuability of a Schur function in terms
    of its Schur parameters}\label{sec4-20221123}

\subsection{On some connections between the largest shifts $V_T$ and $V_{T^*}$ and the
pseudocontinuability of the corresponding c.o.f. $\te$}

Let $\te\in\cS$. Assume that $\Dl$ is a simple unitary colligation
of type \eqref{3.1} which satisfies $\te_\Dl(\zeta)=\te(\zeta)$. We
suppose that the Schur parameter sequence of $\te$ belongs to $\Ga
l_2$. Then from \lmref{L2.7} it follows that in this and only in
this case the contraction $T$ (resp. $T^*$) contains a nontrivial
largest shift $V_T$ (resp. $V_{T^*}$). Hereby, the multiplicities
of the shifts $V_T$ and $V_{T^*}$ coincide and are equal to one.
We consider decompositions \eqref{1.7} and \eqref{1.201}. Let
\begin{align}
    \cN_{\cF\cG}&:=\cH_\cF\cap\cH_\cG^\perp ,&
    \cN_{\cG\cF}&:=\cH_\cG\cap\cH_\cF^\perp,\label{4.2}\\
    \cH_{\cF\cG}&:=\cH_\cF\ominus\cN_{\cF\cG},&
    \cH_{\cG\cF}&:=\cH_\cG\ominus\cN_{\cG\cF}.\label{4.3}
\end{align}
Then
\aagf\cH=\cH_\cG^\perp\oplus\cH_{\cG\cF}\oplus\cN_{\cG\cF},\label{4.4}\zzgf
\aagf\cH=\cN_{\cF\cG}\oplus\cH_{\cF\cG}\oplus\cH_\cF^\perp.\label{eq: 3.25}\zzgf

We note that with respect to the orthogonal decomposition \eqref{4.4} the contraction $T$ admits the triangulation
\begin{align}
    T=\begin{pmatrix}
    V_T&*&*\\0&T_{\mathfrak G\mathfrak F}&*\\0&0&\wt V_{T_{\mathfrak G}}
    \end{pmatrix},\label{eq: 1.38}
\end{align}
where $V_T$ is the maximal shift contained in $T$
and $\wt V_{T_{\mathfrak G}}$ is the maximal coshift contained in $T_{\mathfrak G}$
(see representation \eqref{1.11}). Hence, it follows from \eqref{1.201}, \eqref{1.11}, \eqref{4.4}, and \eqref{eq: 1.38} that with respect to the orthogonal decomposition
\begin{align*}
\cH_{\mathfrak G} = \mathfrak H_{\mathfrak G\mathfrak F}\oplus\mathfrak N_{\mathfrak G\mathfrak F}.
\end{align*}
the operator $T_\mathfrak{G}$ admits the block representation
\begin{align}\label{eq: 11.38}
T_{\mathfrak G} = \begin{pmatrix}{T_{\mathfrak G\mathfrak F}}&*\\0&{\Tilde{V}_{T_{\mathfrak G}}}\end{pmatrix}.
\end{align}
Analogously, the orthogonal decomposition \eqref{eq: 3.25} corresponds to the triangulation
\begin{align}
    T=\begin{pmatrix}
    {V_{T_{\mathfrak F}}}&*&*\\0&T_{\mathfrak F\mathfrak G}&*\\0&0&{\wt V_T}
    \end{pmatrix},\label{eq: 10.38}
\end{align}
where $\wt V_T$ is the maximal coshift contained in $T$
and $V_{T_{\mathfrak F}}$ is the maximal shift contained in $T_{\mathfrak F}$
(see representation \eqref{1.10}).
Hence, it follows from \eqref{1.7}, \eqref{1.10}, \eqref{eq: 3.25}, and \eqref{eq: 10.38} that with respect to the orthogonal decomposition
\begin{align*}
    \cH_{\mathfrak F} = \mathfrak N_{\mathfrak F\mathfrak G}\oplus\mathfrak H_{\mathfrak F\mathfrak G}.
\end{align*}
the operator $T_\mathfrak{F}$ admits the block representation
\begin{align*}
T_{\mathfrak F} = \begin{pmatrix}
{V_{T_{\mathfrak F}}}&*\\0&{T_{\mathfrak F\mathfrak G}}
\end{pmatrix}.
\end{align*}
From \eqref{4.2} and \eqref{4.3} it follows that
\begin{align}\label{eq: 1.27}
    \mathfrak H_{\mathfrak G\mathfrak F}=\overline{P_{\mathfrak H_\mathfrak G}\mathfrak H_\mathfrak F},~~~\mathfrak H_{\mathfrak F\mathfrak G}=\overline{P_{\mathfrak H_\mathfrak F}\mathfrak H_\mathfrak G}.
\end{align}
Thus,
\begin{align}\label{eq: 1.28}
    \operatorname{dim}\mathfrak H_{\mathfrak G\mathfrak F}=\operatorname{dim}\mathfrak H_{\mathfrak F\mathfrak G}.
\end{align}

The following criterion of pseudocontinuability of a noninner Schur function (see, e.g., \cite[Theorem 4.5]{D06}) plays an important role in our subsequent investigation.
\begin{thm}\label{t1.23}
    Let $\theta\in\mathcal S$ and let $\Delta$ be a simple unitary colligation of the form \eqref{3.1} which satisfies $\theta_\Delta=\theta$. Then the conditions $\mathfrak N_{\mathfrak G\mathfrak F}\not=\{0\}$ and $\mathfrak N_{\mathfrak F\mathfrak G}\not=\{0\}$ are equivalent. They are satisfied if and only if $\theta\in\mathcal S\Pi\setminus J.$
\end{thm}

\begin{cor}\label{c1.24}
If $\theta\in\mathcal S\mathcal P\setminus J$, then $\mathfrak N_{\mathfrak G\mathfrak F}\not=\{0\}$ and $\mathfrak N_{\mathfrak F\mathfrak G}\not=\{0\}$.
\end{cor}

\thref{t1.23} will be complemented by the following result
(see Arov\cite{A4}) which is obtained in \cite[Theorem 4.6]{D06} in another way.

\begin{thm}[\cite{A4}]\label{thm4.5}
Let $\te$ be a function of class $\cS$ such that its Schur
parameter sequence $(\ga_j)_{j=0}^\infty$ belongs to $\Ga l_2$.
Assume that $\Dl$ is a simple unitary colligation of the form
\eqref{3.1} which satisfies $\te_\Dl(\zeta)=\te(\zeta)$. Then $\te$
is a rational function if and only if $\dim\cH_{\cG\cF}<\infty\ \
({\mbox\ \ resp.\ \ }\dim\cH_{\cF\cG}<\infty)$. If
$\dim\cH_{\cG\cF}<\infty$ then $\dim\cH_{\cG\cF}$ (resp.
$\dim\cH_{\cF\cG}$) is the smallest number of elementary $2\times
2$--Blaschke--Potapov factors in a finite Blaschke--Potapov
product of the form \eqref{4.2B} with block $\te$.
\end{thm}

\begin{lem}\label{lm4.7}
It holds \aagf\cN_{\cG\cF}=\ker \cQ^*(\ga)\label{4.21}\zzgf where
$\cQ(\ga)$ is that Hankel operator in $\cH_\cF^\perp$ the matrix
representation of which with respect to the basis
$(\psi_j)_{j=1}^\infty$ has the form \eqref{3.43}.
\end{lem}

\begin{proof}
From \eqref{4.2} it follows that $h\in\cN_{\cG\cF}$ if and only if
$h\in\cH_\cF^\perp$ and $h\perp\cH_\cG^\perp$. Combining this with
the fact that the vectors \eqref{3.16} form an orthonormal basis in
$\cH_\cG^\perp$ we obtain \eqref{4.21}.
\end{proof}

\subsection{Construction of a countable closed vector system in $\cH_{\cG\cF}$ and
investigation of the properties of the sequence
$(\si_n)_{n=1}^\infty$ of Gram determinants of this system}

Let $\te(\zeta)\in\cS$ and assume that $\Dl$ is a simple unitary
colligation of the form  \eqref{3.1} which satisfies
$\te_\Dl(\zeta)=\te(\zeta)$. As in the preceding chapter it is
assumed that the Schur parameter sequence $(\ga_j)_{j=0}^\infty$
of $\te(\zeta)$ belongs to $\Ga l_2$.

\begin{thm}[\zitaa{D06}{\cthm{5.1}}]\label{thm5.1}
The linear span of vectors \aagf
h_n:=\phi_n-\Pi_n\sum\limits_{j=1}^n
\ol{L_{n-j}(W^j\ga)}\wt\psi_j,\ n\in\N\label{5.1}\zzgf is dense in
$\cH_{\cG\cF}$. Here $(\phi_k)_{k=1}^\infty$ and
$(\wt\psi_k)_{k=1}^\infty$ denote the orthonormal systems taken
from the canonical basis \eqref{2.12} and the conjugate canonical basis
\eqref{E3.3-20221201}, respectively, whereas $W, (L_k(\ga))_{k=1}^\infty$ and
$(\Pi_k)_{k=1}^\infty$ are given via \eqref{3.14}, \eqref{3.15} and
\eqref{3.17}.
\end{thm}

\begin{proof}
Since $(\phi_k)_{k=1}^\infty$ is an orthonormal basis in $\cH_\cF$, from the first equality in \eqref{eq: 1.27} it follows that the vectors $h_n=P_{\cH_\cG}\phi_n,\ n\in\N$, form a closed system in
$\cH_{\cG\cF}$. Since $(\wt\psi_k)_{k=1}^\infty$ is an orthonormal
basis in $\cH_\cG^\perp$, the identities
$h_n=P_{\cH_\cG}\phi_n=\phi_n-\sum\limits_{j=1}^\infty(\phi_n,\wt\psi_j)\wt\psi_j,\
n\in\N,$ hold true. It remains to note that from the
decompositions \eqref{3.16} we obtain
$$(\phi_n,\wt\psi_j)=\bcase{\Pi_n\ol{L_{n-j}(W^j\ga)}}{j\le
n,}{0}{j>n.}$$
\end{proof}

\begin{cor}[\zitaa{D06}{\ccor{5.2}}]\label{cor5.2}
It holds
\begin{equation}
    \begin{pmatrix}
        (h_1,h_1) & (h_2,h_1) & \hdots & (h_n,h_1) \\
        (h_1,h_2) & (h_2,h_2) & \hdots & (h_n,h_2)\\
        \vdots & \vdots &  & \vdots \\
        (h_1,h_n) & (h_2,h_n) & \hdots & (h_n,h_n)\\
    \end{pmatrix}
    =I-\cL_n(\ga)\cL_n^*(\ga),\ n\in\N\label{5.2}
\end{equation}
where
\aagf
\!\!\!\!\!\! \cL_n(\ga)=\left(\begin{array}{ccccc}
\Pi_1 & 0 & 0 & \hdots & 0\\
\Pi_2 L_1(W\ga) & \Pi_2 & 0 & \hdots & 0\\
\Pi_3 L_2(W\ga) & \Pi_3 L_1(W^2\ga) & \Pi_3 & \hdots & 0\\
\vdots & \vdots & \vdots & & \vdots \\
\Pi_n L_{n-1}(W\ga) & \Pi_n L_{n-2}(W^2\ga) & \Pi_n L_{n-3}(W^3\ga) & \hdots & \Pi_n\\
\end{array}\right)\label{5.3}\zzgf
is the $n$--th order principal submatrix of the matrix $\cL(\ga)$
given in \eqref{3.42}.
\end{cor}

\begin{proof}
The identities \eqref{5.2} are an immediate consequence of
\eqref{5.1}.
\end{proof}

In the sequel, the matrices \aagf
\cA_n(\ga):=I_n-\cL_n(\ga)\cL_n^*(\ga),\ n\in\N \label{5.4}\zzgf
and their determinants \aagf \si_n(\ga):=\bcase{1}{n=0,}{\det
\cA_n(\ga)}{n\in\N}\label{5.5}\zzgf will play an important role.
They have a lot of remarkable properties. In order to prove these
properties we need the following result which follows from
\thref{thm3.11} and \coref{cor3.12}.

\begin{lem}[\zitaa{D06}{\clem{5.3}}]\label{lm5.3}
It holds \aagf \cL_n(\ga)=\cM_n(\ga)\cL_n(W\ga)\label{5.6}\zzgf
where $\cL_n(\ga)$ is given via \eqref{5.3} whereas
\begin{equation}
    \cM_n(\ga)
    =
    \begin{psmallmatrix}
        D_{\ga_1}       & 0                             & 0             & \hdots & 0\\
        -\ga_1\ol\ga_2  & D_{\ga_2}                     & 0             & \hdots & 0\\
        -\ga_1D_{\ga_2}\ol\ga_3 & -\ga_2\ol\ga_3           & D_{\ga_3}     & \hdots & 0\\
        \vdots          & \vdots                        & \vdots        & \ddots & \vdots\\
        -\ga_1\prod_{j=2}^{n-1}D_{\ga_j}\ol\ga_n & -\ga_2\prod_{j=3}^{n-1}D_{\ga_j}\ol\ga_n & -\ga_3\prod_{j=4}^{n-1}D_{\ga_j}\ol\ga_n & \hdots & D_{\ga_n}
    \end{psmallmatrix}\label{5.7}
\end{equation}
is the $n$--th order principal submatrix of the matrix $\cM(\ga)$
given in \eqref{3.48}. Hereby, \aagf
I_n-\cM_n(\ga)\cM_n^*(\ga)=\eta_n(\ga)\eta_n^*(\ga),\
n\in\N,\label{5.8}\zzgf where \aagf
\eta_n(\ga)=\col(\ol\ga_1,\ol\ga_2D_{\ga_1},\dotsc,\ol\ga_n\prod\limits_{j=1}^{n-1}D_{\ga_j}).\label{5.9}\zzgf
\end{lem}

\begin{cor}[\zitaa{D06}{\ccor{5.4}}]\label{cor5.4}
The multiplicative decompositions
$$\cL_n(\ga)=\cM_n(\ga)\cdot \cM_n(W\ga)\cdot \cM_n(W^2\ga)\cdot\dotso,\ n\in\N$$
hold true.
\end{cor}

\begin{proof}
From the form \eqref{5.3} of the matrices $\cL_n(\ga)$ it can be seen that\\
$\lim\limits_{m\to\infty}\cL_n(W^m\ga)=I_n$ for all $n\in\N$. Now
using \eqref{5.6} we obtain the assertion.
\end{proof}

\begin{thm}[\cite{D3},\zitaa{D06}{\cthm{5.5}}]\label{thm5.5}
Let $\te(\zeta)$ be a function from $\cS$ the sequence
$(\ga_j)_{j=0}^\infty$ of Schur parameters of which belongs to
$\Ga l_2$. Assume that $\Dl$ is a simple unitary colligation of
the form \eqref{3.1} which satisfies $\te_\Dl(\zeta)=\te(\zeta)$.
Then the matrices $\cA_n(\ga)$ (see \eqref{5.4}) and their
determinants $(\si_n(\ga))_{n=1}^\infty$ have the following
properties: \aai\item[(1)] For $n\in\N$, it hold
$0\le\si_n(\ga)<1$ and $\si_n(\ga)\ge\si_{n+1}(\ga)$. Moreover,
$\lim\limits_{n\to\infty}\si_n(\ga)=0$. \item[(2)] If there exists
some $n_0\in\{0,1,2,\dotsc\}$ which satisfies $\si_{n_0}(\ga)>0$
and $\si_{n_0+1}(\ga)=0$, then $\rank \cA_n(\ga)=n_0$ for $n\ge
n_0$ holds true. Hereby, $n_0=\dim\cH_{\cG\cF}(=\dim\cH_{\cF\cG})$
where $\cH_{\cG\cF}$ and $\cH_{\cF\cG}$ are given via \eqref{4.3}.
Conversely, if $\dim\cH_{\cG\cF}(=\dim\cH_{\cF\cG})$ is a finite
number $n_0$ then $\si_{n_0}(\ga)>0$ and $\si_{n_0+1}(\ga)=0$.
\item[(3)] It holds \aagf
\cA_n(\ga)=\eta_n(\ga)\eta_n^*(\ga)+\cM_n(\ga)\cA_n(W\ga)\cM_n^*(\ga),\
n\in\N\label{5.10}\zzgf where $\cM_n(\ga)$ and $\eta_n(\ga)$ are
defined via \eqref{5.7} and \eqref{5.9}, respectively. \item[(4)]
Let $(\la_{n,j}(\ga))_{j=1}^n$ denote the increasingly ordered
sequence of eigenvalues of the matrix $\cA_n(\ga), n\in\N,$ where
each eigenvalue is counted with its multiplicity, then
\aagf 0&\le&\la_{n,1}(W\ga)\le\la_{n,1}(\ga)\le\la_{n,2}(W\ga)\le\la_{n,2}(\ga)\le\dotsb\nnu\\
&\le&\la_{n,n}(W\ga)\le\la_{n,n}(\ga)<1.\label{5.11}\zzgf Thus,
the eigenvalues of the matrices $\cA_n(\ga)$ and $\cA_n(W\ga)$
interlace. \item[(5)] For $n\in\N$, it holds
\aagf\Ga(h_1,h_2,\dotsc,h_n,G^*(1))=\si_n(W\ga)\prod\limits_{j=0}^n(1-|\ga_j|^2)\label{5.12}\zzgf
where $\Ga(h_1\!,\!h_2\!,\!\ldots\!,\!h_n,\!G^*(1))$ is the Gram
determinant of the vectors $(h_k)_{k=1}^n$ given by \eqref{5.1} and
the vector $G^*(1)$ defined by \eqref{2.2}. Hereby, the rank of the Gram
matrix of the vectors $(h_k)_{k=1}^n$ and $G^*(1)$ is equal to
$\rank\!\cA_n\!(W\ga)+1$. \item[(6)] If $\si_n(\ga)>0$ for every
$n\in\N$ then the sequence
$$\biggl(\prod\limits_{j=0}^n(1-|\ga_j|^2)\frac{\si_n(W\ga)}{\si_n(\ga)}\biggr)_{n=1}^\infty$$
monotonically decreases. Moreover, \aagf\lim\limits_{n\to\infty}
\frac{\si_n(W\ga)}{\si_n(\ga)}=\frac{1}{\Pi_0^2}\|P_{\cN_{\cG\cF}}G^*(1)\|^2\label{5.13}\zzgf
where $\Pi_0$ and $\cN_{\cG\cF}$ are defined via formulas
\eqref{3.17} and \eqref{4.2}, respectively. \item[(7)] Assume that
$\si_n(\ga)>0$ for every $n\in\N$. Then $\si_n(W^m\ga)>0$ for
every $m, n\in\N$. Moreover, if the limit \eqref{5.13} is positive,
then
\aagf\lim\limits_{n\to\infty}\frac{\si_n(W^{m+1}\ga)}{\si_n(W^m\ga)}>0\label{5.14}\zzgf
for every $m\in\N$. \zzi\end{thm}

\begin{proof} (1) Since by \coref{cor5.2} $\cA_n(\ga)$ is a Gram matrix then $\si_n(\ga)\ge0,\ n\in\N$.
On the other side, in view of $\ga\in\Ga l_2$ the matrix
$\cL_n(\ga)$ is invertible. Thus, from \eqref{5.4} we infer
$\si_n(\ga)<1,\ n\in\N$.
From \eqref{5.3} we get the block partition
\begin{equation}\label{E4.22-1111}
    \cL_{n+1}(\ga)
    =\matr{\cL_n(\ga)}{0_{n\times 1}}{b_{n}^*(\ga)}{\Pi_{n+1}}
\end{equation}
where \aagf
b_n(\ga)=\Pi_{n+1}\col(\ol{L_n(W\ga)},\ol{L_{n-1}(W^2\ga)},\dotsc,\ol{L_1(W^n\ga)}).\label{5.15A}\zzgf
From this it follows that
\aagf
\cA_{n+1}(\ga)=\matr{\cA_n(\ga)}{-\cL_n(\ga)b_n(\ga)}{-b_n^*(\ga)\cL_n^*(\ga)}{1-\Pi_{n+1}^2-b_n^*(\ga)b_n(\ga)}.\label{5.15}\zzgf
From \eqref{5.15} we find
\aagf\cA_{n+1}(\ga)=F_{n,1}(\ga)\matr{\cA_n(\ga)}{0}{0}{\cA_n^{[c]}(\ga)}F_{n,1}^*(\ga)\label{5.17}\zzgf
where $F_{n,1}(\ga)=\matr{I_n}{0}{X_{n,1}(\ga)}{1}\ ,\
X_{n,1}(\ga)=-b_n^*(\ga)\cL_n^*(\ga)\cA_n^{-1}(\ga)$ ,  \ and
\aagf\cA_n^{[c]}(\ga)=1-\Pi_{n+1}^2-b_n^*(\ga)b_n(\ga)-b_n^*(\ga)\cL_n(\ga)\cA_n^+(\ga)\cL_n(\ga)b_n(\ga)\label{5.17A}\zzgf
is the Schur complement of the matrix $\cA_n(\ga)$ in the matrix
$\cA_{n+1}(\ga)$. The symbol $\cA_n^+(\ga)$ stands for the
Moore--Penrose inverse of the matrix $\cA_n(\ga)$ (see, e.g.,
\cite[part 1.1]{DFK92}). Thus, \aagf
\si_{n+1}(\ga)=\si_n(\ga)\cA_n^{[c]}(\ga).\label{5.16}\zzgf In
view of $\cA_{n+1}(\ga)\ge 0$ we have $\cA_n^{[c]}(\ga)\ge0$.
Taking into account $\Pi_n>0$ and
$\lim\limits_{n\to\infty}\Pi_n=1$ from this and \eqref{5.17A} we
obtain $0\le \cA_n^{[c]}(\ga)<1$ and
$\lim\limits_{n\to\infty}\cA_n^{[c]}(\ga)=0$. Now (1) follows from
\eqref{5.16}.
\\ (3)
Using \eqref{5.6} we get \aagf
\cA_n(\ga)&=&I_n-\cL_n(\ga)\cL_n^*(\ga)
=I_n-\cM_n(\ga)\cL_n(W\ga)\cL_n^*(W\ga)\cM_n^*(\ga)\nnu\\
&=&I_n-\cM_n(\ga)\cM_n^*(\ga)+\cM_n(\ga)\cA_n(W\ga)\cM_n^*(\ga).\nnu\zzgf
Combining this with \eqref{5.8} we obtain \eqref{5.10}.
\\ (2) 
Assume that $n_0\in\{0,1,2,3,\dotsc\}$ satisfies
$\si_{n_0}(\ga)>0$ and $\si_{n_0+1}(\ga)=0$. If $n_0=0$ then using
$\si_1(\ga)=1-\prod\limits_{j=1}^\infty(1-|\ga_j|^2)$, we infer
$\ga_j=0, j\in\N$. Thus, \eqref{5.3} implies $\cL_n(\ga)=I_n,
n\in\N$ and $\cA_n(\ga)=0, n\in\N$. Consequently, from \eqref{5.2}
it follows that $\dim\cH_{\cG\cF}=0$. In view of \eqref{eq: 1.28} this
means $\dim\cH_{\cF\cG}=0$.

Let $n_0\in\N$. From \eqref{5.3} we get the block partition
\aagf\cL_{n+1}(\ga)=\matr{\Pi_1}{0}{B_{n+1}(\ga)}{\cL_n(W\ga)}\label{5.21}\zzgf
where \aagf
B_{n+1}(\ga)=\col(\Pi_2L_1(W\ga),\Pi_3L_2(W\ga),\dotsc,\Pi_{n+1}L_n(W\ga)).\label{5.22}\zzgf
From this, we obtain the block representation \aagf
\cA_{n+1}(\ga)=\matr{1-\Pi_1^2}{-\Pi_1B_{n+1}^*(\ga)}{-\Pi_1B_{n+1}(\ga)}{\cA_n(W\ga)-B_{n+1}(\ga)B_{n+1}^*(\ga)}.\label{5.23}\zzgf
We consider this block representation for $n=n_0+1$. Since $\det
\cA_{n_0+1}(\ga)=0$, in view of \eqref{5.11}, we have $\det
\cA_{n_0+1}(W\ga)=0$. Assume that $x\in\C^{n_0+1}, x\ne0$ and
$x\in\ker \cA_{n_0+1}(W\ga)$. Then from \eqref{5.23} we see that
the vector $\wt x:=\cl{0}{x}$ belongs to $\ker \cA_{n_0+2}(\ga)$.

Now we consider the block representation \eqref{5.15} for
$n=n_0+1$. Let $y\in\C^{n_0+1}, y\ne0$ and $y\in\ker
\cA_{n_0+1}(\ga)$. Then \eqref{5.15} implies that the vector $\wt
y:=\cl{y}{0}$ belongs to $\ker \cA_{n_0+2}(\ga)$. If the vectors
$\wt x$ and $\wt y$ are collinear then from their construction we
get that $\ker \cA_{n_0+2}(\ga)$ contains the vector
$w=\dcol{0}{z}{0}$ where $z\in\C^{n_0}$ and $z\ne0$. Then the
representation \eqref{5.15} for $n=n_0+1$ implies that
$\cl{0}{z}\in\ker \cA_{n_0+1}(\ga)$. Now using representation
\eqref{5.23} for $n=n_0+1$ we obtain $z\in\ker \cA_{n_0}(\ga)$.
However, $\si_{n_0}(\ga)>0$ and consequently $\ker
\cA_{n_0}(\ga)=\{0\}$. From this contradiction we infer that the
vectors $\wt x$ and $\wt y$ are not collinear. This means $\dim\
\ker \cA_{n_0+2}(\ga)\ge2$. Thus $\rank \cA_{n_0+2}(\ga)\le n_0$.
On the other side, using \eqref{5.15} we obtain $\ \rank
\cA_{n_0+2}(\ga)\ge\rank \cA_{n_0}(\ga)=n_0$. Hence, $\rank
\cA_{n_0+2}(\ga)=n_0$.

Applying the method of mathematical induction to the matrices
$\cA_{n_0+m}(\ga)$ by analogous considerations we get $\rank
\cA_{n_0+m}(\ga)=n_0$ for $m\in\N$. Now using \eqref{5.2},
\eqref{eq: 1.28} and the fact that $(h_n)_{n=1}^\infty$ is a system of
vectors which is total in $\cH_{\cG\cF}$ we find
$\dim\cH_{\cF\cG}=\dim\cH_{\cG\cF}=n_0$. The converse statement
follows immediately from \eqref{5.2} and the above considerations.
\\ (5) Because of $G^*(1)\in\cH_\cG$ and $\wt\psi_j\in\cH_\cG^\perp, j\in\N,$
from (2.65) and \eqref{5.1} we get
$$(G^*(1),h_k)=(G^*(1),\phi_k)=\ol\ga_k\prod\limits_{j=0}^{k-1}D_{\ga_j},\ k\in\N.$$
Combining this with \eqref{5.9} it follows that
\aagf &&\col((G^*(1),h_1),(G^*(1),h_2),\dotsc,(G^*(1),h_n))\nnu\\
&=&D_{\ga_0}\col(\ol\ga_1,\ol\ga_2D_{\ga_1},\dotsc,\ol\ga_n\prod\limits_{j=1}^{n-1}D_{\ga_j})=D_{\ga_0}\eta_n(\ga).\nnu\zzgf
Thus, taking into account $(G^*(1),G^*(1))=1-|\ga_0|^2$ and using
\eqref{5.2} and \eqref{5.10} we get \aagf
&&\Ga(h_1,h_2,\dotsc,h_n,G^*(1))= \left|\begin{array}{cc}
\cA_n(\ga) & D_{\ga_0}\eta_n(\ga) \\
D_{\ga_0}\eta_n^*(\ga) & 1-|\ga_0|^2
\end{array}\right|\nnu\\
&=&(1-|\ga_0|^2) \left|\begin{array}{cc}
\eta_n(\ga)\eta_n^*(\ga)+\cM_n(\ga)\cA_n(W\ga)\cM_n^*(\ga) & \eta_n(\ga) \\
\eta_n^*(\ga) & 1
\end{array}\right|.\nnu\zzgf
Subtracting now the $(n+1)$--th column multiplied by $\ga_1$ from
the first column and, moreover for $k\in\{2,\dotsc,n\}$,
subtracting the $(n+1)$--th column multiplied by
$\ga_k\prod\limits_{j=1}^{k-1}D_{\ga_j}$ from the $k$--th column,
we obtain \aagf \Ga(h_1,h_2,\dotsc,h_n,G^*(1)) &=&(1-|\ga_0|^2)
\left|\begin{array}{cc}
\cM_n(\ga)\cA_n(W\ga)\cM_n^*(\ga) & \eta_n(\ga) \\
0 & 1
\end{array}\right|\nnu\\
&=&(1-|\ga_0|^2)\si_n(W\ga)|\det \cM_n(\ga)|^2.\label{5.24}\zzgf
From \eqref{5.7} we see $\det
\cM_n(\ga)=\prod\limits_{j=1}^nD_{\ga_j}$. Thus, \eqref{5.12}
follows from \eqref{5.24}. From the concrete form of the matrix
$$\matr{\cM_n(\ga)\cA_n(W\ga)\cM_n^*(\ga)}{\eta_n(\ga)}{0}{1}$$
it is clear that the rank of the Gram matrix of the vectors
$(h_k)_{k=1}^n$ and $G^*(1)$ is equal to $\rank \cA_n(W\ga)+1$.
\\ (6) From \eqref{5.12} we get
\aagf\prod\limits_{j=1}^n(1-|\ga_j|^2)\frac{\si_n(W\ga)}{\si_n(\ga)}
=\frac{\Ga(h_1,h_2,\dotsc,h_n,G^*(1))}{\Ga(h_1,h_2,\dotsc,h_n)},\
n\in\N.\label{5.25}\zzgf Denote by $P_n$ the orthoprojection from
$\cH$ onto $\cH_\cG\ominus\Lin\{h_1,h_2,\dotsc,h_n\},n\in\N$.
Because of $G^*(1)\in\cH_\cG$ using well known properties of Gram
determinants (see, e.g., Akhiezer/Glasman \cite[Chapter I]{MR1255973}) we
see
\aagf\frac{\Ga(h_1,h_2,\dotsc,h_n,G^*(1))}{\Ga(h_1,h_2,\dotsc,h_n)}=\|P_nG^*(1)\|^2.\label{5.26}\zzgf
This implies that the sequence on the left hand side of formula
\eqref{5.25} is monoto\-nically decreasing. Since the sequence
$(h_n)_{n=1}^\infty$ is total in $\cH_{\cG\cF}$ the decomposition
\eqref{4.4} shows that $P_{\cN_{\cG\cF}}$ is the strong limit of
the sequence $(P_n)_{n\in\N}$. Therefore, \eqref{5.13} follows from
\eqref{5.25} and \eqref{5.26}.
\\ (7) Assume that $\si_n(\ga)>0$ for every $n\in\N$. Then the block representation
\eqref{5.23} shows that $\si_n(W\ga)>0$ for every $n\in\N$. From
this by induction we get $\si_n(W^m\ga)>0$ for all $n,m\in\N$.
Assume now that the limit \eqref{5.13} is positive. This means that
$\cN_{\cG\cF}\ne\{0\}$ and $P_{\cN_{\cG\cF}}G^*(1)\ne0$ are
satisfied. As already mentioned, the operator
$\Rstr_{\cN_{\cG\cF}}T^*$ is the maximal unilateral shift
contained in $T_\cG^*$. Denote by $\tau, \|\tau\|=1,$ a basis
vector of the generating wandering subspace of this shift. Then
the sequence $(T^{*(n-1)}\tau)_{n\in\N}$ is an orthonormal basis
in $\cN_{\cG\cF}$. Since $\cN_{\cG\cF}\subseteq\cH_\cF^\perp$ (see
\eqref{4.2}) and since the part $(\psi_k)_{k=1}^\infty$ of the
canonical basis (2.12) is an orthonormal basis in $\cH_\cF^\perp$,
we obtain the representation
\aagf\tau=\be_1\psi_1+\be_2\psi_2+\dotsb+\be_n\psi_n+\dotsb\label{5.27}\zzgf
where $\be_j=(\tau,\phi_j), j\in\N$. Because of
$T^*\psi_j=\psi_{j+1}, j\in\{1,2,\dotsc\}$, we get \aagf
T^{*k}\tau=\be_1\psi_{k+1}+\be_2\psi_{k+2}+\dotsb+\be_n\psi_{k+n}+\dotsb,\
k\in\N.\label{5.28}\zzgf From (2.68) we see
\aagf(G^*(1),\psi_1)=\prod\limits_{j=0}^\infty D_{\ga_j}\ ,\
(G^*(1),\psi_k)=0\ ,\ k\in\{2,3,\dotsc\}.\label{5.29}\zzgf
Combining \eqref{5.28} and \eqref{5.29} it follows that
$(G^*(1),T^{*k}\tau)=0,\ k\in\N$. Thus,
$$P_{\cN_{\cG\cF}}G^*(1)=\sum\limits_{k=0}^\infty(G^*(1),T^{*k}\tau)T^{*k}\tau
=(G^*(1),\tau)\tau=\ol\be_1\prod\limits_{j=0}^\infty
D_{\ga_j}\tau.$$ This means \aagf \|P_{\cN_{\cG\cF}}G^*(1)\|=
|\be_1|\prod\limits_{j=0}^\infty D_{\ga_j}.\label{5.30}\zzgf Thus,
the condition $\|P_{\cN_{\cG\cF}}G^*(1)\|\ne0$ is equivalent to
$(\tau,\psi_1)\ne0$. This is equivalent to the fact that $\psi_1$
is not orthogonal to $\cN_{\cG\cF}$. Now we pass to the model
based on the sequence $W\ga=(\ga_1,\ga_2,\ga_3,\dotsc)$. We will
denote the corresponding objects associated with this model by a
lower index $1$. For example, $G_1, \cN_{\cG\cF,1}, \psi_{j1}$
etc. The identity \eqref{4.21} takes now the form $\
\cN_{\cG\cF,1}=\ker \cQ_1^*(\ga)$ where the matrix of the
operator $\cQ_1(\ga)$ is obtained by deleting the first row (or
first column) in the matrix of $\cQ(\ga)$. Therefore, if the
vector $\tau$ with coordinate sequence $(\be_j)_{j=1}^\infty$ (see
\eqref{5.27}) belongs to $\cN_{\cG\cF}$ then in view of \eqref{4.21}
it belongs to $\ker \cQ^*(\ga)$. Thus, in view of the Hankel
structure of $\cQ^*(\ga)$ it follows that the vector with these
coordinates also belongs to $\ker \cQ_1^*(\ga)$. Hence, in view of
\eqref{4.21} this vector belongs to $\cN_{\cG\cF,1}$. Thus, the
condition $\be_1\ne0$ implies that $\psi_{1,1}$ is not orthogonal
to $\cN_{\cG\cF;1}$. This is equivalent to $\|P_{\cN_{\cG\cF,1}}
G_1^*(1)\|\ne0$. Hence, if the limit \eqref{5.13} is positive then the
limit \eqref{5.14} is positive for $m=1$. The case
$m\in\{2,3,4,\dotsc\}$ is handled by induction.
\\ (4) See \cite[Theorem 5.5]{D06}.
\end{proof}

\begin{rem}\label{r1.28}
From Theorems \ref{thm5.1} and \ref{thm5.5} it follows that $\operatorname{dim}\cH_{\cG\cF}=m$ if and only if the first $m$ vectors of the sequence $(h_k)_{k=1}^\infty$ are linear independent (which means they form a basis of the space $\cH_{\cG\cF}$), whereas all vectors of the remaining sequence $(h_k)_{k=m+1}^\infty$ are linear combinations of the sequence $(h_k)_{k=1}^m$.
\end{rem}

Using considerations as in the proof of statement (7) of the
preceding Theorem and taking into account the "'layered"'
structure of the model (see \thref{th2.1} and \coref{cor3.6}), we
obtain the following result.

\begin{cor}[\zitaa{D06}{\ccor{5.6}}]\label{cor5.6}
Suppose that the assumptions of \thref{thm5.5} are fulfilled.
Moreover, assume that $\si_n(\ga)>0$ for all $n\in\N$. Suppose
that there exists an index $m\in\{0,1,2,\dotsc\}$ for which
\eqref{5.14} is satisfied and denote by $m_0(\ga)$ the smallest
index with this property. Then for $m\ge m_0(\ga)$ the limit
\eqref{5.14} is positive. The number $m_0(\ga)$ is characterized by
the following condition. If $\tau$ is a normalized basis vector of
the generating wandering subspace of $V_{T_\cG^*}$ then
\aagf\tau=\be_{m_0(\ga)+1}\psi_{m_0(\ga)+1}+\be_{m_0(\ga)+2}\psi_{m_0(\ga)+2}+\dotsb{\mbox\
\ and\ \ }\be_{m_0(\ga)+1}\ne0.\label{5.31}\zzgf Hereby, the
relations \aagf
m_0(W\ga)=\bcase{m_0(\ga)}{m_0(\ga)=0,}{m_0(\ga)-1}{m_0(\ga)\ge
1}\label{5.32}\zzgf hold true.
\end{cor}

\begin{defn}[\zitaa{D06}{\cdef{5.7}}]\label{def5.7}
Assume that $\ga\in\Ga l_2$. Let $\te(\zeta)$ be the Schur
function associated with $\ga$ and let $\Dl$ be a simple unitary
colligation of the form \eqref{3.1} which satisfies
$\te(\zeta)=\te_\Dl(\zeta)$. If $\cN_{\cG\cF}\ne\{0\}$ then the
number $m_0(\ga)$ characterized by condition \eqref{5.31} is called
the level of the subspace $\cN_{\cG\cF}$ or also the level of the
sequence $\ga$. If $\cN_{\cG\cF}=\{0\}$ we set $m_0(\ga):=\infty$
\end{defn}

Thus, it is convenient to consider the vectors
$\psi_1,\psi_2,\psi_3,\dotsc$ as levels of the subspace
$\cH_\cF^\perp$. Hereby, we will say that the vector $\psi_k,\
k\in\{1,2,3,\dotsc\}$ determines the $k$--th level. Then the
number $m_0(\ga)$ expresses the number of levels which have to be
overcome in order to "reach" the subspace $\cN_{\cG\cF}$.

\thref{t1.23} implies that a function $\te(\zeta)$ belongs to
$\cS\Pi\setminus J$ if and only if $m_0(\ga)<\infty$. Hereby, as
\eqref{5.30} shows, the verification of the statement
$\cN_{\cG\cF}\ne\{0\}$ with the aid of the vector $G^*(1)$ is only
possible in the case $m_0(\ga)=0$, this means that $\cN_{\cG\cF}$
"begins" at the first level. Therefore, if $\cN_{\cG\cF}\ne\{0\}$
but $P_{\cN_{\cG\cF}}G^*(1)=0$ then it is necessary to pass from
the sequence $\ga$ to the sequence $W\ga$. Then from \eqref{5.32}
it follows that the subspace $\cN_{\cG\cF}$ will be "found" after
a finite number of such steps.

\subsection{Some criteria of pseudocontinuability of Schur functions}

\begin{thm}[\cite{D98},\zitaa{D06}{\cthm{5.8}}]\label{thm5.8} 
Let $\te(\zeta)$ be a function from $\cS$ the sequence $(\ga_j)_{j=0}^\infty$ of Schur parameters of
which belongs to $\Ga l_2$. Then the vector
$$\xi(\ga):=(Q(W\ga),Q(W^2\ga),\dotsc,Q(W^n\ga),\dotsc)$$ where
$Q(\ga)$ is given in \eqref{3.18}, belongs to $l_2$. The function
$\te(\zeta)$ admits a pseudocontinuation into $\D_e$ if and only
if $\xi(\ga)$ is not cyclic for the coshift $W$ (see \eqref{3.14})
in $l_2$.
\end{thm}

\begin{proof}
The vector $\xi(\ga)$ is not cyclic for $W$ in $l_2$ if and only
if $\ker \cQ(\ga)\ne\{0\}$, where $\cQ(\ga)$ is defined via
\eqref{3.43}. The Hankel structure of $\cQ(\ga)$ implies that $\ker
\cQ(\ga)\ne\{0\}$ if and only if $\ker \cQ^*(\ga)\ne\{0\}$. Now
the assertion of the Theorem follows from \lmref{lm4.7} and
\thref{t1.23}.
\end{proof}

The following series of quantitative criteria starts with a
criterion which characterizes the Schur parameters of a rational
function of the Schur class $\cS$.

\begin{thm} [\cite{D3},\zitaa{D06}{\cthm{5.9}}] \label{thm5.9}
Let $\te(\zeta)\in\cS$ and denote $\ga=(\ga_j)_{j=0}^\ome$ the
sequence of its Schur parameters. Then the function $\te(\zeta)$
is rational if and only if one of the following two conditions is
satisfied: \aai\item[(1)] $\ome<\infty,$ i.e., $|\ga_\ome|=1$.
\item[(2)] $\ga\in\Ga l_2$ and there exists an index $n\in\N$ such
that $\si_n(\ga)=0$, where $\si_n(\ga)$ is defined via \eqref{5.5}.
\zzi Hereby: \aai\item[(1a)] $\ome=0$ if and only if
$\te(\zeta)\equiv\ga_0,\ |\ga_0|=1$. \item[(1b)] $\ome\in\N$ if
and only if $\te(\zeta)$ is a finite Blaschke product of degree
$\ome$. \zzi Let $\ga\in\Ga l_2$. If $n_0\in\{0,1,2,\dotsc\}$
satisfies $\si_{n_0}(\ga)>0$ and $\si_{n_0+1}(\ga)=0$ then:
\aai\item[(2a)] $n_0=0$ if and only if $\te(\zeta)\equiv\ga_0,\
|\ga_0|<1$, i.e., if and only if $\te(\zeta)$ is not a constant
function with unitary value but a block of a constant unitary
$2\times 2$ matrix. \item[(2b)] $n_0\in\N$ if and only if
$\te(\zeta)$ is not a finite Blaschke product, but a block of a
finite $2\times 2$--matrix--valued Blaschke--Potapov product of
the form \eqref{4.2B} where $n_0$ is the smallest number of
elementary Blaschke--Potapov factors forming such a $2\times
2$--Blaschke--Potapov product. \zzi
\end{thm}

\begin{proof}
All what concerns condition (1) is the well known criterion of
Schur \cite[Part I]{Schur} who described the Schur parameters of
finite Blaschke products. Condition (2) follows from the
corresponding assertions (2) of Theorems \ref{thm5.5} and
\ref{thm4.5}.
\end{proof}

\begin{thm} [\cite{D3},\zitaa{D06}{\cthm{5.10}}]\label{thm5.10}
Let $\te(\zeta)\in\cS$ and denote by $\ga=(\ga_j)_{j=0}^\infty$
the sequence of its Schur parameters. Let $\si_n(\ga),\
n\in\{0,1,2,\dotsc\}, $ be the determinants defined via
\eqref{5.5}.Then $\te(\zeta)\in\cS\Pi\setminus J$ if and only if \
$\ga\in\Ga l_2$ and one of the following conditions is satisfied:
\aai\item[(a)] There exists an index $n\in\N$ such that
$\si_n(\ga)=0$. \item[(b)] If $\si_n(\ga)>0$ for all $n\in\N$ then
there exists a number $m\in\{0,1,2,\dotsc\}$ such that
\aagf\lim\limits_{n\to\infty}\frac{\si_n(W^{m+1}\ga)}{\si_n(W^m\ga)}>0.\label{5.33}\zzgf
\zzi  Suppose that there exists an index $m$ for which \eqref{5.33}
is satisfied and denote $m_0(\ga)$ the smallest index with this
property. Then \eqref{5.33} is satisfied for all $m\ge m_0(\ga)$.
The number $m_0(\ga)$ is characterized by condition \eqref{5.31},
i.e., $m_0(\ga)$ is the level of the sequence $\ga$.
\end{thm}

\begin{proof}
\thref{thm5.9} implies that condition (a) is satisfied if and only
if the function $\te(\zeta)$ is rational and therefore belongs to
$\cS\Pi\setminus J$. In the case that $\te(\zeta)$ is not rational
the assertions of the Theorem follow from the assertions (6) and
(7) of \thref{thm5.5}, \coref{cor5.6} and \thref{t1.23}.
\end{proof}

For the proof of the next criterion we need additional facts about
the matrices $\cA_n(\ga),\ n\in\N,$ and their determinants.

\begin{lem} [\zitaa{D06}{\clem{5.11}}]\label{lm5.11}
Assume that $\ga\in\Ga l_2$ and $\si_{n+1}(\ga)>0$ for some
$n\in\N$. Then
\aagf\prod\limits_{j=0}^n(1-|\ga_j|^2)\frac{\si_n(W\ga)}{\si_n(\ga)}
=(1+\frac{1}{1-|\ga_1|^2}\La_n^*(\ga)\cA_n^{-1}(W\ga)\La_n(\ga))^{-1}\label{5.34}\zzgf
where \aagf\!\!\La_n(\ga)=\col(\ol\ga_1\ ,\ \ol\ga_2D_{\ga_2}^{-1}
\ ,\ \ol\ga_3D_{\ga_2}^{-1}D_{\ga_3}^{-1}\ ,\dotsc, \
\ol\ga_n\prod\limits_{j=2}^nD_{\ga_j}^{-1}).\label{5.35}\zzgf
\end{lem}

\begin{proof}
From formula \eqref{5.23} it follows that in the considered case
the matrix $\cA_n(W\ga)$ is invertible. Therefore, taking into
account the invertibility of $\cM_n(\ga)$ and using \eqref{5.10} we
get \aagf
\cA_n(\ga)=\cM_n(\ga)\cA_n^{\frac{1}{2}}(W\ga)[X_n(\ga)X_n^*(\ga)+I_n]\cA_n^{\frac{1}{2}}(W\ga)\cM_n^*(\ga)\label{5.36}\zzgf
where
$X_n(\ga)=\cA_n^{-\frac{1}{2}}(W\ga)\cM_n^{-1}(\ga)\eta_n(\ga)$.
By direct computation it is checked that \aagf
\cM_n(\ga)\La_n(\ga)=D_{\ga_1}\eta_n(\ga).\label{5.37}\zzgf Thus,
$X_n(\ga)=D_{\ga_1}^{-1}\cA_n^{-\frac{1}{2}}(W\ga)\La_n(\ga).$
Taking the determinant in \eqref{5.36} and using the form
\eqref{5.7} of the matrix $\cM_n(\ga)$ we obtain
$$\si_n(\ga)=\prod\limits_{j=0}^n(1-|\ga_j|^2)\si_n(W\ga)\det(I_n+X_n(\ga)X_n^*(\ga)).$$
From this and the identity
$\det(I_n+X_n(\ga)X_n^*(\ga))=1+X_n^*(\ga)X_n(\ga)$ (see, e.g.,
\cite[Lemma 1.1.8]{DFK92}) we obtain \eqref{5.34}.
\end{proof}

\begin{lem} [\zitaa{D06}{\clem{5.12}}]\label{lm5.12}
Assume that $\ga\in\Ga l_2$ and that $\si_{n}(\ga)=0$ for some
$n\in\N$. Denote by $m_0(\ga)$ the level of the sequence $\ga$,
i.e., $m_0(\ga)$ is characterized by condition \eqref{5.31}. Then:
\aai\item[(a)] For $m\ge m_0(\ga)$ it holds $\rank
\cA_n(W^m\ga)=\rank \cA_n(W^{m+1}\ga),\ n\in\N$. \item[(b)] If $\
m_0(\ga)\ge 1\ ,\ m\in\{0,1,\dotsc,m_0(\ga)-1\}$ and $n_0(m)$ is
such that $\si_{n_0(m)}(W^m\ga)>0$ and $\si_{n_0(m)+1}(W^m\ga)=0$
then
\aagf&&\rank \cA_n(W^m\ga)\ge\rank \cA_n(W^{m+1}\ga),\ n\in\{1,2,\dotsc,n_0(m)-1\},\nnu\\
&&\rank \cA_n(W^m\ga)=\rank \cA_n(W^{m+1}\ga)+1,\ n\ge
n_0(m).\nnu\zzgf \zzi
\end{lem}

\begin{proof} (a)
It suffices to consider the case $m_0(\ga)=0$. In the opposite
case it is necessary to change from $\ga$ to $W^{m_0(\ga)}\ga$.
Thus, assume that $m_0(\ga)=0$. Because of $\be_1\ne0$ from
\eqref{5.30} we get $P_{\cN_{\cG\cF}}G^*(1)\ne0$. This means, for
arbitrary $n\in\N$ the rank of the Gram matrices of the vectors
$(h_j)_{j=1}^n$ is one smaller than the rank of the Gram matrix of
the vectors $(h_j)_{j=1}^n$ and $G^*(1)$. Now for the case $m=0$
the assertion
follows from statement (5) of \thref{thm5.5}. The case $m>0$ can be treated analogously. \\
(b) It suffices to consider the case $m=0$. The other cases can be
considered analogously. Thus, let $m=0$. In the considered case we
proceed as above and take into account that now $\be_1=0$. Hence,
$P_{\cN_{\cG\cF}}G^*(1)=0$ and $G^*(1)\in\cH_{\cG\cF}$.
\end{proof}

\begin{thm} [\cite{D3},\zitaa{D06}{\cthm{5.13}}]\label{thm5.13}
Let $\te(\zeta)\in\cS$ and denote by $\ga$ the sequence of its
Schur parameters. Then $\te(\zeta)\in\cS\Pi\setminus J$ if and
only if $\ga\in\Ga l_2$ and there exist numbers
$m\in\{0,1,2,\dotsc\}$ and $c>0$, which depend on $m$, such that
\aagf\matr{\cA(W^{m+1}\ga)}{\La(W^m\ga)}{\La^*(W^m\ga)}{c}\ge
0\label{5.38}\zzgf where $$\cA(\ga)=I-\cL(\ga)\cL^*(\ga),$$
\aagf\La(\ga)=\col(\ol\ga_1\ ,\ \ol\ga_2D_{\ga_2}^{-1} \ ,\
\ol\ga_3D_{\ga_2}^{-1}D_{\ga_3}^{-1}\ ,\dotsc, \
\ol\ga_n\prod\limits_{j=2}^nD_{\ga_j}^{-1}\
,\dotsc)\label{5.39}\zzgf and $\cL(\ga)$ is given via \eqref{3.42}.
Suppose that there exists an index $m$ for which \eqref{5.38} is
satisfied and denote by $m_0(\ga)$ the smallest index with this
property. Then \eqref{5.38} is satisfied for all $m\ge m_0(\ga)$.
The number $m_0(\ga)$ is characterized by condition \eqref{5.31},
i.e., $m_0(\ga)$ is the level of the sequence $\ga$.
\end{thm}

\begin{cor} [\zitaa{D06}{\ccor{5.14}}]\label{cor5.14}
Let $\te(\zeta)\in\cS$ and denote by $\ga$ the sequence of its
Schur parameters. Then $\te(\zeta)\in\cS\Pi\setminus J$ if and
only if $\ga\in\Ga l_2$ and there exists an index
$m\in\{0,1,2,\dotsc\}$ for which the vector $\La(W^m\ga)$ belongs
to the range of the operator $\cA^{\frac{1}{2}}(W^{m+1}\ga)$.
Suppose that there exists such an index $m$ and denote by
$m_0(\ga)$ the smallest one. Then for all $m\ge m_0(\ga)$ the
vector $\La(W^m\ga)$ belongs to the range of the operator
$\cA^{\frac{1}{2}}(W^{m+1}\ga)$. The number $m_0(\ga)$ is
characterized by condition \eqref{5.31}. This means that $m_0(\ga)$
is the level of the sequence $\ga$.
\end{cor}

\begin{proof}
Because of $\cA(W^m\ga)\ge 0\ ,\ m\in\{0,1,2,\dotsc\}, $ the
assertion follows from \thref{thm5.13} and the well known
criterion for nonnegative Hermitian block matrices (see, e.g.,
\cite[Lemma 2.1]{BDK}).
\end{proof}

\begin{rem}[\zitaa{D06}{\crem{5.15}}]\label{re5.15}
The matrix representation \eqref{3.9} implies
\[
    \cA(\ga)
    =I-\cL(\ga)\cL^*(\ga)
    =\cR(\ga)\cR^*(\ga).
\]
Therefore, \coref{cor5.14} remains true if the range of the operator \(\cA^\frac{1}{2}(W^{m+1}\ga)\) is replaced by the range of the operator \(\cR(W^{m+1}\ga)\).
\end{rem}

\section{General properties of the Schur parameter sequences of pseudocontinuable
    Schur functions}\label{sec5}

\subsection{On some properties of the Schur parameter sequences of pseudocontinuable
Schur functions}

In the term $\La_n^*(\ga)\cA_n^{-1}(W\ga)\La_n(\ga)$ (see \lmref{lm5.11}) the parameter $\ga_1$ is only contained in $\La_n(\ga)$.
This enables us to give a more concrete description of the dependence of this expression on $\ga_1$.
For this we consider the representation \eqref{5.23}.
We assume that for $n\in\N$ the matrix $\cA_{n+1}(\ga)$ is invertible and introduce the notations
\begin{equation}
H_n(\ga)
:=\cA_n(W\ga)-B_{n+1}(\ga)B_{n+1}^*(\ga)\label{5.45}
\end{equation}
and
\begin{equation}
H_n^{[c]}(\ga)
:=1-\Pi_1^2-\Pi_1^2B_{n+1}^*(\ga)H_n^{-1}(\ga)B_{n+1}(\ga).\label{5.46}
\end{equation}
Then from \eqref{5.23} it follows
\begin{multline*}
    \cA_{n+1}(\ga)
    =\matr{1}{-\Pi_1B_{n+1}^*(\ga)H_n^{-1}(\ga)}{0}{I_n}\matr{H_n^{[c]}(\ga)}{0}{0}{H_n(\ga)}\\
    \times\matr{1}{0}{-\Pi_1H_n^{-1}(\ga)B_{n+1}(\ga)}{I_n}.
\end{multline*}
Thus,
\[\begin{split}
    &\cA_{n+1}^{-1}(\ga)\\
    &=\matr{1}{0}{\Pi_1H_n^{-1}(\ga)B_{n+1}(\ga)}{I_n}\matr{\frac{1}{H_n^{[c]}(\ga)}}{0}{0}{H_n^{-1}(\ga)}\matr{1}{\Pi_1B_{n+1}^*(\ga)H_n^{-1}(\ga)}{0}{I_n}\\
    &=\frac{1}{H_n^{[c]}(\ga)}\cl{1}{\Pi_1H_n^{-1}(\ga)B_{n+1}(\ga)}\left(1,\Pi_1B_{n+1}^*(\ga)H_n^{-1}(\ga)\right)+\matr{0}{0}{0}{H_n^{-1}(\ga)}.
\end{split}\]
Using this product representation and the equality
\(\La_{n+1}(\ga)=\begin{psmallmatrix}\ol\ga_1\\D_{\ga_2}^{-1}\La_n(W\ga)\end{psmallmatrix}\) we find
\aagf
    &&\La_{n+1}^*(\ga)\cA_{n+1}^{-1}(W\ga)\La_{n+1}(\ga)\label{5.47}\\
    &=&\frac{1}{H_n^{[c]}(W\ga)}|\ga_1+
    \Pi_3\La_n^*(W\ga)H_n^{-1}(W\ga)B_{n+1}(W\ga)|^2+\nnu\\
    &&+\frac{1}{1-|\ga_2|^2}\La_n^*(W\ga)H_n^{-1}(W\ga)\La_n(W\ga).\nnu
\zzgf
Hereby, $\ga_1$ occurs only in the expression in the modules.

\begin{defn}\label{de5.16}
Denote $\Pi\Ga$ (resp. $\Pi\Ga l_2$) the set of all $\ga\in\Ga$
for which the associated Schur function belongs to $\cS\Pi$ (resp.
$\cS\Pi\setminus J$).
\end{defn}

\begin{lem}[\zitaa{D06}{\clem{5.17}}]\label{lm5.17}
Let $\ga\in\Pi\Ga l_2$. Assume that $\si_n(\ga)>0$ for all
$n\in\N$ and $m_0(\ga)=0$. Then
\aagf\lim\limits_{n\to\infty}H_n^{[c]}(\ga)=0\label{5.48}\zzgf
where $H_n^{[c]}(\ga)$ is given via \eqref{5.46}.
\end{lem}

\begin{proof}
In view of \eqref{5.45} for $n\in\N$ we get
\aagf &&B_{n+1}^*(\ga)H_n^{-1}(\ga)B_{n+1}(\ga)\nnu\\
&=&B_{n+1}^*(\ga)\cA_n^{-\frac{1}{2}}(W\ga)
(I_n-\cA_n^{-\frac{1}{2}}(W\ga)B_{n+1}(\ga)B_{n+1}^*(\ga)\cA_n^{-\frac{1}{2}}(W\ga))^{-1}\cdot\nnu\\
&&\cdot\cA_n^{-\frac{1}{2}}(W\ga)B_{n+1}(\ga)=\sum\limits_{k=1}^\infty
q_n^k(\ga)=\frac{q_n(\ga)}{1-q_n(\ga)}\nnu\zzgf where
$q_n(\ga)=B_{n+1}^*(\ga)\cA_n^{-1}(W\ga)B_{n+1}(\ga)\ ,\ n\in\N.$
Thus, \aagf
H_n^{[c]}(\ga)=1-\Pi_1^2-\Pi_1^2\frac{q_n(\ga)}{1-q_n(\ga)}\ ,\
n\in\N.\label{5.50}\zzgf

Using the block partition \eqref{5.21} of the matrix
$\cL_{n+1}(\ga)$ we obtain for $n\in\N$ by analogy with the
derivation of the formulas \eqref{5.16} (see, e.g., \cite[Lemma
1.1.7]{DFK92}) \aagf &&\!\!\!\!\!\!\!\!\!\!\!\si_{n+1}(\ga)
=\det(I_{n+1}-\cL_{n+1}^*(\ga)\cL_{n+1}(\ga))\nnu\\
&&\!\!\!\!\!\!\!\!\!\!\!=\det\matr{1-\Pi_1^2-B_{n+1}^*(\ga)B_{n+1}(\ga)}{-B_{n+1}^*(\ga)\cL_n(W\ga)}{-\cL_n^*(W\ga)B_{n+1}(\ga)}{I_n-\cL_n^*(W\ga)\cL_n(W\ga)}\nnu\\
&&\!\!\!\!\!\!\!\!\!\!\!=\si_n(W\ga) \{ \ 1\!-\!\Pi_1^2\!-
B_{n+1}^*(\ga)( \ I_n+\cL_n(W\ga)\cA_n^{-1}(W\ga)\cL_n^*(W\ga) \ )B_{n+1}(\ga) \}\nnu\\
&&\!\!\!\!\!\!\!\!\!\!\!=\si_n(W\ga)(1-\Pi_1^2-B_{n+1}^*(\ga)\cA_n^{-1}(W\ga)B_{n+1}(\ga)).\nnu\zzgf
This means $\si_{n+1}(\ga)=\si_n(W\ga)(1-\Pi_1^2-q_n(\ga))\ ,\
n\in\N$. Comparing this expression with \eqref{5.16} we obtain
$$1=\frac{\si_n(W\ga)}{\si_n(\ga)}\cdot\frac{1-\Pi_1^2-q_n(\ga)}{\cA_n^{[c]}(\ga)}.$$
It holds $\lim\limits_{n\to\infty}\cA_n^{[c]}(\ga)=0$. Hereby, in
view of $m_0(\ga)=0, $ the limit \eqref{5.13} is positive. Thus,
$\lim\limits_{n\to\infty}q_n(\ga)=1-\Pi_1^2$. Now \eqref{5.50}
implies \eqref{5.48}.
\end{proof}

\begin{lem}[\zitaa{D06}{\clem{5.18}}]\label{lm5.18}
Let $\ga\in\Pi\Ga l_2$. Assume that $m_0(\ga)=0$ and that there
exists an index $n_0\in\N$ such that $\si_{n_0}(\ga)>0$ and
$\si_{n_0+1}(\ga)=0$ are satisfied. Then there exists a unique
constant vector $a=\col(a_1,\dotsc,a_{n_0})$ such that $a_1\neq0$
and for $j\in\{0,1,2,\dotsc\}$ the relations \aagf
(I_{n_0}-\cL_{n_0}^*(W^j\ga)\cL_{n_0}(W^j\ga))a=\frac{1}{\Pi_{n_0+j+1}}b_{n_0}(W^j\ga),\label{5.51}\zzgf
\aagf \cl{\Pi_{n_0+j+1}\cL_{n_0}(W^j\ga)a}{1}\in\ker
\cA_{n_0+1}(W^j\ga)\label{5.52}\zzgf and \aagf
\cM_{n_0+1}^*(W^j\ga)\cl{\!\!\Pi_{n_0+j+1}\cL_{n_0}(W^j\ga)a\!\!}{1}
=D_{\ga_{n_0+j+1}}\cl{\!\!\Pi_{n_0+j+2}\cL_{n_0}(W^{j+1}\ga)a\!\!}{1}\label{5.53}\zzgf
are fulfilled where $\Pi_n, b_n(\ga)$ and $\cM_n(\ga)$ are defined
via \eqref{3.17}, \eqref{5.15A} and \eqref{5.7}, respectively.
\end{lem}

\begin{proof}
From the assumptions of the lemma we obtain analogously to
\eqref{5.17} \aagf
\cA_{n_0+1}(\ga)=\matr{I_{n_0}}{0}{X_{n_0,1}(\ga)}{1}\matr{\cA_{n_0}(\ga)}{0}{0}{0}\matr{I_{n_0}}{X_{n_0,1}^*(\ga)}{0}{1},\label{5.54}\zzgf
where \aagf
X_{n_0,1}^*(\ga)=-\cA_{n_0}^{-1}(\ga)\cL_{n_0}(\ga)b_{n_0}(\ga).\label{5.55}\zzgf
Let
$a(\ga):=(I_{n_0}-\cL_{n_0}^*(\ga)\cL_{n_0}(\ga))^{-1}b_{n_0}(\ga).$
Thus, \aagf
(I_{n_0}-\cL_{n_0}^*(\ga)\cL_{n_0}(\ga))a(\ga)=b_{n_0}(\ga).\label{5.57}\zzgf
From \eqref{5.54} and \eqref{5.55} we see that the vector \aagf
\cl{\cL_{n_0}(\ga)a(\ga)}{1}\label{5.56}\zzgf belongs to $\ker
\cA_{n_0+1}(\ga)$.
Because of $m_0(\ga)=0$ then \lmref{lm5.12} implies
that for arbitrary $j\in\N$ the relations
\aagf\si_{n_0}(W^j\ga)>0\ ,\
\si_{n_0+1}(W^{j+1}\ga)=0\label{5.58}\zzgf hold true. This means
$\dim\ker \cA_{n_0+1}(W^j\ga)=1\ ,\ j\in\{0,1,2,\dotsc\}$.
For this reason, all computations can be done in the same way if we
replace $\ga$ by $W^j\ga\ ,\ j\in\N$. Thus, for
$j\in\{0,1,2,\dotsc\}$ we have \aagf
\cl{\cL_{n_0}(W^j\ga)a(W^j\ga)}{1}\in\ker
\cA_{n_0+1}(W^j\ga).\label{5.59}\zzgf

Using \eqref{5.10} and \eqref{5.58} we infer
$$\cM_{n_0+1}^*(W^j\ga)(\ker \cA_{n_0+1}(W^j\ga))=\ker \cA_{n_0+1}(W^{j+1}\ga)\ ,\ j\in\{0,1,2,\dotsc\}.$$
This means for $j\in\{0,1,2,\dotsc\}$ \aagf
\cM_{n_0+1}^*(W^j\ga)\cl{\cL_{n_0}(W^j\ga)a(W^j\ga)}{1}=k_j\cl{\cL_{n_0}(W^{j+1}\ga)a(W^{j+1}\ga)}{1}\
.\label{5.60}\zzgf Hereby, from \eqref{5.7} we get \aagf
k_j=D_{\ga_{n_0+j+1}}\ ,\ j\in\{0,1,2,\dotsc\}.\label{5.61}\zzgf

Using \eqref{5.10} it follows $\eta_{n_0+1}(W^j\ga)\perp\ker
\cA_{n_0+1}(W^j\ga)$, $j\in\{0,1,2,\dotsc\}$. Therefore, in view
of \eqref{5.8}, the operators $\cM_{n_0+1}^*(W^j\ga)$ and
$\cM_{n_0+1}^{-1}(W^j\ga)$ coincide on the subspace $\ker
\cA_{n_0+1}(W^j\ga)$. Combining this with \eqref{5.6} and
\eqref{5.7} we find for $j\in\{0,1,2,\dotsc\}$ the equations \aagf
&&\kern-2ex\cM_{n_0+1}^*(W^j\ga)\cl{\cL_0(W^j\ga)a(W^j\ga)}{1}
=\cM_{n_0+1}^{-1}(W^j\ga)\cl{\cL_0(W^j\ga)a(W^j\ga)}{1}\nnu\\
&=&\matr{\cM_{n_0}^{-1}(W^j\ga)}{0}{*}{*}\cl{\cL_0(W^j\ga)a(W^j\ga)}{1}
=\cl{\cL_0(W^{j+1}\ga)a(W^j\ga)}{*}.\nnu\zzgf Taking into account
\eqref{5.60} and \eqref{5.61} from this we get
$$a(W^{j+1}\ga)=D_{\ga_{n_0+j+1}}^{-1}a(W^j\ga)\ ,\ j\in\{0,1,2,\dotsc\}.$$
This means \aagf
a(W^j\ga)=\prod\limits_{k=1}^jD_{\ga_{n_0+k}}^{-1}a(\ga)\ ,\
j\in\{0,1,2,\dotsc\}.\label{5.62}\zzgf If we set
$a:=\Pi_{n_0+1}^{-1}a(\ga)$ then \eqref{5.62} implies
$a(W^j\ga)=\Pi_{n_0+j+1}a\ ,\ j\in\{0,1,2,\dotsc\}$. Substituting
this expression into formulas \eqref{5.57} for $W^j\ga$ instead of
$\ga$, \eqref{5.59} and \eqref{5.60} we obtain \eqref{5.51},
\eqref{5.52} and \eqref{5.53}, respectively.

If we assume that $a_1=0$ then representation \eqref{5.3} shows
that the first component of the vector \eqref{5.56}  is $0$ . Then
representation \eqref{5.23} implies that $\ker
\cA_{n_0}\!(W\ga)\ne\!\!0$. This contradiction shows that $a_1\neq
0$. Finally, the uniqueness of the vector $a$ follows from
\eqref{5.57}.
\end{proof}

Before formulating the next result we note that all functions $\La_n(\ga)$, $H_n(\ga)$ and $B_n(\ga)$ only depend on $(\ga_1,\ga_2,\dotsc)$.
This means that the functions $\La_n(W^m\ga)$, $H_n(W^m\ga)$ and $B_n(W^m\ga)$ only depend on $(\ga_{m+1},\ga_{m+2},\dotsc)$.

\begin{thm}[\zitaa{D06}{\cthm{5.19}}]\label{thm5.19}
Assume $\ga\in\Pi\Ga l_2$. Denote by $m_0(\ga)$ the level of the
sequence $\ga$. Then for every $m\ge m_0(\ga)+1$ the element
$\ga_m$ is uniquely determined by the subsequent elements
$\ga_{m+1},\ga_{m+2},\dotsc$. Moreover, the following statements
hold true: \aai\item[(1)] Assume that $\si_n(\ga)>0$ for all
$n\in\N$. Then \aagf&&\ga_m=-\Pi_{m+2}\cdot
\lim\limits_{n\to\infty}\La_n^*(W^m\ga)H_n^{-1}(W^m\ga)B_{n+1}(W^m\ga)\ ,\nnu\\
&&\ \ \ \ \ \ \ \ \ \ \ \ \ \ \ \ \ \ \ \ \ \ \ \ \ \ \ \ \ \ \ \
\ \ \ \ \ \ \ \ \ \ \ \ \ \ \ \ \ \ \ \ m\ge
m_0(\ga)+1\label{5.63}\zzgf where $\Pi_n, \La_n(\ga), H_n(\ga)$
and $B_{n+1}(\ga)$ are defined via \eqref{3.17}, \eqref{5.35},
\eqref{5.45} and \eqref{5.22}, respectively. \item[(2)] Assume that
there exists an $n\in\N$ such that $\si_n(\ga)=0$ is satisfied.
Let $n_0\in\{0,1,2,\dotsc\}$ be chosen such that
$\si_{n_0}(W^{m_0(\ga)}\ga)>0$ and
$\si_{n_0+1}(W^{m_0(\ga)}\ga)=0$. Then there exists a function
$w(\ga)=w(\ga_1,\ga_2,\dotsc)$ such that the identities
\aagf\ga_m=w(W^{m}\ga)\ ,\ m\ge m_0(\ga)+1\label{5.64}\zzgf are
fulfilled. Hereby, we have the following cases: \aai\item[(2a)] If
$n_0=0$ then $w(\ga_1,\ga_2,\dotsc)\equiv0,$ i.e., $\ga_m=0$ for
$m\ge m_0(\ga)+1$. \item[(2b)] If $n_0\in\N$ then \aagf
w(\ga)=-\frac{1}{w_1(\ga)}\sum\limits_{k=1}^{n_0}\ga_{k}w_{k+1}(\ga)\prod\limits_{j=1}^{k}D_{\ga_j}^{-1},\label{5.65}\zzgf
where for $k\in\{1,2,\dotsc,n_0\}$ \aagf \ \ \ \ \ \
w_k(\ga)=\Pi_{n_0+1}\Pi_{k}\sum\limits_{j=1}^k a_jL_{k-j}(W^{j}
\ga){\mbox\ \ and\ \ }w_{n_0+1}(\ga)\equiv 1.\label{5.66}\zzgf
Hereby, the constant vector $a=\col(a_1,a_2,\dotsc,a_{n_0})$
satisfies \eqref{5.51} for $j\ge m_0(\ga)$. \zzi \zzi\end{thm}

\begin{proof}
Without loss of generality we assume that $m_0(\ga)=0$. If
$\si_n(\ga)>0$ for all $n\in\N$ then \lmref{lm5.11} implies that
the expression \eqref{5.47} has to be bounded if $n\to\infty$.
Thus, in the case $m=1$ formula \eqref{5.63} follows from the
boundedness of the expressions \eqref{5.47} and \eqref{5.48}. For
arbitrary $m\ge 2$ formula \eqref{5.63} is verified analogously by
passing from the sequence $\ga$ to the sequence $W^{m-1}\ga$.

Assume now that there exists an $n\in\N$ such that $\si_n(\ga)=0$
is satisfied. Without loss of generality, as above, we assume that
$m_0(\ga)=0$. If $n_0=0$ then $\si_1(\ga)=0$, i.e.,
$1-\prod\limits_{j=1}^\infty(1-|\ga_j|^2)=0$. This implies (2a).

Suppose now that $n_0\in\N$.
Then (see the proof of \thref{thm5.13}) there exists $m\ge0$ and a constant $c>0$ such that the inequality
\aagf
    \matr{\cA_{n_0+1}(W^{m+1}\ga)}{\La_{n_0+1}(W^m\ga)}{\La_{n_0+1}^*(W^m\ga)}{c}
    \ge0\label{5.68}
\zzgf
holds true.
Let
\begin{equation}
    Y(\ga)
    =\cl{\Pi_{n_0+1}\cL_{n_0}(\ga)a}{1},\label{5.69}
\end{equation}
where the vector $a$ satisfies \eqref{5.51}.
Then \eqref{5.52} implies \(Y(W\ga)\in\ker \cA_{n_0+1}(W\ga)\).
From this and \eqref{5.68} for $m=0$ we infer
\aagf
    \La_{n_0+1}^*(\ga)Y(W\ga)
    =0.\label{5.70}
\zzgf
Using \eqref{5.69} and \eqref{5.3} we see that $Y(\ga)$ has the
form
\[
    Y(\ga)
    =\col(w_1(\ga),w_2(\ga),\dotsc,w_{n_0+1}(\ga))
\]
where the sequence $(w_j(\ga))_{j=1}^{n_0+1}$ is defined via
\eqref{5.66}.
Taking into account \eqref{5.35} and substituting the coordinates of $Y(\ga)$ in \eqref{5.70} we obtain the identity \eqref{5.64} for $m=1$.
Hereby, $w(\ga)$ has the form \eqref{5.65}.
Passing now from $\ga$ to $W^{m-1}\ga$ and repeating the above considerations we obtain from \lmref{lm5.18} the formulas \eqref{5.64} for $m\in\{2,3,4,\dotsc\}$.
\end{proof}

The theorems proved above motivate the introduction of the
following notations (see \zitaa{D06}{\cdef{5.20}}):

\begin{defn}\label{def5.20}
The elements $\ga$ of the set $\Pi\Ga l_2$ are called
$\Pi$--sequences. A $\Pi$--sequence $\ga$ is called pure if
$m_0(\ga)=0$. If $\ga,\ga'\in\Ga l_2$ then $\ga'$ is called a
extension of $\ga$ if there exists an $n\in\N$ such that
$W^n\ga'=\ga$ is satisfied. If $\ga$ is a pure $\Pi$--sequence and
$\ga'$ is an extension of $\ga$ then $\ga'$ is called a regular
extension of $\ga$ if $\ga'$ is also a pure $\Pi$--sequence.
\end{defn}

\begin{defn}\label{D5.6-1111}
Let $\gamma=(\gamma_j)_{j=0}^\infty\in\Gamma l_2$.
\begin{itemize}
    \item[(a)] Suppose that there exists some positive integer $n$ such that $\sigma_n(\gamma)=0$. Then the nonnegative integer $n_0$ satisfying $\sigma_{n_0}(\gamma)>0$ and $\sigma_{n_0+1}(\gamma)=0$ is called the \emph{rank of the sequence $\gamma$}. In this case we will write $\rank\gamma=n_0$ to indicate that $\gamma$ has the finite rank $n_0$.
    \item[(b)] If $\sigma_n(\gamma)>0$ for all $n\in\N_0$, then $\gamma$ is called \emph{a sequence of infinite rank}.
\end{itemize}
In the cases (a) and (b), we write $\rank\gamma=n_0$ and $\rank\gamma=\infty$, respectively.
\end{defn}

\begin{rem}\label{R5.7-1111}
Let $\te(\zeta)$ be the Schur function associated with $\ga$ and let $\Dl$ be a simple unitary colligation of type \eqref{3.1} which satisfies $\te(\zeta)=\te_\Dl(\zeta)$.
Then the equations
\[
\rank\ga
=\dim\cH_{\cG\cF}
(=\dim\cH_{\cF\cG})
\]
hold.
\end{rem}

\thref{thm5.19} shows that in the case of a pure $\Pi$--sequence
$\ga=(\ga_j)_{j=0}^\infty$ every element $\ga_n\ ,\ n\in\N, $ is
uniquely determined by the sequence $\ga=(\ga_j)_{j=n+1}^\infty$.
Therefore, every $\Pi$--sequence $\ga$ is a extension of a pure
$\Pi$--sequence $W^{m_0(\ga)}\ga$.

Let us consider an arbitrary $\Pi$--sequence
$\ga=(\ga_j)_{j=0}^\infty$. Then obviously the sequences
$(\ga_j)_{j=1}^\infty$ and $\ga=(\ga_j)_{j=-1}^\infty$ where
$|\ga_{-1}|<1$ are $\Pi$--sequences. This means that as well
deleting an arbitrary finite number of first elements of a
$\Pi$--sequence as finite extension of a $\Pi$--sequence gives us
again a $\Pi$--sequence. However, if $\ga=(\ga_j)_{j=0}^\infty$ is
a $\Pi$--sequence then the freedom of choice is restricted only to
the first $m_0(\ga)+1$ elements $(\ga_j)_{j=0}^{m_0(\ga)}$.
Beginning with the element $\ga_{m_0(\ga)+1}$ all the following
elements of the sequence $\ga$ are uniquely determined by the
corresponding subsequent ones. Namely, the existence of a
determinate chain $(\ga_j)_{m_0(\ga)+1}^\infty$ ensures the
pseudocontinuability of the corresponding function
$\te(\zeta)\in\cS$. Therefore, in order to understand the
phenomenon of pseudocontinuability it will be necessary to study
the structure of pure $\Pi$--sequences. \thref{thm5.19} shows that
a regular extension of a pure $\Pi$--sequence is always unique and
preserves this structure.

Let $\ga$ be a pure $\Pi$--sequence and $\ga'$ one of its
nonregular one-step extensions. Then, as it follows from
\thref{t1.23}, \lmref{lm4.7} and the structure of the kernel of a
Hankel matrix, an arbitrary extension of $\ga'$ can never be a
pure $\Pi$--sequence.

The combination of statement (2) of \thref{thm5.5} and
\thref{thm5.9} shows that a $\Pi$--sequence $\ga$ has finite rank
if and only if its associated function $\te(\zeta)$ is rational.
Hereby, this rank coincides with the smallest number of elementary
$2\times 2$--Blaschke--Potapov factors of type \eqref{4.2A}
occuring in a finite Blaschke--Potapov product which has the block
$\te$.

\lmref{lm5.12} shows that a regular extension of a pure
$\Pi$--sequence of finite rank has the same rank. On the other
hand, the rank of every nonregular one-step extension of a pure
$\Pi$--sequence is one larger. Since every $\Pi$--sequence $\ga$
is a nonregular $m_0(\ga)$-steps extension of a pure
$\Pi$--sequence $V^{m_0(\ga)}\ga$ we have
\aagf\rank\ga=m_0(\ga)+\rank W^{m_0(\ga)}\ga.\label{5.70A}\zzgf
Hereby, $\rank W^{m_0(\ga)}\ga=\rank W^{m_0(\ga)+n}\ga \ , \
n\in\{1,2,3,\dotsc\}.$
\subsection{The structure of pure $\Pi$--sequences of rank $0$ or $1$}

\begin{lem}[\zitaa{D06}{\clem{5.21}}]\label{lm5.21}
Every $\Pi$--sequence $\ga$ of rank $0$ is pure and has the form
\aagf\ga=(\ga_0,0,0,0,\dotsc)\ ,\ |\ga_0|<1.\label{5.71}\zzgf
Conversely, every sequence of type \eqref{5.71} is a pure
$\Pi$--sequence of rank $0$.
\end{lem}

\begin{proof}
Indeed, if $\rank\ga=0$ then $\si_1(\ga)=0,$ i.e.,
$1-\prod\limits_{j=1}^\infty(1-|\ga_j|^2)=0$. This implies
$\ga_j=0\ ,j\in\{1,2,3,\dotsc\}$. The converse statement is
obvious.
\end{proof}

Thus, $\Pi$--sequences of type
$(\ga_0,\ga_1,\dotsc,\ga_n,0,0,\dotsc)\ ,\ |\ga_n|>0\ ,\ n\in\N$
are never pure. They are $n$-step extensions of a pure
$\Pi$--sequence of type \eqref{5.71} where $|\ga_0|>0$. Obviously,
every such sequence has rank $n$.

\begin{thm}[\cite{D3},\zitaa{D06}{\cthm{5.22}}]\label{thm5.22}
A sequence $\ga=(\ga_j)_{j=0}^\infty\in\Ga$ is a pure
$\Pi$--sequence of first rank if and only if $\ga_1\neq 0$ and
there exists a complex number $\la$ such that the conditions \aagf
0<|\la|\le 1-|\ga_1|\label{5.72}\zzgf and
\aagf\ga_{m+1}=\la\frac{\ga_m}{\prod\limits_{j=1}^m(1-|\ga_j|^2)}\
,\ m\in\N\label{5.73}\zzgf are satisfied.
\end{thm}

\begin{proof}
Assume that $\ga$ is a pure $\Pi$--sequence of first rank. Using
\thref{thm5.19} we see that in the case (2b) for $n_0=1$ the
function $w(\ga)$ has the form
$$w(\ga)=-\frac{1}{w_1(\ga)}\ga_1D_{\ga_1}^{-1}w_2(\ga).$$
Hereby, we have $w_1(\ga)=\Pi_2\Pi_1a_1$ and
$w_2(\ga)=w_{n_0+1}(\ga)=1$. Thus,
$w(\ga)=-\frac{\ga_1}{a_1\Pi_1^2}$. From this and \eqref{5.64} we
see that the elements of the sequence $\ga$ are related by the
identities $\ga_m=-\frac{\ga_{m+1}}{a_1\Pi_{m+1}^2}\ ,\ m\ge 1.$
This means
$$\ga_{m+1}=-a_1\Pi_{m+1}^2\ga_m=-a_1\Pi_1^2\frac{\ga_m}{\prod\limits_{j=1}^m(1-|\ga_j|^2)}.$$
Setting $\la:=-a_1\Pi_1^2$ this gives us \eqref{5.73}. Hereby,
because of $a_1\neq 0$ we have $\la\neq 0$. From \eqref{5.73} it
follows $\ga_1\neq 0$ since otherwise we would have that $\ga$ has
rank $0$. Thus, $|\ga_j|>0$ for $j\in\N$.

From \eqref{5.73} we get
$$\frac{|\ga_{m+1}|}{|\ga_m|}=\frac{|\la|}{\prod\limits_{j=1}^m(1-|\ga_j|^2)}\ ,\ m\in\N.$$
Thus,
$\lim\limits_{n\to\infty}\frac{|\ga_{m+1}|}{|\ga_m|}=\frac{|\la|}{\Pi_1^2}$.
In view of $\ga\in\Ga l_2$, this implies
\aagf|\la|\le\Pi_1^2<1.\label{5.75}\zzgf The identities
\eqref{5.73} can be rewritten in the form
\aagf\Pi_1D_{\ga_1}D_{\ga_2}\cdot\ldots\cdot
D_{\ga_m}\ga_{m+1}=\la\ga_m\Pi_{m+1}\ ,\ m\in\N.\label{5.76}\zzgf
Taking into account the equations
$$\sum\limits_{m=1}^\infty D_{\ga_1}^2D_{\ga_2}^2\cdot\ldots\cdot D_{\ga_m}^2|\ga_{m+1}|^2=1-|\ga_1|^2-\Pi_1^2$$
and $\sum\limits_{m=1}^\infty |\ga_m|^2\Pi_{m+1}^2=1-\Pi_1^2, $
from \eqref{5.76} we get
\beql{E5.31-1114}
    \Pi_1^2(1-|\ga_1|^2-\Pi_1^2)
    =|\la|^2(1-\Pi_1^2).
\eeq
Thus, $\Pi_1^2$
is a root of the equation \aagf
x^2-x(1-|\ga_1|^2+|\la|^2)+|\la|^2=0.\label{5.77}\zzgf Hence, this
equation has a root in the interval $(0,1)$. Consequently, taking
into account \eqref{5.75} we obtain \eqref{5.72}.

Conversely, assume that $0<|\ga_1|<1$ and that the conditions
\eqref{5.72} and \eqref{5.73} are satisfied. Then
\aagf|\ga_2|=\frac{|\la|}{1-|\ga_1|}\frac{|\ga_1|}{1+|\ga_1|}\le\frac{|\ga_1|}{1+|\ga_1|}.\label{5.78}\zzgf
The identities \eqref{5.73} can be rewritten for
$m\in\{2,3,4,\dotsc\}$ in the form
\aagf\ga_{m+1}=\la_1\frac{\ga_m}{\prod\limits_{j=2}^m(1-|\ga_j|^2)}\label{5.79}\zzgf
where $\la_1=\frac{\la}{1-|\ga_1|^2}$. From \eqref{5.78} we see
$0<|\ga_2|<1$. Hereby, it can be immediately checked that \aagf
0<|\la_1|\le1-|\ga_2|.\label{5.80}\zzgf Thus, after replacing
$\la$ by $\la_1$ and $\ga_j$ by $\ga_{j+1}, j\in\{1,2,3,\dotsc\}$
the conditions \eqref{5.72} and \eqref{5.73} are still in force and
go over in the conditions \eqref{5.80} and \eqref{5.79}. In
particular, this implies
$$|\ga_3|=\frac{|\la_1|}{1-|\ga_2|}\frac{|\ga_2|}{1+|\ga_2|}\le\frac{|\ga_2|}{1+|\ga_2|}.$$
Applying now the principle of mathematical induction we obtain
$$|\ga_{m+1}|\le\frac{|\ga_m|}{1+|\ga_m|}\ ,\ m\in\N.$$
This implies that the inequalities
\[\begin{split}
    |\ga_{m+1}|
    \le\frac{\frac{|\ga_{m-1}|}{1+|\ga_{m-1}|}}{1+\frac{|\ga_{m-1}|}{1+|\ga_{m-1}|}}
    &=\frac{|\ga_{m-1}|}{1+2|\ga_{m-1}|}\\
    &\le\frac{|\ga_{m-2}|}{1+3|\ga_{m-2}|}
    \le\dotsb
    \le\frac{|\ga_1|}{1+m|\ga_1|} ,\qquad m\in\N,
\end{split}\]
hold true.
Hence, $\ga\in\Ga l_2$.
Hereby, we have $\si_1(\ga)>0$.

Using \eqref{5.73} we find \aagf L_1(\ga_1,\ga_2,\dotsc)
&=&-\sum\limits_{m=1}^\infty\ga_m\ol\ga_{m+1} %
=-\ol\la\sum\limits_{m=1}^\infty\ga_m\frac{\ol\ga_m}{\prod\limits_{j=1}^m(1-|\ga_j|^2)}\nnu\\
&=&-\frac{\ol\la}{\Pi_1^2}\sum\limits_{m=1}^\infty|\ga_m|^2\Pi_{m+1}^2 %
=-\frac{\ol\la}{\Pi_1^2}(1-\Pi_1^2).\label{5.81}\zzgf On the other
hand, rewriting \eqref{5.73} in the form
$$\ga_m=\frac{\Pi_1^2}{\la}\frac{\ga_{m+1}}{\prod\limits_{j=m+1}^\infty(1-|\ga_j|^2)}\ ,\ m\in\N \  $$
we obtain \aagf &&\!\!\!\!L_1(\ga_1,\ga_2,\dotsc)
=-\sum\limits_{m=1}^\infty\ga_m\ol\ga_{m+1} %
=-\frac{\Pi_1^2}{\la}\sum\limits_{m=1}^\infty\frac{\ga_{m+1}}{\prod\limits_{j=m+1}^\infty(1-|\ga_j|^2)}\ol\ga_{m+1}\nnu\\
&=&\!\!\!\!-\frac{\Pi_1^2}{\la\Pi_2^2}(|\ga_2|^2+|\ga_3|^2(1-|\ga_2|^2)
+\dotsb+|\ga_m|^2\prod\limits_{j=1}^{m-1}(1-|\ga_j|^2)+\dotsb)\nnu\\
&=&-\frac{\Pi_1^2}{\la\Pi_2^2}(1-\Pi_2^2).\nnu\zzgf Combining this
with \eqref{5.81} we get
$|L_1(W\ga)|^2=\frac{(1-\Pi_1^2)(1-\Pi_2^2)}{\Pi_2^2}$. Thus,
\aagf\si_2(\ga)&=& \left|\begin{array}{cc}
1-\Pi_1^2 & -\Pi_1\Pi_2\ol{L_1(W\ga)} \\
-\Pi_1\Pi_2L_1(W\ga) & 1-\Pi_2^2(1+|L_1(W\ga)|^2)
\end{array}\right|\nnu\\
&=&(1-\Pi_1^2)(1-\Pi_2^2)-\Pi_2^2|L_1(W\ga)|^2=0.\nnu\zzgf Hence,
the sequence $\ga$ has rank $1$. Since the sequence $\ga$ is not
an extension of a sequence of rank $0$, in view of \eqref{5.70A},
it is pure.
\end{proof}

\begin{cor}[\zitaa{D06}{\ccor{5.23}}]\label{cor5.23}
Let $\ga=(\ga_j)_{j=0}^\infty\in\Ga l_2$. Then it is
\aagf|\sum\limits_{j=1}^\infty\ga_j\ol\ga_{j+1}|\le\frac{(1-\prod\limits_{j=1}^{\infty}(1-|\ga_j|^2))(1-\prod\limits_{j=2}^{\infty}(1-|\ga_j|^2))}{\prod\limits_{j=2}^{\infty}(1-|\ga_j|^2)}.\label{5.81A}\zzgf
Equality \ holds true if and only if there exists a complex number
\  $\la$ such that $0\le|\la|\le 1-|\ga_1|$ and the conditions
\eqref{5.73} are satisfied. In this case we have:
\aai
    \item[\textnormal{(1)}] If $\ga_1=0$ then the sequence $\ga$ is a pure $\Pi$--sequence of
    rank $0$.
    \item[\textnormal{(2)}] If $\ga_1\ne0$ and $\la=0$ then the sequence $\ga$ is a nonregular one-step extension of a pure $\Pi$--sequence of rank $0$.
    \item[\textnormal{(3)}] If $\ga_1\ne0$ and $\la\ne0$ then the sequence $\ga$ is a pure $\Pi$--sequence of rank $1$.
\zzi
\end{cor}

\begin{proof}
The inequality \eqref{5.81A} is equivalent to the condition
$\si_2(\ga)\ge 0$. For this reason equality holds if and only if
$\si_2(\ga)=0$. However, this occurs if and only if we have one of
the three cases mentioned in \coref{cor5.23}.
\end{proof}

As examples we consider the functions $\frac{1+\zeta}{2}$ and
$\frac{1}{2-\zeta}$ which belong to $\cS\Pi\setminus J$. As it was
shown by I. Schur \cite[part II]{Schur}, their Schur parameter
sequences are
$(\frac{1}{2},\frac{2}{3},\frac{2}{5},\frac{2}{7},\dotsc)$ and
$(\frac{1}{2},\frac{1}{3},\frac{1}{4},\frac{1}{5},\dotsc)$,
respectively. We note that both sequences fulfill the conditions
of \thref{thm5.22} with values $\la=\frac{1}{3}$ and
$\frac{2}{3}$, respectively. Thus, both sequences are pure
$\Pi$--sequences of rank $1$. Furthermore, it can be easily
checked that the functions
$$\te(\zeta)=e^{i\si}\frac{w(1+\al)+e^{i\beta}\zeta(1-\al
w)}{(1+\al)-e^{i\beta}\zeta(\al-\ol w)}\ ,\
\si,\beta\in\mathbb{R}\ ,\ w\in\D\ ,\ \al>0 \  $$ belong to
$\cS\Pi\setminus J$ and that their Schur parameter sequence
$(\ga_k)_{k=0}^\infty$ is given by \aagf\ga_0=e^{i\si}w\ ,\
\ga_n=\frac{e^{i(\si+n\beta)}}{\al+n}\ ,\ n\in\N.\label{5.83}\zzgf
Using the identity
$1-|\ga_n|^2=\frac{(\al+n-1)(\al+n+1)}{(\al+n)^2}$ \ it can be
checked by straightforward computations that the sequence
\eqref{5.83} also satisfies the conditions of \thref{thm5.22} with
$\la=e^{i\beta}\frac{\al}{\al+1}$. Hence, the sequence \eqref{5.83}
is a pure $\Pi$--sequence of rank $1$, too.

\section{The $\cS$--recurrence property of the Schur parameter sequences associated with non--inner rational Schur functions}\label{sec6}
In this section we follow \cite[\csec{2}]{DFK}.

Let $\theta\in\cS$ and let
\begin{equation}\label{thetaTaylor}
\theta(\zeta)=\sum_{j=0}^\infty c_j\zeta^j,\ \ \zeta\in\D,
\end{equation}
be the Taylor series representation of $\theta$. Moreover, let $(\gamma_j)_{j=0}^w$ be the Schur parameter sequence associated with $\theta$.
As it was shown by I. Schur  \cite[part I]{Schur}, for each integer $n$ satisfying $0\leq n<w$, the identities
\begin{equation}\label{gammaN}
\gamma_n=\Phi_n(c_0,c_1,\dotsc,c_n)
\end{equation}
and
\begin{equation}\label{cN}
c_n=\Psi_n(\gamma_0,\gamma_1,\dotsc,\gamma_n)
\end{equation}
hold true.
Here, I. Schur presented an explicit description of the function $\Phi_n$.
For the function $\Psi_n$, he obtained the formula
\begin{equation}\label{cN+1}
    \Psi_n(\gamma_0,\gamma_1,\dotsc,\gamma_n)
    =\gamma_n\cdot\prod_{j=0}^{n-1}\left(1-|\gamma_j|^2\right)+\widetilde{\Psi}_{n-1}(\gamma_0,\dotsc,\gamma_{n-1})
\end{equation}
where $\widetilde{\Psi}_{n-1}$ is a polynomial of the variables $\gamma_0,\overline{\gamma_0},\dotsc,\gamma_{n-1},\overline{\gamma_{n-1}}$.

It should be mentioned that the explicit form of the functions $\widetilde{\Psi}_{n-1}$ was described in \cite{D98}.
Thus, for every integer $n$ satisfying $0\leq n<w$, the sequences $(c_k)_{k=0}^n$ and $(\gamma_k)_{k=0}^n$ can each be expressed in terms of the other.

It is known (see, e.g., Proposition 1.1 in \cite{Ber}) that the power series
\begin{equation}\label{sumcj}
\sum_{j=0}^\infty c_jz^j
\end{equation}
can be written as a quotient $\frac{P}{Q}$ of two polynomials $P$ and $Q$ where $Q(z)=1-q_1z-\dotsb -q_rz^r$ if and only if there exists some $m\in\N_0$ such that for each integer $n$ with $n\geq m$ the relation
\begin{equation}\label{cn+1}
c_{n+1}=q_1c_n+q_2c_{n-1}+\dotsb+q_rc_{n-r+1}
\end{equation}
holds true. In this case the sequence $c=(c_j)_{j=0}^\infty$ is said to be a recurrent sequence of $r$--th order and formula (\ref{cn+1}) is called a recurrence formula of order $r$.\\ 
From (\ref{cn+1}), (\ref{cN}) and (\ref{cN+1}) it follows that the rationality of a function $\theta\in\cS$ can be characterized by relations of the form
\begin{equation}\label{gammaN+1}
\gamma_{n+1}=g_n(\gamma_0,\gamma_1,\dotsc,\gamma_n),\ \ n\geq n_0,
\end{equation}
where $(g_n)_{n\geq n_0}$ is some sequence of given functions. It should be mentioned that the functions $(g_n)_{n\geq n_0}$ obtained in this way do not have such an explicit structure which enables us to perform a detailed analysis of the Schur parameter sequences of functions belonging to the class $\cR\cS\backslash J$.\\
The main goal of this subsection is to present a direct derivation of the relations (\ref{gammaN+1}) and, in so doing, characterize the Schur parameter sequences associated with functions from $\cR\cS\backslash J$.\\
We rewrite equation (\ref{cn+1}) in a different way. Here we consider the vectors
\begin{equation}\label{q}
q:=(-q_r,-q_{r-1},\dotsc,-q_1,1)^T
\end{equation}
and
\begin{equation}\label{mu}
\mu_{r+1}(c):=(\overline{ c_0},\overline{ c_1},\dotsc,\overline{ c_{r-1}},\overline{ c_r})^T.
\end{equation}
For each $n\in\{m,m+1,\dotsc\}$ we have then
$$\mu_{r+1}(W^{n-r+1}c)=(\overline{ c_{n-r+1}},\overline{ c_{n-r+2}},\dotsc,\overline{ c_{n}},\overline{ c_{n+1}})^T,$$
where $W$ is the coshift given by \eqref{3.14}. Thus, for each integer $n$ with $n\geq m$, the recursion formula (\ref{cn+1}) can be rewritten as an orthogonality condition in the form
\begin{equation}\label{recform}
\left(q,\mu_{r+1}(W^{n-r+1}c)\right)_{\C^{r+1}}=0,
\end{equation}
where $(\cdot,\cdot)_{\C^{r+1}}$ denotes the usual Euclidean inner product in the space $\C^{r+1}$ (i.e., $(x,y)_{\C^{r+1}}=y^*x$ for all $x,y\in\C^{r+1}$).\\
Let the series (\ref{sumcj}) be the Taylor series of a function $\theta\in\cR\cS\backslash J$ and let $\gamma=(\gamma_j)_{j=0}^\infty$ be the sequence of Schur parameters associated with $\theta$. Then it will turn out that the recurrence property of the Taylor coefficient sequence $(c_j)_{j=0}^\infty$ implies some type of recurrence relations for the sequence $\gamma=(\gamma_j)_{j=0}^\infty$. With this in mind we introduce the following notion.

\begin{defn}[\zitaa{DFK}{Definition~2.1}]\label{defn21}
Let $\gamma=(\gamma_j)_{j=0}^\infty\in\Gamma$. Then the sequence $\gamma$ is called \emph{$\cS$-recurrent} if there exist $r\in\N$
and some vector $p=(p_r,p_{r-1},\dotsc,p_0)^T\in\C^{r+1}$ with $p_0\neq 0$ such that for all integers $n$ with $n\geq r$ the relations
\begin{equation}\label{Srecurrent}
\left(p,\left[\stackrel{\longrightarrow}{\prod_{k=0}^{n-r-1}}\fM_{r+1}(W^k\gamma)\right]\eta_{r+1}(W^{n-r}\gamma)\right)_{\C^{r+1}}=0
\end{equation}
are satisfied, where the matrix $\fM_{r+1}(\gamma)$ and the vector $\eta_{r+1}(\gamma)$ are defined via \eqref{5.7} and \eqref{5.9}, respectively. In this case the vector $p$ is called an \emph{$r$--th order $S$--recurrence vector associated with $\gamma$}.
\end{defn}

\begin{rem}[\zitaa{DFK}{\crem{2.2}}]\label{rem33}
If we compare the vectors $\mu_{r+1}(c)$ and $\eta_{r+1}(\gamma)$ introduced in \eqref{mu} and \eqref{5.9}, respectively, then we see that the numbers $\overline{ \gamma_k}$ in the vector $\eta_{r+1}(\gamma)$ are multiplied with the factor $\prod_{j=1}^{k-1}D_{\gamma_j}$ which can be thought of as a weight factor. Moreover, contrary to \eqref{recform}, the vector $\eta_{r+1}(W^{n-r}\gamma)$ is paired in \eqref{Srecurrent} with the matrix product
$$\stackrel{\longrightarrow}{\prod_{k=0}^{n-r-1}}\fM_{r+1}(W^k\gamma).$$
In the case $n=r$ the latter product has to be interpreted as the unit matrix $I_{r+1}$.
\end{rem}

The following result plays an important role in our subsequent considerations.
\begin{lem}[\zitaa{DFK}{\clem{2.3}}]\label{lem23}
Let $\gamma=(\gamma_j)_{j=0}^\infty\in\Gamma l_2$ and let $n\in\N$. Then $\cA_n(\gamma)$ defined via \eqref{5.4} can be represented via
\begin{equation}\label{Anxi}
\cA_n(\gamma)=\sum_{j=0}^\infty \xi_{n,j}(\gamma)\xi_{n,j}^*(\gamma),
\end{equation}
where
\begin{equation}\label{xi}
\xi_{n,j}(\gamma):=\left[\stackrel{\longrightarrow}{\prod_{k=0}^{j-1}}\fM_n(W^k\gamma)\right]\eta_n(W^j\gamma),\ \ j\in\N_0.
\end{equation}
\end{lem}
\begin{proof}
Applying \eqref{5.10} to $W\gamma$ instead of $\gamma$ we obtain
$$\cA_n(W\gamma)=\eta_n(W\gamma)\eta_n^*(W\gamma)+\fM_n(W\gamma)\cA_n(W^2\gamma)\fM_n^*(W\gamma).$$
Inserting this expression into \eqref{5.10} we get
\begin{eqnarray*}
\cA_n(\gamma)&=&\eta_n(\gamma)\eta_n^*(\gamma)+\fM_n(\gamma)\eta_n(W\gamma)\eta_n^*(W\gamma)\fM_n^*(\gamma)\\
&&+\fM_n(\gamma)\fM_n(W\gamma)\cA_n(W^2\gamma)\fM_n^*(W\gamma)\fM_n^*(\gamma).
\end{eqnarray*}
This procedure will be continued now.
Taking into account the contractivity of the matrices $\fM_n(W^j\gamma)$, $j\in\N_0$, and the limit relation $\lim_{m\rightarrow\infty}\cA_n(W^m\gamma)=0_{(n+1)\times(n+1)}$, which follows from \eqref{5.4}, we obtain (\ref{Anxi}).
\end{proof}
Let $r\in\N$. Using (\ref{xi}) one can see that condition \eqref{Srecurrent}, which expresses $\cS$--recurrence of $r$--th order, can be rewritten in the form
\begin{equation}\label{Srecurrent2}
\left(p,\xi_{r+1,j}(\gamma)\right)_{\C^{r+1}}=0,\ \ j\in\N_0.
\end{equation}
Thus the application of Lemma \ref{lem23} leads us immediately to the following result.
\begin{prop}[\zitaa{DFK}{Proposition~2.4}]\label{prop24}
Let $\gamma=(\gamma_j)_{j=0}^\infty\in\Gamma l_2$. Further, let $r\in\N$ and let $p=(p_r,\dotsc,p_0)^T$ $\in\C^{r+1}$. Then $p$ is an $r$--th order $\cS$--recurrence vector associated with $\gamma$ if and only if $p_0\neq 0$ and
\begin{equation}\label{pKern}
p\in\ker\cA_{r+1}(\gamma).
\end{equation}
\end{prop}
Now we are able to prove one of the main results of this section.
It establishes an important connection between the $\cS$--recurrence property of a sequence $\gamma\in\Gamma l_2$ and the rationality of the Schur function $\theta$, the Schur parameter sequence of which is $\gamma$.
\begin{thm}[\zitaa{DFK}{\cthm{2.5}}]\label{thm25}
Let $\gamma=(\gamma_j)_{j=0}^\infty\in\Gamma$ and let $\theta$ be the Schur function with Schur parameter sequence $\gamma$. Then $\theta\in\cR\cS\backslash J$ if and only if $\gamma$ is an $\cS$--recurrent sequence belonging to $\Gamma l_2$.
\end{thm}
\begin{proof}
    From Theorems~\ref{thm4.2} and~\ref{thm5.9} it follows that $\theta\in\cR\cS\backslash J$ if and only if $\gamma$ belongs to $\Gamma l_2$ and there exists some $r\in\N$ such that $\sigma_{r+1}(\gamma)=0$. In this case, we infer from part~(2) of Theorem~\ref{thm5.5} that there exists an $n_0\in\N_0$ such that $\sigma_{n_0}(\gamma)>0$ and $\sigma_{n_0+1}(\gamma)=0$. If $n_0=0$, then
$$0=\sigma_1(\gamma)=1-\prod_{j=1}^\infty(1-|\gamma_j|^2).$$
Thus, $\gamma_j=0$ for all $j\in\N$. This implies that $\theta$ is the constant function in $\D$ with value $\gamma_0\in\D$. If $n_0\in\N$, then we have $\det\cA_{n_0}(\gamma)>0$ and $\det\cA_{n_0+1}(\gamma)=0$. The condition $\det\cA_{n_0+1}(\gamma)=0$ is equivalent to the existence of a nontrivial vector $p=(p_{n_0},\dotsc,p_0)^T\in\C^{n_0+1}$ which satisfies
\begin{equation}\label{pKernN0}
p\in\ker\cA_{n_0+1}(\gamma).
\end{equation}
Equation \eqref{5.15} for $n=n_0$ has the form
\begin{equation}\label{BlockAn1}
\cA_{n_0+1}(\gamma)=\begin{pmatrix}
                \cA_{n_0}(\gamma) & -\fL_{n_0}(\gamma)b_{n_0}(\gamma)\\
                - b_{n_0}^*(\gamma)\fL_{n_0}^*(\gamma) & 1-\Pi_{n_0+1}^2-b_{n_0}^*(\gamma)b_{n_0}(\gamma)
                \end{pmatrix}.
\end{equation}
From (\ref{BlockAn1}) it follows that $p_0\neq 0$. Indeed, if we would have $p_0=0$, then from (\ref{BlockAn1}) we could infer $(p_{n_0},\dotsc,p_1)^T\in\ker\cA_{n_0}(\gamma)$. In view of $\det\cA_{n_0}(\gamma)\neq 0$, this implies $(p_{n_0},\dotsc,p_1)^T=0_{n_0\times 1}$ which is a contradiction to the choice of $p$. Now the asserted equivalence follows immediately from Proposition \ref{prop24}.
\end{proof}

\begin{prop}[\zitaa{DFK}{Proposition~2.6}]\label{prop26}
Let $n_0\in\N$, and let $\gamma=(\gamma_j)_{j=0}^\infty$ be a sequence which belongs to $\Gamma l_2$ and satisfies $\rank\gamma=n_0$.
Then:
\begin{itemize}
    \item[\textnormal{(a)}] The sequence $\gamma$ is $\cS$--recurrent and $n_0$ is the minimal order of $\cS$--recurrence vector associated with $\gamma$. There is a unique $n_0$--th order $\cS$--recurrence vector $p=(p_{n_0},\dotsc,p_0)^T$ of $\gamma$ which satisfies $p_0=1$.
    \item[\textnormal{(b)}] Let $r$ be an integer with $r\geq n_0$ and let $p$ be an $n_0$--th order $\cS$--recurrence vector associated with $\gamma$.
    \begin{itemize}
        \item[\textnormal{(b1)}] Let the sequence $(\widetilde g_j)_{j=1}^{r-n_0+1}$ of vectors from $\C^{n_0+1}$ be defined by
        \begin{multline}\label{wtgj}
        \widetilde g_1:=p,\;\widetilde g_2:=\fM_{n_0+1}^*(\gamma)p ,\dotsc\\
        \dotsc,\widetilde g_{r-n_0+1}:=\left[\stackrel{\longleftarrow}{\prod_{k=0}^{r-n_0+1}}\fM_{n_0+1}^*(W^k\gamma)\right]p.
        \end{multline}
        Then the $\C^{r+1}$--vectors
        \begin{equation}\label{gj}
        g_1:=\left(\begin{matrix}\widetilde g_1\\0\\\vdots\\0\end{matrix}\right),\ \ g_2:=\left(\begin{matrix}0\\\widetilde g_2\\0\\\vdots\\0\end{matrix}\right),\ \ \dotsc, \ \ g_{r-n_0+1}:=\left(\begin{matrix}0\\\vdots\\0\\\widetilde g_{r-n_0+1}\end{matrix}\right)
        \end{equation}
        form a basis of $\ker\cA_{r+1}(\gamma)$.
        \item[\textnormal{(b2)}] The sequence $\gamma$ has $\cS$--recurrence vector of $r$--th order and every such vector $\widehat p$ has the shape
        \begin{equation}\label{alphag}
        \widehat p=\alpha_1 g_1+\alpha_2 g_2+\dotsb+\alpha_{r-n_0+1}g_{r-n_0+1},
        \end{equation}
        where $(\alpha_j)_{j=1}^{r-n_0+1}$ is a sequence of complex numbers satisfying $\alpha_{r-n_0+1}\neq 0$.
    \end{itemize}
\end{itemize}
\end{prop}
\begin{proof}
(a) From Definition \ref{D5.6-1111} the relation
\begin{equation}\label{n0}
n_0=\min\{r\in\N_0:\ker\cA_{r+1}(\gamma)\neq \{0_{(r+1)\times 1}\}\}
\end{equation}
follows. The block decomposition (\ref{BlockAn1}) shows that
\begin{equation}\label{dimKern}
\dim[\ker\cA_{n_0+1}(\gamma)]=1.
\end{equation}
Let $p=(p_{n_0},\dotsc,p_0)^T\in\ker\cA_{n_0+1}(\gamma)\backslash\{0_{(n_0+1)\times 1}\}$. As in the proof of Theorem \ref{thm25} it can be shown then that $p_0\neq 0$. Thus, Proposition \ref{prop24} yields that $p$ is an $n_0$--th order $\cS$--recurrence vector associated with $\gamma$. Taking into account (\ref{dimKern}) and applying again Proposition \ref{prop24}, we see that there is a unique $n_0$--th order $\cS$--recurrence vector $p=(p_{n_0},\dotsc,p_0)^T$ associated with $\gamma$ which satisfies $p_0=1$. In particular, $\gamma$ is $\cS$--recurrent. In view of (\ref{n0}), applying Proposition \ref{prop24} we see that $n_0$ is the minimal order of $\cS$--recurrence
vector associated with $\gamma$.\\
(b1) In the case $r=n_0$ the assertion is already proved above.
Let $r=n_0+1$.
Using \eqref{5.21} and \eqref{5.4}, we obtain the block decomposition
\begin{equation}\label{BlockAn2}
\cA_{r+1}(\gamma)=\begin{pmatrix}
                1-\Pi_{1}^2 & -\Pi_1B_{r+1}^*(\gamma)\\
                -B_{r+1}^*(\gamma)\Pi_1 & \cA_r(W\gamma)-B_{r+1}(\gamma)B_{r+1}^*(\gamma)
                \end{pmatrix}.
\end{equation}
In view of $p\in\ker\cA_{n_0+1}(\gamma)$ the block decomposition (\ref{BlockAn1}) with $n=n_0+1$ implies that $g_1=\begin{pmatrix}p\\0\end{pmatrix}$ belongs to $\ker\cA_{n_0+2}(\gamma)$. Furthermore, using \eqref{5.10} with $n=n_0$, we see that
$$\fM_{n_0+1}^*(\gamma)p\in\ker\cA_{n_0+1}(W\gamma).$$
Now the block decomposition (\ref{BlockAn2}) implies that the vector $g_2=\begin{pmatrix}0\\\widetilde{g_2}\end{pmatrix}$ also belongs to $\ker\cA_{n_0+2}(\gamma)$. In view of $p_0\neq 0$ and the triangular shape of the matrix $\fM_{n_0+1}^*(\gamma)$ (see \eqref{5.7}), we see that the last component of the vector $g_2$ does not vanish. Thus, the vectors $g_1$ and $g_2$ are linearly independent vectors belonging to $\ker\cA_{n_0+2}(\gamma)$. Since part (b) of Theorem~\ref{thm5.5} implies that $\dim[\ker\cA_{n_0+2}(\gamma)]=2$, we obtain that $g_1$ and $g_2$ form a basis of $\ker\cA_{n_0+2}(\gamma)$. One can prove the assertion by induction for arbitrary $r\in\{n_0,n_0+1,\dotsc\}$.\\
(b2) This follows immediately by combining Proposition \ref{prop24} with part (b1).
\end{proof}
Proposition \ref{prop26} leads us to the following notion.
\begin{defn}[\zitaa{DFK}{Definition~2.7}]\label{defn27}
Let $n_0\in\N$ and let $\gamma=(\gamma_j)_{j=0}^\infty$ be a sequence which belongs to $\Gamma l_2$ and satisfies $\rank\gamma=n_0$. Then the unique $n_0$--th order $\cS$--recurrence vector $p=(p_{n_0},\dotsc,p_0)^T$ satisfying $p_0=1$ is called \emph{the basic $\cS$--recurrence vector associated with $\gamma$}.
\end{defn}
Let $\gamma=(\gamma_j)_{j=0}^\infty\in\Gamma l_2$ be an $\cS$--recurrent sequence and let $p$ be the basic $\cS$--recurrence vector associated with $\gamma$. Then Proposition \ref{prop26} shows that all $\cS$--recurrence vectors associated with $\gamma$ can be obtained from $p$.\\
Our next consideration is aimed at working out the recurrent character of formula \eqref{Srecurrent}. More precisely, we will verify that, for each integer $n$ with $n\geq r$, the element $\gamma_{n+1}$ can be expressed in terms of the preceding members $\gamma_0,\dotsc,\gamma_n$ of the sequence $\gamma$. In view of Proposition \ref{prop26}, this is the content of the following result.
\begin{thm}[\zitaa{DFK}{\cthm{2.8}}]\label{thm28}
Let $\gamma=(\gamma_j)_{j=0}^\infty$ be a $\cS$--recurrent sequence which belongs to $\Gamma l_2$ and
let $p=(p_r,p_{r-1},\dotsc,p_0)^T$ be an $r$--th order $\cS$--recurrence vector associated with $\gamma$. Further, let
\begin{equation}\label{lambda}
\lambda:=-\frac{1}{p_0}\left[\prod_{k=1}^rD_{\gamma_k}\right]\cdot(p_r,p_{r-1},\dotsc,p_1)^T
\end{equation}
Then for every integer $n$ with $n\geq r$ the relations
\begin{multline}\label{DarGammaN+1}
    \gamma_{n+1}
    =\left[\prod_{s=1}^nD_{\gamma_s}^{-1}\right]\\
    \times\left(\left[\stackrel{\longleftarrow}{\prod_{j=0}^{n-r-1}}\fM_r^{-1}(W^j\gamma)\right]\lambda,\left[\prod_{k=n-r+1}^nD_{\gamma_k}^{-1}\right]\eta_r(W^{n-r}\gamma)\right)_{\C^r}
\end{multline}
hold where $D_{\gamma_j}$, $W$, $\fM_r(\gamma)$, and $\eta_r(\gamma)$ are defined via \eqref{E2.19}, \eqref{3.14}, \eqref{5.7}, and \eqref{5.9}, respectively.
\end{thm}
\begin{proof}
Since $p$ is an $r$--th order $\cS$--recurrence vector associated with $\gamma$, the relation \eqref{Srecurrent} is satisfied.
From Definition \ref{defn21} it follows that $p_0\neq0$. We rewrite \eqref{Srecurrent} in the form
\begin{equation}\label{Srecurrent3}
\left(\left[\stackrel{\longleftarrow}{\prod_{k=0}^{n-r-1}}\fM_{r+1}^*(W^k\gamma)\right]p\ \ ,\ \ \eta_{r+1}(W^{n-r}\gamma)\right)_{\C^{r+1}}=0.
\end{equation}
In view of Proposition \ref{prop24}, we have $p\in\ker\cA_{r+1}(\gamma)$. Applying \eqref{5.10} for $n=r+1$, we obtain
\begin{equation}\label{thm28_22}
\left(p,\eta_{r+1}(\gamma)\right)_{C^{r+1}}=0
\end{equation}
and
\begin{equation}\label{thm28_23}
\fM_{r+1}^*(\gamma)\cdot p\in\ker\cA_{r+1}(W\gamma).
\end{equation}
Using \eqref{5.8} for $n=r+1$, we see that, for all $x\in\C^{r+1}$ which are orthogonal to $\eta_{r+1}(\gamma)$, the identity $\fM_{r+1}^*(\gamma)x=\fM_{r+1}^{-1}(\gamma)x$ holds true. Thus, from (\ref{thm28_22}) we infer
\begin{equation}\label{thm28_24}
\fM_{r+1}^*(\gamma)p=\fM_{r+1}^{-1}(\gamma)p.
\end{equation}
Bearing (\ref{thm28_23}), (\ref{thm28_24}), and Lemma~\ref{lm5.3} in mind, replacing $p$ in these considerations with $\fM_{r+1}^*(\gamma)p$, we obtain
$$\fM_{r+1}^*(W\gamma)\fM_{r+1}^*(\gamma) p\in\ker\cA_{r+1}(W^2\gamma)$$
and
$$\fM_{r+1}^*(W\gamma)\fM_{r+1}^*(\gamma)p=\fM_{r+1}^{-1}(W\gamma)\fM_{r+1}^{-1}(\gamma)p.$$
Thus, by induction we get
\begin{equation}\label{thm28_25}
\left[\stackrel{\longleftarrow}{\prod_{k=0}^{n-r-1}}\fM_{r+1}^*(W^k\gamma)\right]p=\left[\stackrel{\longleftarrow}{\prod_{k=0}^{n-r-1}}\fM_{r+1}^{-1}(W^k\gamma)\right]p.
\end{equation}
From \eqref{lambda} we see that the vector $p$ can be written in the form
\begin{equation}\label{thm28_26}
p=p_0\begin{pmatrix}-\left[\prod_{s=1}^rD_{\gamma_s}^{-1}\right]\lambda\\ 1\end{pmatrix}.
\end{equation}
From \eqref{5.7} we infer that the matrix $\fM_{r+1}(\gamma)$ has the block decomposition
\begin{equation}\label{thm28_27}
\fM_{r+1}(\gamma)=\begin{pmatrix}\fM_r(\gamma) & 0_{r\times 1}\\ \star&D_{\gamma_{r+1}}\end{pmatrix}.
\end{equation}
Formula (\ref{thm28_27}) implies the block representation
\begin{equation}\label{thm28_28}
\fM_{r+1}(\gamma)=\begin{pmatrix}\fM_r^{-1}(\gamma) & 0_{r\times 1}\\ \star&D_{\gamma_{r+1}}^{-1}\end{pmatrix}.
\end{equation}
Combining (\ref{thm28_28}) and (\ref{thm28_26}), we conclude that the right--hand side of (\ref{thm28_25}) can be rewritten in the form
\begin{equation}\label{thm28_29}
\left[\stackrel{\longleftarrow}{\prod_{k=0}^{n-r-1}}\fM_{r+1}^{-1}(W^k\gamma)\right]p=p_0\cdot\begin{pmatrix}-\left[\prod_{s=1}^rD_{\gamma_s}^{-1}\right]\left[\stackrel{\longleftarrow}{\prod_{k=0}^{n-r-1}}\fM_r^{-1}(W^k\gamma)\right]\lambda\\k_{n-r}\end{pmatrix}
\end{equation}
where $k_{n-r}$ is some complex number. On the other hand, taking into account (\ref{thm28_27}) and (\ref{thm28_26}), we find that the left--hand side of (\ref{thm28_25}) can be expressed by
\begin{equation}\label{thm28_30}
\left[\stackrel{\longleftarrow}{\prod_{k=0}^{n-r-1}}\fM_{r+1}^*(W^k\gamma)\right]p=p_0\cdot\begin{pmatrix}\star\\\prod_{k=0}^{n-r-1}D_{\gamma_{r+1+k}}\end{pmatrix}.
\end{equation}
The combination of (\ref{thm28_25}), (\ref{thm28_29}), and (\ref{thm28_30}) yields
\begin{equation}\label{thm28_31}
k_{n-r}=\prod_{k=0}^{n-r-1}D_{\gamma_{r+1+k}}.
\end{equation}
Combining (\ref{thm28_25}), (\ref{thm28_29}), and (\ref{thm28_31}) yields
\begin{equation}\label{thm28_32}
\left[\stackrel{\longleftarrow}{\prod_{k=0}^{n-r-1}}\fM_{r+1}^*(W^k\gamma)\right]p=p_0\cdot\begin{pmatrix}-\left[\prod_{s=1}^rD_{\gamma_s}^{-1}\right]\left[\stackrel{\longleftarrow}{\prod_{k=0}^{n-r-1}}\fM_r^{-1}(W^k\gamma)\right]\lambda\\\prod_{k=0}^{n-r-1}D_{\gamma_{r+1+k}}\end{pmatrix}.
\end{equation}
From \eqref{5.9} we get
$$\eta_{r+1}(\gamma)=\begin{pmatrix}\eta_r(\gamma)\\\overline{\gamma_{r+1}}\prod_{k=1}^{r}D_{\gamma_k}\end{pmatrix}.$$
Consequently,
\begin{equation}\label{thm28_33}
\eta_{r+1}(W^{n-r}\gamma)=
\begin{pmatrix}\eta_r(W^{n-r}\gamma)\\\overline{\gamma_{n+1}}\prod_{k=1}^{r}D_{\gamma_{k+n-r}}\end{pmatrix}.
\end{equation}
Using (\ref{thm28_32}) and (\ref{thm28_33}), we infer
\begin{eqnarray}\label{thm28_34}
&&\hspace{-1cm}\left(\left[\stackrel{\longleftarrow}{\prod_{k=0}^{n-r-1}}\fM_{r+1}^*(W^k\gamma)\right]p\ \ ,\ \ \eta_{r+1}(W^{n-r}\gamma)\right)_{\C^{r+1}}\nonumber\\
&=&p_0\cdot\left[\left(-\left[\prod_{s=1}^rD_{\gamma_s}^{-1}\right]\cdot\left[\stackrel{\longleftarrow}{\prod_{j=0}^{n-r-1}}\fM_r^{-1}(W^j\gamma)\right]\lambda\ \ ,\ \ \eta_r(W^{n-r}\gamma)\right)_{\C^r}\right.\nonumber\\
&&\hspace{1cm}\left.+\left(\prod_{k=0}^{n-r-1}D_{\gamma_{r+1+k}}\right)\gamma_{n+1}\left(\prod_{k=1}^rD_{\gamma_{k+n-r}}\right)\right].
\end{eqnarray}
Taking into account (\ref{thm28_31}), (\ref{thm28_34}), and $p_0\neq0$, a straightforward computation yields (\ref{DarGammaN+1}). Thus, the proof is complete.
\end{proof}

\begin{rem}[\zitaa{DFK}{\crem{2.9}}]\label{rem29}
Let $\gamma=(\gamma_j)_{j=0}^\infty$ be a sequence which satisfies the assumptions of Theorem \ref{thm28}. Then it is possible to recover the whole sequence $\gamma$ via the formulas \eqref{DarGammaN+1} from the section $(\gamma_j)_{j=0}^r$ and the vector $\lambda=(\lambda_1,\dotsc,\lambda_r)^T$. Indeed, for $n=r$ we have
\[\begin{split}
    \gamma_{r+1}
    &=\left[\prod_{k=1}^rD_{\gamma_k}^{-2}\right]\eta_r^*(\gamma)\lambda\\
    &=\frac{1}{(1-|\gamma_1|^2)\dotsm(1-|\gamma_r|^2)}\cdot\sum_{k=1}^r\lambda_{r-k+1}\gamma_k\left(\prod_{j=1}^kD_{\gamma_j}\right).    
\end{split}\]
In the case $n=r+2$ the vector $\lambda$ has to be replaced by $\fM_r^{-1}(W\gamma)$ and the sequence $(\gamma_j)_{j=0}^r$ has to be replaced by $(\gamma_{j+1})_{j=0}^r$. The matrix $\fM_r(\gamma)$ depends on the section $(\gamma_j)_{j=1}^r$. Thus, the matrix $\fM_r^{-1}(W\gamma)$ depends on the section $(\gamma_{j+1})_{j=1}^r$. Consequently, formula \eqref{DarGammaN+1} yields an expression for $\gamma_{r+2}$ in terms of the sequence $(\gamma_j)_{j=1}^{r+1}$. Continuing this procedure inductively we see that, for all integers $n$ with $n\geq r$, formula \eqref{DarGammaN+1} produces an expression for $\gamma_{n+1}$ which is based on the section $(\gamma_j)_{j=0}^n$. Consequently, the sequence $(\gamma_j)_{j=0}^\infty$ is completely determined by the section $(\gamma_j)_{j=0}^r$ and the vector $\lambda$. It should be mentioned that in the case $n_0=1$, which corresponds to a sequence of rank 1, for each $n\in\N$ formula \eqref{DarGammaN+1} has the form
$$\gamma_{n+1}=\lambda\cdot\frac{\gamma_n}{\prod_{j=1}^n(1-|\gamma_j|^2)}.$$
Observe that, for this particular case $n_0=1$, it was derived in \rthm{thm5.22} (see \eqref{5.73}).
\end{rem}
Our next goal can be described as follows. Let $n_0\in\N$ and let $\gamma=(\gamma_j)_{j=0}^\infty\in\Gamma l_2$ be a sequence which satisfies $\rank\gamma=n_0$. Furthermore, let $r$ be an integer with $r\geq n_0$ and let $p$ be an $r$-th $\cS$--recurrence vector associated with $\gamma$. Then we will show that the identity
$$\prod_{k=1}^r(1-|\gamma_k|^2)-\Pi_1^2=\lambda^*\left( \ (\fL_r^{-1}(\gamma))^*\fL_r^{-1}(\gamma)-I_r\right)\lambda$$
holds, where $\Pi_1$, $\lambda$ and $\fL_r(\gamma)$ are defined via \eqref{3.17}, \eqref{lambda}, and \eqref{5.3}, respectively. To accomplish this we still need to make some preparations. In this way, we will be led to several results that are, by themselves, of interest.\\
Let $n\in\N$. Then the symbol $\left\|. \right\|_{\C^n}$ stands for the Euclidean norm in the space $\C^n$.
\begin{lem}[\zitaa{DFK}{\clem{2.10}}]\label{lem210}
Let $\gamma=(\gamma_j)_{j=0}^\infty$ be a sequence from $\D$. Let $n\in\N$ and let $\eta_n(\gamma)$ be defined via \eqref{5.9}. Then
$$1-\left\|\eta_n(\gamma)\right\|_{\C^n}^2=\prod_{j=1}^nD_{\gamma_j}^2.$$
\end{lem}
\begin{proof}
For $n=1$ the asserted equation obviously holds. Now let $n\ge 2$. Then from \eqref{5.9} we see the block decomposition
$$\eta_n(\gamma)=\begin{pmatrix}\eta_{n-1}(\gamma)\\\overline{\gamma_n}\left[\prod_{k=1}^{n-1}D_{\gamma_k}\right]\end{pmatrix}.$$
Thus, taking into account the definition of the Euclidean norm, we get
$$\left\|\eta_n(\gamma)\right\|_{\C^n}^2=\left\|\eta_{n-1}(\gamma)\right\|_{\C^{n-1}}^2+|\gamma_n|^2\left[\prod_{k=1}^{n-1}D_{\gamma_k}^2\right].$$
Now, the assertion follows immediately by induction.
\end{proof}
\begin{lem}[\zitaa{DFK}{\clem{2.11}}]\label{lem211}
Let $n\in\N$. Furthermore, let the nonsingular complex $n\times n$ matrix $\fM$ and the vector $\eta\in\C^n$ be chosen such that
\begin{equation}\label{lem211_36}
I_n-\fM\fM^*=\eta\eta^*
\end{equation}
holds. Then $1-\left\|\eta\right\|_{\C^n}^2>0$ and the vector
\begin{equation}\label{lem211_37}
\widetilde\eta:=\frac{1}{\sqrt{1-\left\|\eta\right\|_{\C^n}^2}}\fM^*\eta
\end{equation}
satisfies
\begin{equation}\label{lem211_38}
I_n-\fM^*\fM=\widetilde\eta\widetilde\eta^*.
\end{equation}
\end{lem}
\begin{proof}
The case $\eta=0_{n\times 1}$ is trivial. Now suppose that $\eta\in\C^n\backslash\{0_{n\times 1}\}$. From (\ref{lem211_36}) we get
\begin{equation}\label{lem211_40}
(I_n-\fM\fM^*)\eta=\eta\eta^*\eta=\left\|\eta\right\|_{\C^n}^2\cdot\eta
\end{equation}
and consequently
\begin{equation}\label{lem211_41}
\fM\fM^*\eta=(1-\left\|\eta\right\|_{\C^n}^2)\cdot\eta.
\end{equation}
Hence $1-\left\|\eta\right\|_{\C^n}^2$ is an eigenvalue of $\fM\fM^*$ with corresponding eigenvector $\eta$. Since the matrix $\fM\fM^*$ is nonsingular positive Hermitian, we have $1-\left\|\eta\right\|_{\C^n}^2>0$. Using (\ref{lem211_40}) we infer
\begin{equation}\label{lem211_43}
(I_n-\fM^*\fM)\fM^*\eta=\fM^*(I_n-\fM\fM^*)\eta=\left\|\eta\right\|_{\C^n}^2\cdot\fM^*\eta.
\end{equation}
Taking into account (\ref{lem211_41}) we can conclude
\begin{equation}\label{lem211_46}
\left\|\fM^*\eta\right\|_{\C^n}^2=\eta^*\fM\fM^*\eta=\eta^*\left[(1-\left\|\eta\right\|_{\C^n}^2)\cdot\eta\right]=(1-\left\|\eta\right\|_{\C^n}^2)\cdot\left\|\eta\right\|_{\C^n}^2
\end{equation}
and therefore from (\ref{lem211_37}) we have
\begin{equation}\label{lem211_44}
\left\|\widetilde\eta\right\|_{\C^n} = \left\|\eta\right\|_{\C^n} >0.
\end{equation}
Formulas (\ref{lem211_43}), (\ref{lem211_37}) and (\ref{lem211_44}) show that $\left\|\widetilde\eta\right\|_{\C^n}^2$ is an eigenvalue of $I_n-\fM^*\fM$ with corresponding eigenvector $\widetilde\eta$. From (\ref{lem211_36}) and $\eta\neq0_{n\times 1}$ we get $$\rank(I_n-\fM^*\fM)=\rank(I_n-\fM\fM^*)=1.$$
So for each vector $h$ we can conclude
$$(I_n-\fM^*\fM)h=(I_n-\fM^*\fM)(h,\frac{\widetilde\eta \ }{ \ \left\|\widetilde\eta\right\|_{\C^n}})_{_{\C^n}}
\frac{\widetilde\eta \ }{ \ \left\|\widetilde\eta\right\|_{\C^n}}=(h,\widetilde\eta)_{_{\C^n}}\widetilde\eta=\widetilde\eta\widetilde\eta^* \cdot h.$$
\end{proof}
\begin{prop}[\zitaa{DFK}{Proposition~2.12}]\label{prop212}
Let $\gamma=(\gamma_j)_{j=0}^\infty\in\Gamma l_2$ and let $n\in\N$. Then
$$I_n-\fM_n^*(\gamma)\fM_n(\gamma)
=\frac{1}{\prod_{k=1}^n(1-|\gamma_k|^2)}\fM_n^*(\gamma)\eta_n(\gamma)\eta_n^*(\gamma)\fM_n(\gamma)$$
where $\fM_n(\gamma)$ and $\eta_n(\gamma)$ are defined via \eqref{5.7} and \eqref{5.9}, respectively.
\end{prop}
\begin{proof}
    The combination of Lemma~\ref{5.3}, Lemma \ref{lem210} and Lemma \ref{lem211} yields the assertion.
\end{proof}
The following result should be compared with Lemma \ref{lem23}. Under the assumptions of Lemma \ref{lem23} we will verify that for each $n\in\N$ the right defect matrix $I_n-\fL_n^*(\gamma)\fL_n(\gamma)$ admits a series representation which is similar to the series representation for the left defect matrix $I_n-\fL_n(\gamma)\fL_n^*(\gamma)$.
\begin{prop}[\zitaa{DFK}{Proposition~2.13}]\label{prop213}
    Let $\gamma=(\gamma_j)_{j=0}^\infty\in\Gamma l_2$, let $n\in\N$ and let $\fL_n(\gamma)$ be defined via \eqref{5.3}.
    Then
\begin{equation}\label{prop213_46}
I_n-\fL_n^*(\gamma)\fL_n(\gamma)=\sum_{j=0}^\infty\tau_{n,j}(\gamma)\tau_{n,j}^*(\gamma)
\end{equation}
where for each $j\in\N_0$ the matrix $\tau_{n,j}(\gamma)$ is defined via
\begin{equation}\label{taunj}
\tau_{n,j}(\gamma):=\left(\prod_{k=j+1}^{j+n}D_{\gamma_k}^{-1}\right)\left[\stackrel{\longleftarrow}{\prod_{k=j}^\infty}\fM_n^*(W^k\gamma)\right]\eta_n(W^j\gamma)
\end{equation}
and where $D_{\gamma_k}$, $W$, $\fM_n(\gamma)$, and $\eta_n(\gamma)$ are given by \eqref{E2.19}, \eqref{3.14}, \eqref{5.7}, and \eqref{5.9}, respectively.
\end{prop}
\begin{proof}
From \eqref{5.6} we get
\begin{multline*}
    I_n-\fL_n^*(\gamma)\fL_n(\gamma)
    =\fL_n^*(W\gamma)\left[I_n-\fM_n^*(\gamma)\fM_n(\gamma)\right]\fL_n(W\gamma)\\
    +I_n-\fL_n^*(W\gamma)\fL_n(W\gamma).
\end{multline*}
Considering now $\fL_n(W\gamma)$ instead of $\fL_n(\gamma)$ and repeating the above procedure, we obtain, after $m$--steps, the formula
\begin{multline}\label{prop213_48}
    I_n-\fL_n^*(\gamma)\fL_n(\gamma)\\
    =\sum_{j=0}^{m-1}\fL_n^*(W^{j+1}\gamma)\left[I_n-\fM_n^*(W^j\gamma)\fM_n(W^j\gamma)\right]\fL_n(W^{j+1}\gamma)\\
    +I_n-\fL_n^*(W^m\gamma)\fL_n(W^m\gamma).
\end{multline}
From this, taking into account that for $n\in\N$
\[
    \lim_{m\to\infty}\cL_n(W^m\ga)
    =I_q,
\]
we obtain
\begin{equation}\label{prop213_49}
I_n-\fL_n^*(\gamma)\fL_n(\gamma)=\sum_{j=0}^\infty\fL_n^*(W^{j+1}\gamma)\left[I_n-\fM_n^*(W^j\gamma)\fM_n(W^j\gamma)\right]\fL_n(W^{j+1}\gamma).
\end{equation}
For each $j\in\N_0$, from \rcor{cor5.4} %
\begin{equation}\label{prop213_51}
\fL_n(W^j\gamma)=\stackrel{\longrightarrow}{\prod_{k=j}^{\infty}}\fM_n(W^k\gamma)
\end{equation}
follows and Proposition \ref{prop212} implies
\begin{multline}\label{prop213_52}
    I_n-\fM_n^*(W^j\gamma)\fM_n(W^j\gamma)\\
    =\left(\prod_{k=j+1}^{j+n}D_{\gamma_k}^{-2}\right)\fM_n^*(W^j\gamma)\eta_n(W^j\gamma)\eta_n^*(W^j\gamma)\fM_n(W^j\gamma).
\end{multline}
Now the combination of (\ref{prop213_49})--(\ref{prop213_52}) yields (\ref{prop213_46}).
\end{proof}
\begin{lem}[\zitaa{DFK}{\clem{2.14}}]\label{lem214}
Let $\gamma=(\gamma_j)_{j=0}^\infty\in\Gamma l_2$, let $r\in\N$, and let $\Pi_1$ be defined via \eqref{3.17}. Then
$$\sum_{n=r}^\infty|\gamma_{n+1}|^2\left[\prod_{k=1}^n(1-|\gamma_k|^2)\right]=\prod_{k=1}^r(1-|\gamma_k|^2)-\Pi_1^2.$$
\end{lem}
\begin{proof}
Taking into account \eqref{3.17} and \eqref{E2.19}, we obtain
\begin{eqnarray*}
&&\hspace{-1cm}\prod_{k=1}^r(1-|\gamma_k|^2)=\left[\prod_{k=1}^r(1-|\gamma_k|^2)\right]\left[|\gamma_{r+1}|^2+(1-|\gamma_{r+1}|^2)\right]\\
&=&|\gamma_{r+1}|^2\left[\prod_{k=1}^r(1-|\gamma_k|^2)\right]+\prod_{k=1}^{r+1}(1-|\gamma_k|^2)\\
&=&|\gamma_{r+1}|^2\left[\prod_{k=1}^r(1-|\gamma_k|^2)\right]+|\gamma_{r+2}|^2\left[\prod_{k=1}^{r+1}(1-|\gamma_k|^2)\right]+\prod_{k=1}^{r+2}(1-|\gamma_k|^2).
\end{eqnarray*}
Iterating this procedure, for each integer $m$ with $m\geq r$, we get
$$\prod_{k=1}^r(1-|\gamma_k|^2)-\Pi_1^2=\sum_{n=r}^m|\gamma_{n+1}|^2\left[\prod_{k=1}^n(1-|\gamma_k|^2)\right]+\prod_{k=1}^{m+1}(1-|\gamma_k|^2)-\Pi_1^2.$$
This yields the assertion after passing to the limit $m\rightarrow\infty$.
\end{proof}
\begin{thm}[\zitaa{DFK}{\cthm{2.15}}]\label{thm215}
Let $n_0\in\N$ and let $\gamma=(\gamma_j)_{j=0}^\infty$ be a sequence which belongs to $\Gamma l_2$ and satisfies $\rank\gamma=n_0$. Further, let $r$ be an integer with $r\geq n_0$, let $p=(p_r,\dotsc,p_0)^T$ be an $r$--th order $\cS$--recurrence vector associated with $\gamma$, and let $\lambda$ be defined via \eqref{lambda}. Then
\begin{equation}\label{thm215_53}
\prod_{k=1}^r(1-|\gamma_k|^2)-\Pi_1^2=\lambda^*\left( \ (\fL_r^{-1}(\gamma))^*\fL_r^{-1}(\gamma)-I_r\right)\lambda,
\end{equation}
where $\Pi_1$ and $\fL_r(\gamma)$ are defined by \eqref{3.17} and \eqref{5.3}, respectively.
\end{thm}
\begin{proof}
Let $n$ be an integer with $n\geq r$. Then, from Theorem \ref{thm28} by rewriting formula (\ref{DarGammaN+1}) we get the relation
\begin{equation}\label{thm215_55}
\gamma_{n+1}\left(\prod_{k=1}^nD_{\gamma_k}\right)=\left(\prod_{k=n-r+1}^nD_{\gamma_k}^{-1}\right)\eta_r^*(W^{n-r}\gamma)\left(\stackrel{\longleftarrow}{\prod_{k=0}^{n-r-1}}\fM_r^{-1}(W^k\gamma)\right)\lambda.
\end{equation}
From \rcor{cor5.4} it follows that
$$\stackrel{\longleftarrow}{\prod_{k=0}^{n-r-1}}\fM_r^{-1}(W^k\gamma)=\left[\stackrel{\longrightarrow}{\prod_{k=n-r}^\infty}\fM_r(W^k\gamma)\right]\fL_r^{-1}(\gamma).$$
Inserting this into (\ref{thm215_55}), we get
\begin{multline*}
    \gamma_{n+1}\left(\prod_{k=1}^nD_{\gamma_k}\right)\\
    =\left(\prod_{k=n-r+1}^nD_{\gamma_k}^{-1}\right)\eta_r^*(W^{n-r}\gamma)\left[\stackrel{\longrightarrow}{\prod_{k=n-r}^\infty}\fM_r(W^k\gamma)\right]\fL_r^{-1}(\gamma)\lambda.
\end{multline*}
This implies
\begin{eqnarray}\label{thm215_56}
&&\hspace{-1cm}\sum_{n=r}^\infty|\gamma_{n+1}|^2\left(\prod_{k=1}^nD_{\gamma_k}^2\right)\nonumber\\
&=&\sum_{n=r}^\infty\left(\prod_{k=n-r+1}^nD_{\gamma_k}^{-2}\right)\lambda^*(\fL_r^{-1}(\gamma))^* \left[\stackrel{\longleftarrow}{\prod_{k=n-r}^\infty}\fM_r^*(W^k\gamma)\right]\eta_r(W^{n-r}\gamma)\nonumber\\
&&\hspace{2cm}\cdot\eta_r^*(W^{n-r}\gamma)\left[\stackrel{\longrightarrow}{\prod_{k=n-r}^\infty}\fM_r(W^k\gamma)\right]\fL_r^{-1}(\gamma)\lambda\nonumber\\
&=&\lambda^*(\fL_r^{-1}(\gamma))^* \left(\sum_{n=r}^\infty\left(\prod_{k=n-r+1}^nD_{\gamma_k}^{-2}\right)\left[\stackrel{\longleftarrow}{\prod_{k=n-r}^\infty}\fM_r^*(W^k\gamma)\right]\eta_r(W^{n-r}\gamma)\right.\nonumber\\
&&\hspace{2cm}\left.\cdot\eta_r^*(W^{n-r}\gamma)\left[\stackrel{\longrightarrow}{\prod_{k=n-r}^\infty}\fM_r(W^k\gamma)\right]\right)\fL_r^{-1}(\gamma)\lambda
\end{eqnarray}
According to Lemma \ref{lem214} the left--hand side of equation (\ref{thm215_56}) can be rewritten as
\begin{equation}\label{thm215_57}
\sum_{n=r}^\infty|\gamma_{n+1}|^2\left(\prod_{k=1}^nD_{\gamma_k}^2\right)=\prod_{k=1}^r(1-|\gamma_k|^2)-\Pi_1^2.
\end{equation}
Substituting the summation index $j=n-r$ and taking (\ref{taunj}) and (\ref{prop213_46}) into account, we obtain
\begin{eqnarray}\label{thm215_58}
&&\hspace{-1cm}\sum_{n=r}^\infty\left(\prod_{k=n-r+1}^nD_{\gamma_k}^{-2}\right)\left[\stackrel{\longleftarrow}{\prod_{k=n-r}^\infty}\fM_r^*(W^k\gamma)\right]\eta_r(W^{n-r}\gamma)\nonumber\\
&&\hspace{2cm}\cdot\eta_r^*(W^{n-r}\gamma)\left[\stackrel{\longrightarrow}{\prod_{k=n-r}^\infty}\fM_r(W^k\gamma)\right]\nonumber\\
&=&\sum_{j=0}^\infty\left(\prod_{k=j+1}^{j+r}D_{\gamma_k}^{-2}\right)\left[\stackrel{\longleftarrow}{\prod_{k=j}^\infty}\fM_r^*(W^k\gamma)\right]\eta_r(W^j\gamma)\nonumber\\
&&\hspace{2cm}\cdot\eta_r^*(W^j\gamma)\left[\stackrel{\longrightarrow}{\prod_{k=j}^\infty}\fM_r(W^k\gamma)\right]\nonumber\\
&=&\sum_{j=0}^\infty\tau_{n,j}(\gamma)\tau_{n,j}^*(\gamma)=I_r-\fL_r^*(\gamma)\fL_r(\gamma).
\end{eqnarray}
The combination of (\ref{thm215_57}), (\ref{thm215_56}), and (\ref{thm215_58}) yields
\begin{eqnarray*}
\prod_{k=1}^r(1-|\gamma_k|^2)-\Pi_1^2&=&\lambda^*(\fL_r^{-1}(\gamma))^* \left(I_r-\fL_r^*(\gamma)\fL_r(\gamma)\right)\fL_r^{-1}(\gamma)\lambda\\
&=&\lambda^*\left( \ (\fL_r^{-1}(\gamma))^*\fL_r^{-1}(\gamma)-I_r\right)\lambda.
\end{eqnarray*}
Thus, the proof is complete.
\end{proof}

\begin{rem}[\zitaa{DFK}{\crem{2.16}}]\label{rem216}
    We reconsider Theorem \textnormal{\ref{thm215}} in the particular case that $n_0=1$ and $r=1$ holds.
    From \eqref{5.3} we get $\fL_1(\gamma)=\Pi_1$.
    Thus, equation \eqref{thm215_53} has the form $$1-|\gamma_1|^2-\Pi_1^2=|\lambda|^2\left(\frac{1}{\Pi_1^2}-1\right).$$
Hence,
\begin{equation}\label{rem216_59}
    \Pi_1^2\left(1-|\gamma_1|^2-\Pi_1^2\right)=|\lambda|^2\left(1-\Pi_1^2\right).
\end{equation}
Equation \eqref{rem216_59} was obtained in the proof of \rthm{thm5.22} (see \eqref{E5.31-1114}).
The method of proving Theorem~\ref{thm215} is a generalization to the case of a sequence $\gamma=(\gamma_j)_{j=0}^\infty\in\Gamma l_2$ having arbitrary finite rank of the method of proving equation \eqref{rem216_59} in \rthm{thm5.22} which works only for sequences having first rank.
\end{rem}

\section{Characterization of Schur parameter sequences of polynomial Schur functions}\label{sec7-1125}

\subsection{First observations on Schur functions of polynomial type}\label{Sec0.3}

The main aim of this section is the investigation of the Schur parameters of a particular subclass of the class $\mathcal S\Pi.$

A function $\theta\in\mathcal S$ is called \textit{of polynomial type} if there exists a polynomial $P$ with complex coefficients such that $\theta=\operatorname{Rstr.}_\mathbb{D}P$ is satisfied. Clearly, $P$ is then uniquely determined.

We denote by $\mathcal S\mathcal P$ the set of all Schur functions of polynomial type. If $m\in\N_0$, we denote by $\mathcal S\mathcal P_m$ the set of all functions $\theta\in\mathcal S\mathcal P$ of the shape $\theta=\operatorname{Rstr.}_\mathbb{D}P$, where $P$ is a polynomial of exact degree $m$.

Let $m\in\N_0$ and let $\theta\in\mathcal S\mathcal P_m$.
Then we are naturally led to the following two cases.
\begin{enumerate}
    \item $\theta\in\mathcal S\mathcal P_m\cap  J$\\
    Obviously, in this case, there exists a unique $u\in\mathbb T$ such that $\theta(\zeta)=u\cdot \zeta^m$.
    Hence, the Schur parameter sequence of $\theta$ is given by $\gamma_m=u$ and, if $m>0$, additionally, by $\gamma_0=\dotsb=\gamma_{m-1}=0.$
    \item $\theta\in\mathcal S\mathcal P_m\setminus J$\\
    We start with a family of remarkable members of $\mathcal S\mathcal P_m\setminus J$.
    
    \begin{ex}\label{Ex1.18A}
    Let $m\in\N$ and $u\in\mathbb D\setminus\{0\}$.
    Define the function $\theta\colon\mathbb D\to\C$ by $\theta(\zeta):=u\zeta^m$.
    Obviously, then $\theta\in\mathcal S\mathcal P_m\setminus J$ and the Schur parameter sequence $(\ga_j)_{j=0}^\infty$ of $\theta$ is given by
    \[
        \ga_j
        =
        \begin{cases}
            0&\text{if }j\in\mathbb N_0\setminus\{m\}\\
            u&\text{if }j=m
        \end{cases}.
    \]
    \end{ex}
\end{enumerate}
In view of $\mathcal S\mathcal P_m\subseteq\mathcal S\Pi$, the application of Theorem \ref{thm4.2} immediately yields:
\begin{thm}\label{t1.19}
Let $\theta\in\mathcal S\mathcal P_m\setminus J$ with Schur parameters $(\gamma_j)_{j=0}^m$. Then $w=\infty$ and the product
\beql{eq: 1.10}
\prod\limits_{j=0}^\infty(1-|\ga_j|^2)
\eeq
converges.
\end{thm}

\subsection{On a theorem of R. I. Teodorescu}\label{Sec2.6}
The investigations by R. I. Teodorescu (see \cite{T76}, \cite{T78}) on  contractive operators in Hilbert spaces the c. o. f. of which is of polynomial type played an important role in this section. In particular, the following result influenced our approach.
\begin{thm}[{\cite{T78}}]\label{t1.29}
Let $T$ be a contraction on a Hilbert space such that the characteristic operator function of $T$ is a polynomial of degree $m$. Then $\mathfrak H$ admits a decomposition into three orthogonal subspaces such that the corres-ponding matrix of $T$ has the form
\begin{align*}
    \begin{pmatrix}
    T_1&*&*\\0&T_2&*\\0&0&T_3
    \end{pmatrix},
\end{align*}
where $T_1, ~T_3^*$ are isometries and $T_2^m=0.$
\end{thm}

We formulate a modified version of Theorem \ref{t1.29}, which is connected with the above constructed operator model.

\begin{thm}\label{t1.30}
Let $\theta\in\mathcal S$ with Schur parameter sequence $\gamma=(\gamma_j)_{j=0}^w$. Further, let
\beql{eq: 1.11}
\Delta
\defeq(\cH,\cF,\cG;T,F,G,S)
\eeq
be a simple unitary colligation which satisfies $\theta_\Delta=\theta$.
Then the following statements are equivalent:
\begin{itemize}
    \item[\textnormal{(i)}] $\theta\in\mathcal S\mathcal P\setminus J$.
    \item[\textnormal{(ii)}] $w = \infty$, the sequence $\gamma$ belongs to $\Gamma\ell_2$ and there exists an $n\in\N$ such that the operator $T_{\cG\cF}$ taken from the triangulation \eqref{eq: 1.38} is nilpotent.
\end{itemize}
If one of the equivalent conditions \textnormal{(i)} and \textnormal{(ii)} is satisfied, then
\begin{align}\label{eq: 11.45}
\operatorname{degree}\theta=\min\{n\in\N:~(T_{\cG\cF})^n=0\}.
\end{align}
\end{thm}
\begin{rem}\label{r10.11}
Note that the analogous assertion is true also for the triangulation \eqref{eq: 10.38}.
\end{rem}
\begin{proof}
Suppose that the Taylor series representation of $\theta$ is given by
\begin{align}
    \theta(\zeta)=\sum_{n=0}^\infty c_n\zeta^n.\label{eq: 1.39}
\end{align}
From the form \eqref{eq: 1.11} of $\Delta$ and Definition 2.13 we obtain
\begin{align}
    \theta_\Delta(\zeta)=S+\sum_{n=1}^\infty \zeta^nGT^{n-1}F.\label{eq: 1.40}
\end{align}
In view of $\theta_\Delta=\theta$ the identity theorem for holomorphic functions gives us
\begin{align}
    c_0=S,~~~~~c_n=GT^{n-1}F,~n\in\N.\label{eq: 1.41}
\end{align}

(i)$\rightarrow$(ii). Theorem \ref{t1.19} implies that the product \eqref{eq: 1.10} converges. Conse-quently, $\gamma\in\Gamma\ell_2$. Let $m$ be the degree of the polynomial $\theta$. From \eqref{eq: 1.41} we get
\begin{align*}
    F^*(T^*)^rG^*=\overline{c_{r+1}}=0,~r\in\{m,m+1,\dots\}.
\end{align*}
From this it follows
\begin{align}
    F^*(T^*)^k(T^*)^r(T^*)^\ell G^*=0\label{eq: 1.44}
\end{align}
for $k,\ell\in\{0,1,2,\dots\}$ and $r\in\{m,m+1,\dots\}$.
If we fix $k$ and $r$ in \eqref{eq: 1.44} and vary $\ell$, then from second equality \eqref{eq: 1.7} we obtain
\begin{align}
    F^*(T^*)^k(T^*)^rh=0\label{eq: 1.45}
\end{align}
for each $h\in\cH_\cG$, $k\in\{0,1,2,\dots\}$ and $r\in\{m,m+1,\dots\}$.
In \eqref{eq: 1.45} we fix now $r$ and vary $k$. Then, from first equality \eqref{1.9}, we obtain
\begin{align}\label{eq: 10.45}
    (T^*)^rh\in\mathfrak H_\cF^\perp,~r\in\{m,m+1,\dots\}.
\end{align}
for each $h\in\cH_\cG$.
The subspace $\cH_\cG$ is invariant with respect to $T^*$ and $T_{\cG}^*=\operatorname{Rstr.}_{\cH_{\cG}}{T^*}.$
Thus, for each $h\in\cH_\cG$ and $r\in\{m,m+1,\dots\}$,
the equality
\begin{align}
    (T_{\cG}^*)^rh\in\mathfrak H_\cG\cap\cH_\cF^\perp=\mathfrak N_{\cG\cF}\label{eq: 1.46}
\end{align}
follows from \eqref{eq: 10.45} and the first equality \eqref{4.2}.

We note that the contraction $T_{\cG}$ admits the triangulation \eqref{eq: 11.38}.
Thus, the condition \eqref{eq: 1.43} is equivalent to
\begin{align*}
    (T_{\cG\cF})^m=0.
\end{align*}
Hence, if for $r\in\N$ it holds $(T_{\cG\cF})^r=0$, then
\begin{align}r\le\operatorname{degree}\theta.\label{eq: 1.47B}\end{align}
Thus, condition (ii) is satisfied.

(ii)$\rightarrow$(i). The operators $V_t$ and $\widetilde V_{T_\mathfrak{G}}$ occuring in the block representation \eqref{eq: 1.38} are a unilateral shift and a unilateral coshift, respectively. Hence,
\begin{align*}
    V_T^* V_T = I_{\mathfrak G^\perp}, \ \ \ \wt V_{T_{\mathfrak G}} \wt V_{T_{\mathfrak G}}^*
    = I_{\mathfrak N_{\mathfrak G\mathfrak F}}.
\end{align*}
This implies that with respect to the orthogonal decomposition \eqref{4.4} we have the block representations
\begin{align*}
    I - T^* T =   \begin{pmatrix}
    0&0&0\\0&*&*\\0&*&*
    \end{pmatrix}, \ \ \  I - TT^* = \begin{pmatrix}
                    *&*&0\\ *&*&0\\0&0&0
                    \end{pmatrix}.
\end{align*}
From this relation and the conditions of unitarity of a colligation (see \eqref{1.3})
\begin{align*}
    T^*T+G^*G=I_{\mathfrak H}, \ \ \ TT^*+FF^*=I_{\mathfrak H}
\end{align*}
we conclude that with respect to the orthogonal decomposition \eqref{4.4} the mappings $G$ and $F$ have the block representations
\begin{align*}
    G = ( 0, *, * ), \ \ \ F =  \begin{pmatrix}
    *\\ *\\0
    \end{pmatrix}.
\end{align*}
From these representations and triangulation \eqref{eq: 1.38} we obtain
\begin{align}
    c_n=GT^{n-1}F=GP_{\cG\cF}T_{\cG\cF}^{n-1}P_{\cG\cF}F,~n\in\N.\label{eq: 1.43}
\end{align}
where $P_{\cG\cF}$ is the orthogonal projection from $\cH$ onto $\cH_{{\mathfrak G\mathfrak F}}$. Because of (ii) there exists an $n\in\N$ such that $T^n_{\mathfrak G\mathfrak{F}}=0$. Then, in view of \eqref{eq: 1.43},
\begin{align*}
    c_l=0,~l>n.
\end{align*}
Thus, $\theta\in\mathcal S\mathcal P$ and
\begin{align}\label{eq: 1.44A}
\operatorname{degree}\theta\le\min\{n\in\N:~(T_{\cG\cF})^n=0\}.
\end{align}
Taking \ into account $\gamma\in\Gamma\ell_2$, we infer from Theorem \ref{thm4.2} that $\theta\not\in\ J$. Hence, (i)  is satisfied.

From \eqref{eq: 1.47B} and \eqref{eq: 1.44A} we infer
\begin{align*}
    \operatorname{degree}\theta=\min\{n\in\N:~(T_{\cG\cF})^n=0\}.
\end{align*}
\end{proof}

\subsection{The matrix of the operator $T_{\cG\cF}$ with respect to the basis $(h_k)_{k=1}^m$ in the case $\operatorname{dim}\cH_{\cG\cF}=m<+\infty$}\label{Sec3.1}
The following assertion plays an important role in our approach.
\begin{thm}\label{t30.1}
    Let $\theta\in\mathcal S$ be such that its Schur parameter sequence $\gamma=(\gamma_j)_{j=0}^\infty$ satisfies $\gamma\in\Gamma\ell_2$. Further, let $\Delta$ be a simple unitary colligation of the form \eqref{eq: 1.11} which satisfieds $\theta_\Delta=\theta$. Consider the othorgonal decomposition \eqref{4.4}) of $\mathfrak H$ and
let $\operatorname{dim}\mathfrak{H}_{\cG\cF}=m<+\infty$. Then the matrix of the operator $T_{\cG\cF}$
(see \eqref{eq: 1.38})  with respect to the basis $(h_k)_{k=1}^m$ (see \eqref{5.1}) is given by
\begin{equation}
    \begin{psmallmatrix}
        -\overline{\gamma}_0\gamma_1	& -\overline{\gamma}_0D_{\gamma_1}\gamma_2 	& \hdots 	 & -\overline{\gamma}_0\prod\limits_{j=1}^{m-2}D_{\gamma_j}\gamma_{m-1}&
        -\overline{\gamma}_0\prod\limits_{j=1}^{m-1}D_{\gamma_j}\gamma_m -D_{\ga_m}\alpha_1\\
        D_{\gamma_1} 	& -\overline{\gamma}_1\gamma_2  		& \hdots 	& -\overline{\gamma}_1\prod\limits_{j=2}^{m-2}D_{\gamma_j}\gamma_{m-1}&
        -\overline{\gamma}_1\prod\limits_{j=2}^{m-1}D_{\gamma_j}\gamma_m -D_{\ga_m}\alpha_2\\
        0 		& D_{\gamma_2} 			& \hdots 	& -\overline{\gamma}_2\prod\limits_{j=3}^{m-2}D_{\gamma_j}\gamma_{m-1} & -\overline{\gamma}_2\prod\limits_{j=3}^{m-1}D_{\gamma_j}\gamma_m
        -D_{\ga_m}\alpha_3\\
        \vdots 		& \vdots 			& 		& \vdots 					  	 &\vdots \\
        0 		& 0 				& \hdots 	&D_{\gamma_{m-1}}& -\overline{\gamma}_{m-1}\gamma_m-D_{\ga_m}\alpha_m
    \end{psmallmatrix},\label{eq: 2.2}
\end{equation}
where the vector
\begin{align}
    a:=\operatorname{col}(\alpha_1,\alpha_2,\dots,\alpha_m,1)\label{eq: 2.3}
\end{align}
satisfies the homogeneous linear system
\begin{align}
    \mathcal A_{m+1}(\ga)\cdot a=0_{(m+1)\times 1}
\end{align}
and $\mathcal A_n(\ga)$ is given by \eqref{5.4}.
\end{thm}
\begin{proof}
The assumption $\operatorname{dim}\mathfrak H_{\cG\cF}=m<+\infty$ implies (see Remark \ref{r1.28}) that the vectors $(h_k)_{k=1}^m$ form a basis of the space $\mathfrak H_{\cG\cF}$, whereas the vectors $(h_k)_{k=m+1}^\infty$ are linear combinations of them. In the language of determinants $(\sigma_n(\ga))_{n\in\N}$ (see \eqref{5.5}) this means that
\begin{align*}
    \sigma_m(\ga)>0,~~~\sigma_{m+1}(\ga)=0.
\end{align*}
Thus, there exist complex numbers $\alpha_1,\dots,\alpha_n$ such that
\begin{align}
    h_{m+1}=-\alpha_1 h_1-\alpha_2 h_2-\dots-\alpha_m h_m\label{eq: 2.5}.
\end{align}
It can be seen from \eqref{5.2} that the vector
\begin{align*}
    a:=\operatorname{col}(\alpha_1,\alpha_2,\dots,\alpha_m,1)
\end{align*}
satisfies the equation
\begin{align*}
    \mathcal A_{m+1}(\ga)\cdot a=0_{(m+1)\times 1}.
\end{align*}
In order to obtain the shape \eqref{eq: 2.2} of the matrix $T_{\cG\cF}$ we apply the operator $T$ on both sides of \eqref{5.1}. This gives us
\begin{align}
    Th_n=T\phi_n-\Pi_n\sum_{j=1}^n\overline{L_{n-j}(W^j\gamma)}T\wt \psi_j,~~~n\in\N.\label{eq: 2.6}
\end{align}
From the model representation \eqref{2.63} we get
\begin{align*}
    T\phi_1=T_\mathfrak F\phi_1=-\overline{\ga_0}\ga_1\phi_1+D_{\ga_1}\phi_2,
\end{align*}
\begin{align*}
    T\phi_n=T_\mathfrak F\phi_n=-\sum_{k=1}^{n-1}(\overline{\ga_{k-1}}\prod_{j=k}^{n-1}D_{\gamma_j}\ga_n)\phi_k
-\overline{\ga_{n-1}}\ga_n \phi_n + D_{\ga_n}\phi_{n+1},~~~n\geq 2.
\end{align*}
Inserting these expressions into \eqref{eq: 2.6} and taking into account \eqref{3.7}, we obtain
\begin{align}
    Th_1 = -\overline{\ga_0}\ga_1\phi_1+D_{\ga_1}\phi_2 - \Pi_1\wt \psi_2,\label{eq: 20.7}
\end{align}
\begin{align}
Th_n= - \sum_{k=1}^{n-1}(\overline{\ga_{k-1}}\prod_{j=k}^{n-1}D_{\gamma_j}\ga_n)\phi_k
-\overline{\ga_{n-1}}\ga_n \phi_n + D_{\ga_n}\phi_{n+1} - \notag\\
- \Pi_n\sum_{j=1}^n\overline{L_{n-j}(W^j\gamma)}\wt \psi_{j+1},~~~n\geq 2.\label{eq: 2.7}
\end{align}

Taking into account the orthogonality relations $\wt \psi_j\perp\cH_\cG,~j\in\N,$ from \eqref{5.1} we infer
\begin{align*}
    h_n=P_{\cG\cF}h_n=P_{\cG\cF}\phi_n,~~~n\in\N,
\end{align*}
where $P_{\cG\cF}$ is the orthogonal projection from $\cH$ onto $\cH_{{\mathfrak G\mathfrak F}}.$
Further from the triangulation \eqref{eq: 1.38} we get
\begin{align*}
    P_{\cG\cF}Th = T_{\cG\cF}h, \ \ h\in\cH_{\cG\cF}.
\end{align*}
For this reason, applying the orthoprojection $P_{\cG\cF}$ on both sides of the identities \eqref{eq: 20.7}
and \eqref{eq: 2.7} , for $1\le n\le m-1$, we get
\[
    T_{\cG\cF} h_1
    = -\overline{\ga_0}\ga_1 h_1 + D_{\ga_1} h_2
\]
and
\begin{multline*}
    T_{\cG\cF} h_n
    = - \sum_{k=1}^{n-1}\left(\overline{\ga_{k-1}}\prod_{j=k}^{n-1}D_{\gamma_j}\ga_n\right) h _k
    -\overline{\ga_{n-1}}\ga_n h_n + D_{\ga_n} h_{n+1},\\
    2\leq n \leq m-1.
\end{multline*}
Taking into account \eqref{eq: 2.5}, from \eqref{eq: 2.7} it follows, for $n=m$, then
\begin{align*}
    T_{\cG\cF}h_m&=-\sum_{k=1}^{m-1}(\overline{\ga_{k-1}}\prod_{j=k}^{m-1}D_{\gamma_j}\ga_m) h _k
-\overline{\ga_{m-1}}\ga_m h_m + D_{\ga_m} h_{m+1}  \\
    &= - \sum_{k=1}^{m-1}(\overline{\ga_{k-1}}\prod_{j=k}^{m-1}D_{\gamma_j}\ga_m + D_{\ga_m}\alpha_k)) h _k
-(\overline{\ga_{m-1}}\ga_m  + D_{\ga_m}\alpha_m) h_m.
\end{align*}
\end{proof}

The following assertion specifies Theorem \ref{t1.30}.
\begin{thm}\label{t30.27}
Let $\theta\in\mathcal S$ with Schur parameter sequence $\gamma=(\gamma_j)_{j=0}^w$. Let $\Delta$ be a simple unitary colligation of type \eqref{eq: 1.11} which satisfies $\theta_\Delta=\theta$. Further, let $m\in\N.$
Then the following statements are equivalent:
\begin{itemize}
    \item[(i)] $\theta\in\mathcal S\mathcal P_m\setminus J$.
    \item[(ii)] $w = \infty$, the sequence $\gamma$ belongs to $\Gamma\ell_2$, the subspace $\cH_{\cG\cF}$taken from the decomposition \eqref{4.4} satisfies $\operatorname{dim}\mathfrak{H}_{\cG\cF}=m$
and the operator $T_{\cG\cF}$ taken from the triangulation \eqref{eq: 1.38} is nilpotent.
\end{itemize}
If one of the equivalent conditions (i) and (ii) is satisfied, then the nilpotency index
of $T_{\cG\cF}$ equals $m.$
\end{thm}
\begin{proof}
(i)$\rightarrow$(ii). Let $\theta\in\mathcal S\mathcal P_m\setminus J$. Then, from Theorem \ref{t1.30}
we get $w = \infty, \gamma \in\Gamma\ell_2$ and $T_{\cG\cF}$ is nilpotent.
Let $r := \operatorname{dim}\mathfrak{H}_{\cG\cF}$. Then $(T_{\cG\cF})^r = 0.$
Thus, taking into account \eqref{eq: 11.45}, it follows
\begin{align}\label{eq: 21.8}
    r\geq m.
\end{align}
From the shape \eqref{eq: 2.2} of the matrix $T_{\cG\cF}$ with respect to the basis $(h_k)_{k=1}^r$ it follows
\begin{align}\label{eq: 21.9}
    T_{\cG\cF}^{(r-1)} h_1 = \prod_{k=0}^{r-1}D_{\gamma_k} h_r + \sum_{k=1}^{r-1} c_k h_k \neq 0.
\end{align}
Thus, $r\le m$. Combining this with \eqref{eq: 21.8} we get $r=m$.

(ii)$\rightarrow$(i).  In this case, it follows from Theorem \ref{t1.30} that $\theta\in\mathcal S\mathcal P_m\setminus J$. Moreover, for $r=m$ we see from \eqref{eq: 21.9} that in the given case $T_{\cG\cF}^{(m-1)}\not=0$. However, since $T_{\cG\cF}$ is a nilpotent operator and $\operatorname{dim}\mathfrak H_{\cG\cF}=m$, then $T_{\cG\cF}^m=0$. Thus, it follows from \eqref{eq: 11.45} that $\operatorname{degree}\theta=m$.
\end{proof}

\subsection{Schur parameter sequences of polynomial Schur functions. Examples}\label{Sec3.2}
The main aim of this Section~\ref{sec7-1125} is to prove the following result.
\begin{thm}\label{t2.2}
Let $\theta\in\mathcal S$ and let $\Delta$ be a simple unitary colligation of the form \eqref{eq: 1.11} which satisfies $\theta_\Delta=\theta$. Let $\ga=(\ga_j)_{j=0}^w$ be the Schur parameter sequence of $\theta$. Let $m\in\N$. Then $\theta\in\mathcal S\mathcal P_m$ if and only if one of the following two conditions is satisfied:
\begin{itemize}
    \item[(1)] It hold $w=m,~\ga_0=\ga_1=\dotsb=\ga_{m-1}=0,~\ga_m=u\in\mathbb T$. In this case $\theta(\zeta)=u\zeta^m,~\zeta\in\mathbb D$, and $\theta$ is obviously an inner function.
    \item[(2)] It hold $w=\infty$, $\gamma\in\Gamma\ell_2$ and the relation
    \begin{align}
        \sigma_m>0,~~~\sigma_{m+1}=0\label{eq: 2.8}
    \end{align}
    are satisfied, where the determinants $\sigma_k,~k\in\{0,1,2,\dots\}$ are given via \eqref{5.5}, whereas matrix \eqref{eq: 2.2} is nilpotent. In this case $\theta\in\mathcal S\mathcal P_m\setminus J$
and the nilpotency index of matrix \eqref{eq: 2.2} equals $m$.
\end{itemize}
\end{thm}
\begin{proof}

Obviously, condition (1) provides a complete characterization of polynomials, which are inner functions.

Let $\theta\in\mathcal S\mathcal P_m\setminus J$. Then Theorem \ref{t30.27} and Theorem \ref{t30.1}  imply that
$ \operatorname{dim}\mathfrak{H}_{\cG\cF}=m, \ T_{\cG\cF}$ is nilpotent and the nilpotency index of matrix
\eqref{eq: 2.2} equals $m$. Further, in view of Remark \ref{r1.28}, we have
\begin{align*}
    \sigma_m>0,~~~\sigma_{m+1}=0.
\end{align*}
Hence, the necessity of the conditions of the Theorem is proved.

Conversely, the conditions \eqref{eq: 2.8} mean that $\operatorname{dim}\cH_{\cG\cF}=m$. Thus,
we infer from Theorem \ref{t30.27} and Theorem \ref{t30.1} that $\theta\in\mathcal S\mathcal P_m\setminus J$.
\end{proof}
Note that Theorem~\ref{thm28} is important in connection with the consideration of subsequent examples.

\begin{ex}\label{e2.5}
We already noted (see Theorem~\ref{thm5.22}) that I. Schur considered in his fundamental paper \cite[Part II]{Schur} as an example the function
\begin{align}
    \theta(\zeta):=\frac{1+\zeta}{2}, \ \ \zeta\in\mathbb D.\label{eq: 2.12}
\end{align}
Its Schur parameter sequence is given by
\begin{align}
    \ga_0=\frac{1}{2},~\ga_1=\frac{2}{3},~\ga_2=\frac{2}{5},~\ga_3=\frac{2}{7},\dots\label{eq: 2.13}
\end{align}
This example is of interest for us because the function $\theta$ belongs to $\mathcal S\mathcal P_1\setminus J$.
The Schur parameter sequence \eqref{eq: 2.13} has rank one.
Thus, the sequence \eqref{eq: 2.13} satisfies the conditions (see Theorem~\ref{thm5.22})
\begin{align}
    \ga_{n+1}=\lambda\frac{\ga_n}{\prod_{j=1}^n[1-|\ga_j|^2]},~~~n\ge1,\label{eq: 2.14}
\end{align}
where $\lambda=\frac{1}{3}$. Observe that relation \eqref{eq: 2.14} is just the condition \eqref{DarGammaN+1} for $m=1$.

Since the sequence \eqref{eq: 2.13} belongs to $\Gamma\ell_2$ and consists of nonzero elements,
we have in this case
\begin{align*}
    \sigma_1(\ga)=1-\prod_{j=1}^\infty(1-|\ga_j|^2)>0.
\end{align*}
In \cite{D06} or \cite{DFK} it was shown that each sequence for which \eqref{eq: 2.14} is fulfilled, satisfies $\sigma_2(\ga)=0$. Thus, for the sequence \eqref{eq: 2.13} we have $\sigma_1(\ga)>0$ and $\sigma_2(\ga)=0$ what completely agrees with the fact that the function \eqref{eq: 2.12} is a polynomial of first degree.

It remains to prove that in the case $m = 1$ the matrix \eqref{eq: 2.2} formed from the sequence \eqref{eq: 2.13} is nilpotent, which means that the number $-\overline{\ga_0}\ga_1-D_{\ga_1}\alpha_1$ has to be zero. In the given case $\gamma_0=\frac{1}{2}$, $\ga_1=\frac{2}{3}$ and, in view of \eqref{lambda}, we have $-D_{\ga_1}\alpha_1=\lambda $. But we have observed above that $\lambda=\frac{1}{3}$. Consequently,
\begin{align*}
-\overline{\ga_0}\ga_1-D_{\ga_1}\alpha_1=0.
\end{align*}
Hence, the data of the polynomial \eqref{eq: 2.12} are in harmony with Theorem \ref{t2.2}.
\end{ex}

\begin{ex}\label{e2.6}
We consider the function $\theta$ given by
\begin{align}
    \theta(\zeta):=\frac{1+\zeta^2}{2},~~~\zeta\in\mathbb D\label{eq: 2.15}
\end{align}
In this case the function $\theta$ belongs to $\mathcal S\mathcal P_2\setminus J$.
Taking into account Example \ref{e2.5} and the shape of the function \eqref{eq: 2.15}, we obtain
\begin{align}
    \gamma_0=\frac{1}{2},~\ga_1=0,~\ga_2=\frac{2}{3},~\ga_3=0,~\ga_4=\frac{2}{5},~\ga_5=0,~\ga_6=\frac{2}{7},\dots\label{eq: 2.16}
\end{align}
From Theorem \ref{t2.2} and Theorem \ref{thm28} it follows that
the Schur parameter sequence $(\gamma_j)_{j=0}^\infty$ of $\theta$ fulfils the conditions \eqref{DarGammaN+1} for $m=2$. This can be verify directly.

Writing the equations \eqref{DarGammaN+1} for $n=2$ and $n=3$, (i.e., for $\ga_3$ and $\gamma_4$,) and solving the corresponding system of linear equations of $\lambda=(\lambda_1,\lambda_2)^T$, we obtain
\begin{align}
    \lambda_1=\frac{1}{3}, ~\lambda_2=0.\label{eq: 2.17}
\end{align}
Furthermore, we see that the conditions \eqref{DarGammaN+1} are satisfied for $m=2$, $n\geq4$ and
these $\lambda_1, \ \lambda_2$.

In this example as well as in the preceding, we have $\sigma_1(\ga)>0$. If it would be $\sigma_2(\ga)=0$, then as it was shown in \cite{D06}, \cite{DFK} the Schur parameter sequence \eqref{eq: 2.16} would satisfy the condition \eqref{eq: 2.14}. However, it can be easily checked that this is not true. Hence, $\sigma_2(\ga)>0$.
In \cite{DFK} it was shown that if for $m = 2$ the condition \eqref{DarGammaN+1} is fulfilled, then the
corresponding sequence satisfies $\sigma_3(\ga)=0$. This agrees completely with the fact that the function \eqref{eq: 2.15} is a polynomial of second degree.

It remains to prove that in the case $m=2$ the nilpotency index of the matrix \eqref{eq: 2.2}
formed from the sequence \eqref{eq: 2.16} equals 2. This matrix is given by
\begin{align}\begin{pmatrix}
    -\overline{\gamma_0}\gamma_1&-\overline{\gamma_0}D_{\ga_1}\ga_2-D_{\ga_2}\alpha_1\\D_{\ga_1}&-\overline{\ga_1}\ga_2-D_{\ga_2}\alpha_2\end{pmatrix}.\label{eq: 2.18}
\end{align}
Taking into account that the vector $a=(\alpha_1,\alpha_2)^T$ satisfies \eqref{lambda} and the concrete form \eqref{eq: 2.17} of the coordinates of the vector $\lambda=(\lambda_1,\lambda_2)^T$, we obtain
\begin{align*}
    \frac{1}{3}=-D_{\ga_1}D_{\ga_2}\alpha_1,~~\alpha_2=0.
\end{align*}
From this and formulas \eqref{eq: 2.16} for the Schur parameters of the function \eqref{eq: 2.15} we obtain that the matrix \eqref{eq: 2.18} has shape
\begin{align*}
    \begin{pmatrix}0&0\\1&0\end{pmatrix}.
\end{align*}
Thus,
\begin{align*}
    \begin{pmatrix}0&0\\1&0\end{pmatrix}^2=\begin{pmatrix}0&0\\0&0\end{pmatrix}.
\end{align*}
Hence, all conditions of Theorem \ref{t2.2} are satisfied for the polynomial $\theta.$

\end{ex}

    :

    \section{Characterization of Helson-Szeg\H o measures in terms of the Schur parameters of the associated Schur function}\label{sec8}
\subsection{Interrelated quadruples consisting of a probability measure, a normalized Carath\'eodory function, a Schur function and a sequence of contractive complex numbers}
\label{s0}

The central object in this section is the class $\cM_+(\T)$ of all finite nonnegative measures on
the Borelian $\sigma$-algebra $\gB$ on $\T$. A measure $\mu\in\cM_+(\T)$ is called probabiltity measure if $\mu(\T)=1$. 
We denote by $\cM_+^1(\T)$ the subset of all probability measures which belong to $\cM_+(\T)$. 

Now we are going to introduce the subset of Helson-Szeg\H o measures on $\T$. For this reason, we denote by $\Pol$
the set of all trigonometric polynomials, i.e. the set of all functions $f:\T\rightarrow\C$ for which 
there exist a finite subset $I$ of the set $\Z$ of all integers and a sequence $(a_k)_{k\in I}$ from $\C$ such that 
\begin{align}\label{s0-1}
f(t)&=\sum_{k\in I}a_kt^k, \qquad t\in\T.
\end{align}
If $f\in\Pol$ is given via (\ref{s0-1}), then the conjugation $\tilde f$ of $f$ is defined via
\begin{align}\label{s0-2}
\tilde f(t):=-i\sum_{k\in I}(\sgn k)a_kt^k, \qquad t\in\T,
\end{align}
where $\sgn 0:=0$ and where $\sgn k:=\frac k{|k|}$ for each $k\in\Z\setminus\{0\}$.

\begin{defn}\label{s0-d1}
A non-zero measure $\mu$ which belongs to $\cM_+(\T)$ is called a \emph{Helson-Szeg\H o measure} if there exists a positive real constant $C$
such that for all $f\in\Pol$ the inequality
\begin{align}\label{s0-d1-1}
\int_\T|\tilde f(t)|^2\mu(dt)\le C\int_\T|f(t)|^2\mu(dt)
\end{align}
is satisfied.
\end{defn}

If $\mu\in\cM_+(\T)$, then $\mu$ is a Helson-Szeg\H o measure if and only if 
$\alpha\mu$ is a Helson-Szeg\H o measure for each $\alpha\in(0,+\infty)$. Thus, the investigation of 
Helson-Szeg\H o measures can be restricted to the class $\cM_+^1(\T)$.

The main goal of this section is to describe all Helson-Szeg\H o measures $\mu$ belonging to $\cM_+^1(\T)$
in terms of the Schur parameter sequence of some Schur function $\theta$ which will be associated with $\mu$.
For this reason, we will need the properties of Helson-Szeg\H o measures listed below (see Theorem \ref{s0-t1}).
For more information about Helson-Szeg\H o measures, we refer the reader, e.g., to \cite[Chapter~4]{MR0628971}, \cite[Chapter 7]{K}, \cite[Chapter 5]{5}.

Let $\cC$ be the Carath\'eodory class of all functions $\Phi:\D\rightarrow\C$ which are holomorphic in $\D$ 
and which satisfy $\Re\Phi(\zeta)\ge0$ for each $\zeta\in\D$. Furthermore, let 
$$ \cC^0 :=\{\Phi\in\cC:\Phi(0)=1\}.$$

The class $\cC $ is intimately related with the class $\cM_+(\T)$.
According to the Riesz-Herglotz theorem (see, e.g., \cite[Chapter 1]{RR}), 
for each function $\Phi\in\cC $ there exist a unique measure
$\mu\in\cM_+(\T)$ and a unique number $\beta\in\R$ such that
\begin{align}\label{s0-4}
\Phi(\zeta) &= \int_\T\frac{t+\zeta}{t-\zeta}\,\mu(dt)+i\beta,  \qquad \zeta\in\D.
\end{align}
Obviously, $\beta = {\rm Im}\, [\Phi (0)]$. On the other hand, it can be easily checked that, for arbitrary
$\mu\in\cM_+(\T)$ and $\beta \in\R$, the function $\Phi$ which is defined by the
right-hand side of (\ref{s0-4}) belongs to $\cC $. 
If we consider the Riesz-Herglotz representation (\ref{s0-4}) for a function $\Phi\in\cC^0 $,
then $\beta =0$ and $\mu$ belongs to the set $\cM_+^1 (\mathbb T)$.
Actually, in this way we obtain a bijective correspondence between the classes $\cC^0 $ and $\cM_+^1(\T)$.

Let $\te\in\cS$.  Then the function $\Phi:\mathbb D\to\mathbb C$ defined by
\begin{equation}\label{Nr.0.1}
\Phi (\zeta) := \frac{1+\zeta\Theta (\zeta)}{1-\zeta\Theta (\zeta)}
\end{equation}
belongs to the class $\cC^0 $. Note that from \eqref{Nr.0.1} it follows
\begin{equation}\label{Nr.0.1f}
\zeta\Theta (\zeta) = \frac{\Phi (\zeta) - 1}{\Phi (\zeta) + 1}\,, \quad \zeta\in\mathbb D.
\end{equation}
Consequently, it can be easily verified that via \eqref{Nr.0.1} a bijective correspondence between the classes $\cS $
and $\cC^0 $ is established.

In the result of the above considerations we obtain special ordered qudruples $[\mu,\Phi,\Theta,\gamma]$ consisting
of a measure $\mu\in\cM_+^1(\T)$, a function $\Phi\in\cC^0 $, a function $\Theta\in\cS $, 
and Schur parameters $\gamma=(\gamma_j)_{j=0}^\omega\in\Gamma$,
which are interrelated in such way that each of these four objects uniquely
determines the other three ones. For that reason, if one of the four objects is given, we will call the three others
associated with it.

Let $f\in\Pol$ be given by (\ref{s0-1}). Then we consider the Riesz projection $P_+f$ which is defined by
$$ (P_+f)(t):=\sum_{k\in I\cap\N_0}a_kt^k, \qquad t\in\T. $$
Let $\Pol_+:=\Lin\{t^k:k\in\N_0\}$ and $\Pol_-:=\Lin\{t^{-k}:k\in\N\}$ where $\N$ is the set of all positive integers.
Then, clearly, $P_+$ is the projection which projects the linear space $\Pol$ onto $\Pol_+$ parallel to $\Pol_-$.

In view of a result due to Fatou (see, e.g., \cite[Theorem 1.18]{RR}), we will use the following notation:
If $h\in H^2(\D)$, then the symbol $\ul h$ stands for the radial boundary values of $h$, which exist for $m$-a.e. $t\in\T$ where $m$ is the normalized Lebesgue measure on $\T$.
If $z\in\C$, then the symbol $z^*$ stands for the complex conjugate of $z$.

\begin{thm}\label{s0-t1}
Let $\mu\in\cM_+^1(\T)$. Then the following statements are equivalent:
\begin{enumerate}
    \item[\textnormal{(i)}] $\mu$ is a Helson-Szeg\H o measure.
        \item[\textnormal{(ii)}] The Riesz projection $P_+$ is bounded in $L_\mu^2$.
            \item[\textnormal{(iii)}] The sequence $(t^n)_{n\in\Z}$ is a (symmetric or nonsymmetric) basis of $L_\mu^2$.
                \item[\textnormal{(iv)}] $\mu$ is absolutely continuous with respect to $m$ and there is an outer function
$h\in H^2(\D)$ such that $\frac{d\mu}{dm}=|\ul h|^2$ and 
$$ {\rm dist\,}\big(\ul h^*\!/\ul h,H^\infty(\T)\big)<1. $$
\end{enumerate}
\end{thm}

\begin{cor}\label{s0-c1}
Let $\mu\in\cM_+^1(\T)$ be a Helson-Szeg\H o measure. 
Then the Schur parameter sequence $\gamma=(\gamma_j)_{j=0}^\omega$ associated 
with $\mu$ is infinite, i.e., $\omega=\infty$ holds, and the series $\sum_{j=0}^\infty|\gamma_j|^2$ converges, i.e. 
$\gamma\in l_2$.
\end{cor}

\begin{proof}
Let $\theta\in\cS $ be the Schur function associated with $\mu$. Then it is known (see Remark~\ref{re2.8}) that
\begin{align}\label{s0-c1-1}
\prod_{j=0}^\omega(1-|\gamma_j|^2)&=\exp\Big\{\int_\T\ln(1-|\ul\theta(t)|^2)m(dt)\Big\}.
\end{align}
We denote by $\Phi$ the function from $\cC^0 $ associated with
$\mu$. Using (\ref{s0-4}), assumption (iv) in Theorem \ref{s0-t1}, and Fatou's theorem we obtain
\begin{align}\label{s0-c1-2}
1-|\ul\theta(t)|^2 = \frac{4\Re\ul\Phi(t)}{|\ul\Phi(t)+1|^2} = \frac{4|\ul h(t)|^2}{|\ul\Phi(t)+1|^2} 
\end{align}
for $m$-a.e. $t\in\T$. Thus,
\begin{align}\label{s0-c1-3}
\ln(1-|\ul\theta(t)|^2) = \ln 4+2\ln|\ul h(t)|-2\ln|\ul\Phi(t)+1|. 
\end{align}
In view of $\Re\Phi(\zeta)>0$ for each $\zeta\in\D$, the function $\Phi$ is outer. Hence, $\ln|\ul\Phi+1|\in L_m^1$ (see, e.g. 
\cite[Theorem 4.29 and Theorem 4.10]{3}). Taking into account condition (iv) in Theorem \ref{s0-t1},
we obtain $h\in H^2(\D)$. Thus,
we infer from (\ref{s0-c1-3}) that $\ln(1-|\ul\theta(t)|^2)\in L_m^1$. Now the assertion follows
from (\ref{s0-c1-1}).
\end{proof}

\begin{rem}\label{s0-r1}
Let $\mu\in\cM_+^1(\T)$, and let the Lebesgue decomposition of $\mu$ with respect to $m$ be given by
\begin{align}\label{s0-r1-1}
\mu(dt)&=v(t)m(dt)+\mu_s(dt),
\end{align}
where $\mu_s$ stands for the singular part of $\mu$ with repect to $m$. Then the relation $\Re\ul\Phi=v$ holds
$m$-a.e. on $\T$. The identity (\ref{s0-c1-3}) has now the form
$$ \ln(1-|\ul\theta(t)|^2) = \ln 4+\ln v(t)-2\ln|\ul\Phi(t)+1| $$
for $m$-a.e. $t\in\T$. From this and (\ref{s0-c1-1}) now it follows a well-known result, namely, that $\ln v\in L_m^1$
(i.e., $\mu$ is a Szeg\H o measure) if and only if $\omega=\infty$ and $\gamma\in l_2$.
In particular, a Helson-Szeg\H o measure is also a Szeg\H o measure.
\end{rem}

We first wish to mention that our interest in describing the class of Helson-Szeg\H o measures in terms of Schur parameters was 
initiated by conversations with L. B. Golinskii and A. Ya. Kheifets
who studied related questions in joint research with F. Peherstorfer and P. M. Yuditskii (see \cite{GKPY}, \cite{KPY}, \cite{PVY}). In Section \ref{s6} we will comment on some results in \cite{GKPY} which are similar to our own.
The above-mentioned problem is of particular interest, even on its own. Solutions to this problem promise
important applications and new results in scattering theory for CMV matrices (see
\cite{GKPY}, \cite{KPY}, \cite{5}) and in nonlinear Fourier analysis (see \cite{14A}).
Our approach to the description of Helson-Szeg\H o measures differs from the one in \cite{GKPY} in that we investigate
this question for CMV matrices in another basis
(see \cite[Definition 2.2., Theorem 2.13]{D06}), 
namely the one for that CMV matrices have the full GGT representation
(see Simon \cite[p. 261--262, Remarks and Historical Notes]{S}).

\subsection{A unitary colligation associated with a Borelian probability measeure on the unit circle}
\label{s2}
Let $\te\in\cS$. Assume that
\begin{equation}\label{Nr.d1.5b}
\Dl := (\gH, \C, \C; T, F, G, S)
\end{equation}
is a simple unitary colligation satisfying $\te=\te_\Dl$.
In the considered case is $\cF=\cG=\C$.

\begin{prop}\label{s2-p1}
Let $\mu\in\cM_+^1(\T)$ be a Szeg\H o measure. Let $\Theta$ be the function belonging to $\cS $ which is associated with $\mu$ and let $\Dl$ be the simple unitary colligation the characteristic operator function of which coincides with $\Theta$.
Then the spaces \textnormal{(}see \eqref{eq: 1.7}--\eqref{1.201}\textnormal{)}
\begin{align}\label{s2-p1-1}
& \gH_\gF^\bot:=\gH\ominus\gH_\gF, \qquad \gH_\gG^\bot:=\gH\ominus\gH_\gG 
\end{align}
are nontrivial.
\end{prop}

\begin{proof}
Let $\gamma=(\gamma_j)_{j=0}^\omega\in\Gamma$ be the Schur parameter sequence of $\Theta$. Then from Corollary \ref{s0-c1}
we infer that $\omega=\infty$ and that $\gamma\in l_2$.
In this case it follows from \rthm{th2.1} that both spaces \eqref{s2-p1-1} are nontrivial.
\end{proof}

Because of \eqref{eq: 1.7} (for $\cF=\cG=\C$) and \eqref{s2-p1-1} it follows that the subspace $\gH_\gG^\perp$ 
$($resp. $\gH_\gF^\perp)$ is invariant with respect to $T$ $($resp. $T^*)$.
It follows from \rthm{th1.3} that
$$  V_T:= {\rm Rstr}._{\gH_\gG^\perp}T \quad\mbox{and}\quad  V_{T^*}:= {\rm Rstr.}_{\gH_\gF^\perp}T^* $$
are 
exactly the largest shifts contained in $T$ and $T^*$, respectively (see Definition~\ref{D1.9-1207}).

Let $\mu\in\cM^1_+(\T)$. Then our subsequent
considerations are concerned with the investigation of the unitary
operator $U_\mu^\times:L_\mu^2\to L_\mu^2$ which is defined for each $f\in L_\mu^2$ by
\begin{align}\label{Nr.2.0}
(U_\mu^\times f) (t) &:= \ko{t} \cdot f(t),
\qquad t\in\mathbb T.
\end{align}
Denote by $\tau$ the embedding operator of $\C$ into $L_\mu^2$, i.e., $\tau :\C\to L_\mu^2$
is such that for each $c\in\C$ the image $\tau (c)$ of $c$ is the constant function on $\T$
with value $c$. Denote by $\mathbb C_\T$ the subspace of $L_\mu^2$ which is generated by
the constant functions and denote by ${\bf 1}$ the constant function on $\T$ with value $1$.
Then obviously $\tau (\C) =\C_\T$ and $\tau (1) = {\bf 1}$.

We consider the subspace
$$ \gH_\mu := L_\mu^2 \ominus \mathbb C_\mathbb T. $$
Denote by
$$ U_\mu^\times = \left(\begin{matrix} T^\times & F^\times\cr G^\times & S^\times  \end{matrix}\right) $$
the block representation of the operator $U_\mu^\times$ with respect to the orthogonal decomposition
$L_\mu^2 = \gH_\mu\oplus \mathbb C_\mathbb T$. 
Then we obtain from \rthm{th1.8} the following result.

\begin{thm}\label{T8.6-20221201} %
Let $\mu\in\cM_+^1 (\mathbb T)$. Define $T_\mu := T^\times$,\; $F_\mu := F^\times \tau$,\;
$G_\mu := \tau^\ast G^\times$,\;and $S_\mu := \tau^\ast S^\times\tau$. Then
\begin{align}\label{Nr.2.2}
& \Dl_\mu := ( \gH_\mu,\mathbb C, \mathbb C;  T_\mu, F_\mu, G_\mu, S_\mu)
\end{align}
is a simple unitary colligation the characteristic function $\Theta_{\Dl_\mu}$ of which
coincides with the Schur function $\Theta$ associated with $\mu$.
\end{thm}

In view of Theorem~\ref{T8.6-20221201}, the operator $T_\mu$ is a completely nonunitary contraction and if the function $\Phi$ is given by 
(\ref{s0-4}) with $\beta=0$, then from (\ref{Nr.0.1f}) it follows
$$ \zeta\Theta_{\Dl_\mu} (\zeta) = \frac{\Phi (\zeta) -1}{\Phi (\zeta) + 1},
\quad\zeta\in\mathbb D. $$

\begin{defn}
Let $\mu\in\cM_+^1 (\T)$. Then the simple unitary colligation given by \eqref{Nr.2.2} is called 
\emph{the unitary colligation associated with} $\mu$.
\end{defn}

Let $\mu\in\cM_+^1(\T)$ be a Szeg\H o measure and let $\gamma=(\gamma_j)_{j=0}^\omega\in\Gamma$ be the Schur parameter 
sequence associated with $\mu$. Then Remark \ref{s0-r1} shows that $\omega=\infty$ and $\gamma\in l_2$.
Furthermore, we use for all integers $n$ the setting $e_n(t):\mathbb T\to L^2_\mu$ defined by
\begin{align}\label{Nr.EK}
& e_n (t) := t^n .
\end{align}
Thus, we have $e_{-n} = (U_\mu^\times)^n {\bf 1}$,
where $U_\mu^\times$ is the operator defined by (\ref{Nr.2.0}). 
We consider then the system $\{e_0, e_{-1}, e_{-2},\dotsc\}$. 
By the Gram-Schmidt orthogonalization method in the space $L_\mu^2$ we get a unique sequence
$(\varphi_n)_{n=0}^\infty$ of polynomials, where (see Subsection~\ref{sec2.1})
\begin{align}\label{Nr.3.1}
& \varphi_n (t) = \alpha_{n,n} t^{-n} + \alpha_{n,n-1} t^{-(n-1)} + \cdots + \alpha_{n,0},
\; t\in\mathbb T,\quad n\in \N_0,
\end{align}
such that the conditions
\begin{align}\label{Nr.3.2}
&\bigvee_{k=0}^n \varphi_k = \bigvee_{k=0}^n (U_\mu^\times)^k {\bf 1},
\quad \bigl((U_\mu^\times)^n {\bf 1}, \varphi_n\bigr)_{L_\mu^2} > 0, 
\quad n\in \N_0,
\end{align}
are satisfied. 
We note that the second condition in (\ref{Nr.3.2}) is equivalent to \linebreak
$\bigl( {\bf 1}, \varphi_0 \bigr)_{L_\mu^2} > 0$ and 
\begin{align}\label{Nr.3.3}
& \bigl(U_\mu^\times \varphi_{n-1},\varphi_n\bigr)_{L_\mu^2} > 0, 
\quad n\in \N_0.
\end{align}
In particular, since $\mu (\mathbb T)=1$ holds, from the construction of $\varphi_0$ we see that
\begin{align}\label{Nr.4.11}
& \varphi_0 = \boldsymbol{1}.
\end{align}
We consider a simple unitary colligation $\Dl_\mu$ of the type \eqref{Nr.2.2} associated with the measure $\mu$. 
The spaces 
\eqref{eq: 1.7} associated with the unitary colligation $\Dl_\mu$ 
have the forms
\begin{align}\label{s2-17} 
& \gH_{\mu,\gF}=\bigvee_{n=0}^\infty T_\mu^nF_\mu(\C) \quad\mbox{and}\quad
\gH_{\mu,\gG}=\bigvee_{n=0}^\infty(T_\mu^*)^n G_\mu^*(\C) ,
\end{align}
respectively.
Let the sequence of functions $(\varphi'_k)_{k=1}^\infty$ be defined by
\begin{align}\label{Nr.3.9}
& \varphi'_k := T_\mu^{k-1}  F_\mu (1), \quad k\in\N.
\end{align}
In view of the formulas
\begin{align}\label{Nr.3.8}
& \bigvee_{k=0}^n (U_\mu^\times)^k {\bf 1} 
= \left( \bigvee_{k=0}^{n-1} (T_\mu)^k F_\mu (1)\right) \oplus \mathbb C_\mathbb T, 
\quad n\in\N,
\end{align}
it can be seen that the sequence $(\varphi_k)_{k=1}^\infty$ can
be obtained by applying the Gram-Schmidt orthonormalization procedure to 
$(\varphi'_k)_{k=1}^\infty$ with additional consideration of the normalization 
condition (\ref{Nr.3.3}). Thus, we obtain the following result:

\begin{thm}\label{s2-t2}
The system $(\varphi_k)_{k=1}^\infty$ of orthonormal polynomials is a basis in the space 
$\gH_{\mu, \gF}$, and 
\begin{align}\label{s2-t2-1}
& \gH_{\mu,\gF}=\left(\bigvee_{k=0}^\infty \ko{t}^k\right)\ominus\C_\T.
\end{align}
This system can be obtained in the result of the application of the Gram-Schmidt
orthogonalization procedure to the sequence \eqref{Nr.3.9} taking into account the normalization 
condition \eqref{Nr.3.3}.
\end{thm}

\begin{rem}\label{s2-r2}
Analogously to \eqref{s2-t2-1} we have the equation
\begin{align}\label{s2-r2-1}
& \gH_{\mu,\gG}=\left(\bigvee_{k=0}^\infty t^k\right)\ominus\C_\T.
\end{align}
\end{rem}

In view of \eqref{s2-p1-1}, let
$$ \gH_{\mu,\gF}^\perp := \gH_\mu\ominus \gH_{\mu,\gF} , \qquad \gH_{\mu,\gG}^\perp := \gH_\mu\ominus \gH_{\mu,\gG}. $$ 
If $\mu$ is a Szeg\H o measure, then we have
\begin{align}\label{s2-15}
L_\mu^2\ominus\bigvee_{k=0}^\infty e_{-k} 
=\big(\gH_\mu\oplus\C_\T\big)\ominus\big(\gH_{\mu,\gF}\oplus\C_\T\big)
= \gH_\mu\ominus \gH_{\mu,\gF} = \gH_{\mu,\gF}^\perp \ne\{0\}.
\end{align}
So from Proposition \ref{s2-p1} we obtain the known result that in this case the system $(\varphi_n)_{n=0}^\infty$ is not complete in the space $L_\mu^2$.

In our case 
$$  V_{T_\mu}:= {\rm Rstr.}_{\gH_{\mu,\gF}^\perp}T_\mu \qquad
\big(\mbox{resp.}\  V_{T_\mu^*}:= {\rm Rstr.}_{\gH_{\mu,\gF}^\perp}T_\mu^* \big) $$
is the maximal unilateral shift contained in $T_\mu$ (resp. $T_\mu^*$) 
(see Theorem~\ref{th1.3}).
In view of $\delta_{T_\mu} = \delta_{T_\mu^*} =1$ 
the multiplicity of the unilateral shift $V_{T_\mu}$ (resp. $V_{T_\mu^*}$) is equal to $1$.

\begin{prop}\label{3.4}
The orthonormal system of the polynomials $(\varphi_k)_{k=0}^\infty$ is noncomplete in $L_\mu^2$
if and only if the contraction $T_\mu$ $($resp. $T_\mu^*)$ contains a maximal unilateral 
shift $V_{T_\mu}$ $($resp. $V_{T_\mu^*})$ of multiplicity $1$.
\end{prop}

\subsection{On the connection between the Riesz projection
    $\fEl{P}{+}$ and the projection $\mkEl{ \cP^{\gF} }{\mu}{\gG}$
    which projects $\fEl{\gH}{\mu}$ onto $\mkEl{\gH}{\mu}{\gG}$
    parallel to $\mkEl{\gH}{\mu}{\gF}$}
\label{s3}

Let $\mu \in \rklamFunk{ \fEl{\cM^{1}}{+} }{ \dT }$.
We consider the unitary colligation $\fEl{\Delta}{\mu}$
of type \eqref{Nr.2.2} which is associated with the measure $\mu$.
Then the following statement is true.

\begin{thm}[\zitaa{DFK}{\cthm{3.1}}]\label{3-thm1}
    Let $\mu \in \rklamFunk{ \fEl{\cM^{1}}{+} }{ \dT }$ be
    a Szeg\H{o} measure. Then the Riesz projection $\fEl{P}{+}$
    is bounded in $\fEl{L^2}{\mu}$ if and only if the projection
    $\mkEl{ \cP^{\gF} }{\mu}{\gG}$ which projects
    $\fEl{\gH}{\mu}$ onto $\mkEl{\gH}{\mu}{\gG}$
    parallel to $\mkEl{\gH}{\mu}{\gF}$ is bounded.
\end{thm}

\begin{proof}
    For each $n \in \N_0$ we consider particular subspaces of the space $\gH_\mu$, namely
    \begin{align}       \label{3-thm1-1}
        \mkEl{ \iEl{\gH}{n} }{\mu}{\gF}
        :=  \bigvee_{k = 0}^{n}
        { 
            \fEl{T^k}{\mu}
            \rklamFunk{ \fEl{F}{\mu} }{\dC} 
        } \qKomma
        \qquad
        \qquad
        \mkEl{ \iEl{\gH}{n} }{\mu}{\gG}
        :=  \bigvee_{k = 0}^{n}
        { 
            \rklam{ \fEl{ \Adj{T} }{\mu} }^k
            \rklamFunk{ \fEl{ \Adj{G} }{\mu} }{\dC} 
        } \qKomma
    \end{align}
    and
    \begin{align}       \label{3-thm1-2}
        \fEl{ \iEl{\gH}{n} }{\mu}
        :=  \mkEl{ \iEl{\gH}{n} }{\mu}{\gF}
        \vee
        \mkEl{ \iEl{\gH}{n} }{\mu}{\gG} \qKomma
        \qquad
        \qquad
        \mkEl{L}{\mu}{n}
        :=  \fEl{ \iEl{\gH}{n} }{\mu}
        \oplus
        \fEl{\dC}{\dT} \qPunkt
    \end{align}
    Then from \eqref{Nr.3.9}, \eqref{s2-t2-1}, and \eqref{s2-r2-1} %
    we obtain the relations
    \begin{align}       \label{3-thm1-3}
        \mkEl{\gH}{\mu}{\gF}
        =  \bigvee_{n = 0}^{\infty}
        { 
            \mkEl{ \iEl{\gH}{n} }{\mu}{\gF}
        } \qKomma
        \qquad
        \qquad
        \mkEl{\gH}{\mu}{\gG}
        =  \bigvee_{n = 0}^{\infty}
        { 
            \mkEl{ \iEl{\gH}{n} }{\mu}{\gG}
        } \qKomma
    \end{align}
    \begin{align}       \label{3-thm1-4}
        \fEl{\gH}{\mu}
        =  \bigvee_{n = 0}^{\infty}
        { 
            \fEl{ \iEl{\gH}{n} }{\mu}
        } \qKomma
        \qquad
        \qquad
        \fEl{L^2}{\mu}
        =  \bigvee_{n = 0}^{\infty}
        { 
            \mkEl{L}{\mu}{n}
        } \qKomma
    \end{align}
    \begin{align}       \label{3-thm1-5}
        \mkEl{ \iEl{\gH}{n} }{\mu}{\gF}
        =  \rklam{
            \bigvee_{k = 0}^{n}
            { 
                \ko{t}^k
            }
        }
        \ominus
        \fEl{\dC}{\dT} \qKomma
        \qquad
        \qquad
        \mkEl{ \iEl{\gH}{n} }{\mu}{\gG}
        =  \rklam{
            \bigvee_{k = 0}^{n}
            { 
                t^k
            }
        }
        \ominus
        \fEl{\dC}{\dT}\,,
    \end{align}
    and
    \begin{align}       \label{3-thm1-6}
        \fEl{ \iEl{\gH}{n} }{\mu}
        =  \rklam{
            \bigvee_{k = -n}^{n}
            { 
                t^{k}
            }
        }
        \ominus
        \fEl{\dC}{\dT} \qKomma
        \qquad
        \qquad
        \mkEl{L}{\mu}{n}
        =   \bigvee_{k = -n}^{n}
        { 
            t^{k}
        } .
    \end{align}
    Since $\mu$ is a Szeg\H{o} measure, for each $n \in \N_0$
    we obtain
    \begin{align}       \label{3-thm1-7}
        \fEl{ \iEl{\gH}{n} }{\mu}
        = \mkEl{ \iEl{\gH}{n} }{\mu}{\gF}
        \quad \dot{+} \quad
        \mkEl{ \iEl{\gH}{n} }{\mu}{\gG}        \qPunkt
    \end{align}
    Suppose now that the Riesz projection $\fEl{P}{+}$ is bounded
    in $\fEl{L^2}{\mu}$. Let $h \in \fEl{ \iEl{\gH}{n} }{\mu}$.
    Then, because of \fref{3-thm1-6}, the function $h$ has the form
    \begin{align}       \label{3-thm1-8}
        \rklamFunk{h}{t}
        = \sum_{k = -n}^{n}{ \fEl{a}{k} t^k }        \,,\qquad t\in\T.
    \end{align}
    From \fref{3-thm1-5} and \fref{3-thm1-7} we obtain
    \begin{align}       \label{3-thm1-9}
        h = \fEl{h}{\gF} + \fEl{h}{\gG}   \qKomma
    \end{align}
    where
    \begin{align*}
        \fEl{h}{\gF}  \in \mkEl{ \iEl{\gH}{n} }{\mu}{\gF} \qKomma
        \qquad
        \qquad
        \fEl{h}{\gG}  \in \mkEl{ \iEl{\gH}{n} }{\mu}{\gG} \qKomma
    \end{align*}
    \begin{align}       \label{3-thm1-10}
        \rklamFunk{ \fEl{h}{\gF} }{t}
        = \mkEl{a}{0}{\gF} 
        + \sum_{k = 1}^{n}
        {
            \fEl{a}{-k} t^{-k}
        }        \qKomma
        \qquad
        \qquad
        \rklamFunk{ \fEl{h}{\gG} }{t}
        = \mkEl{a}{0}{\gG} 
        + \sum_{k = 1}^{n}
        {
            \fEl{a}{k} t^{k}
        }\,,
    \end{align}
    and
    \begin{align*}
        \fEl{a}{0} 
        = \mkEl{a}{0}{\gF} +  \mkEl{a}{0}{\gG}  \qPunkt
    \end{align*}
    Observe that
    \begin{align}       \label{3-thm1-11}
        \mkEl{ \cP^{\gF} }{\mu}{\gG}  h
        = \fEl{h}{\gG}  \qPunkt
    \end{align}
    On the other hand, we have
    \begin{align}       \label{3-thm1-12}
        h = \fEl{h}{+}  +  \fEl{h}{-} \qKomma
    \end{align}
    where, for each $t\in\T$,
    \begin{align}       \label{3-thm1-13}
        \rklamFunk{ \fEl{h}{+} }{t}
        = \rklamFunk{ \rklam{ \fEl{P}{+}(h) } }{t}
        = \sum_{k = 0}^{n}
        {
            \fEl{a}{k} t^{k}
        }        \qKomma
        \qquad
        \qquad
        \rklamFunk{ \fEl{h}{-} }{t}
        = \sum_{k = 1}^{n}
        {
            \fEl{a}{-k} t^{-k}
        } \qPunkt
    \end{align}
    For a polynomial $\fEl{h}{\gF}$ of the type \fref{3-thm1-10}
    we set
    \begin{align}       \label{3-thm1-14}
        \rklamFunk{ \fEl{h}{\gF} }{0}
        := \mkEl{a}{0}{\gF}    \qPunkt
    \end{align}
    Then from \fref{3-thm1-10}--\fref{3-thm1-14} we infer
    \begin{align}       \label{3-thm1-15}
        \fEl{P}{+}  h
        = \fEl{h}{+}
        = \fEl{h}{\gG}
        + \rklamFunk{ \fEl{h}{\gF} }{0} \cdot  \dEins
        = \mkEl{ \cP^{\gF} }{\mu}{\gG} h
        + \rklamFunk{ \fEl{h}{\gF} }{0} \cdot  \dEins  \qPunkt
    \end{align}
    Observe that
    \begin{align}       \label{3-thm1-16}
        \fEl{h}{-}
        = \fEl{h}{\gF}
        - \rklamFunk{ \fEl{h}{\gF} }{0} \cdot  \dEins \qKomma
    \end{align}
    where in view of \fref{3-thm1-5} we see that
    \begin{align*}
        \fEl{h}{\gF}  \perp \dEins \qPunkt
    \end{align*}
    Let $\fEl{P}{ \fEl{\dC}{\dT} }$ be the orthoprojection from
    $\fEl{L^2}{\mu}$ onto $\fEl{\dC}{\dT}$. Then from \fref{3-thm1-16}
    it follows that
    \begin{align*}
        \rklamFunk{ \fEl{h}{\gF} }{0} \cdot \dEins
        = \fEl{P}{ \fEl{\dC}{\dT} }  \fEl{h}{-}
        = \fEl{P}{ \fEl{\dC}{\dT} }
        \rklam{
            I - \fEl{P}{+}
        } h   \qPunkt
    \end{align*}
    Inserting this expression into \fref{3-thm1-15} we get
    \begin{align*}
        \mkEl{ \cP^{\gF} }{\mu}{\gG} h
        = \fEl{P}{+}  h
        - \fEl{P}{ \fEl{\dC}{\dT} }
        \rklam{
            I - \fEl{P}{+}
        } h  \qPunkt
    \end{align*}
    From this and \fref{3-thm1-4} it follows that the boundedness
    of the projection $\fEl{P}{+}$ in $\fEl{L^2}{\mu}$ implies
    the boundedness of the projection $\mkEl{ \cP^{\gF} }{\mu}{\gG}$
    in $\fEl{\gH}{\mu}$. 

    Conversely, suppose that the projection
    $\mkEl{ \cP^{\gF} }{\mu}{\gG}$ is bounded in $\fEl{\gH}{\mu}$.
    Let $f \in \mkEl{L}{\mu}{n}$. We denote by
    $\fEl{P}{ \fEl{ \iEl{\gH}{n} }{\mu} }$ the orthogonal projection
    from $\fEl{L^2}{\mu}$ onto $\fEl{ \iEl{\gH}{n} }{\mu}$. We set
    \begin{align*}
        h := \fEl{P}{ \fEl{ \iEl{\gH}{n} }{\mu} } f
    \end{align*}
    and use for $h$ the notations introduced
    in \fref{3-thm1-8}--\fref{3-thm1-10}. Let
    \begin{align*}
        f = \fEl{f}{+} + \fEl{f}{-}   \qKomma
    \end{align*}
    where $\fEl{f}{+} :=  \fEl{P}{+}  f$. Then
    \begin{align*}
        f =  \fEl{P}{ \fEl{\dC}{\dT} }  f + h
        =  \fEl{P}{ \fEl{\dC}{\dT} }  f
        + \fEl{h}{\gG}  + \fEl{h}{\gF} \qPunkt
    \end{align*}
    This implies
    \begin{align*}
        \fEl{P}{+}  f
        =  \fEl{P}{ \fEl{\dC}{\dT} }  f
        + \fEl{h}{\gG}
        + \rklamFunk{ \fEl{h}{\gF} }{0} \cdot \dEins \qPunkt
    \end{align*}
    This means
    \begin{align}       \label{3-thm1-17}
        \fEl{P}{+}  f
        =  \fEl{P}{ \fEl{\dC}{\dT} }  f
        + \mkEl{ \cP^{\gF} }{\mu}{\gG}
        \fEl{P}{ \fEl{ \iEl{\gH}{n} }{\mu} }  f
        + \rklamFunk{ \fEl{h}{\gF} }{0} \cdot \dEins \qPunkt
    \end{align}
    The mapping
    $\fEl{h}{\gF} \mapsto \rklamFunk{ \fEl{h}{\gF} }{0}$
    is a linear functional on the set 
    $${\mathcal Pol}_{\le0}:=\Lin\{t^{-k}:k\in\N_0\}.$$
    Since $\mu$ is a Szeg\H{o} measure, the Szeg\H{o}-Kolmogorov-Krein
    Theorem (see, e.g. \cite[Theorem 4.31]{3}) implies that this functional
    is bounded in $\fEl{L^2}{\mu}$ on the set ${\mathcal Pol}_{\le0}$. Thus,
    there exists a constant $C \in \intervalO{0}{+\infty}$ such that
    \begin{align}       \label{3-thm1-18}
        \inorm{
            \rklamFunk{ \fEl{h}{\gF} }{0} \cdot \dEins
        }{ \fEl{L^2}{\mu} }
        &=    \abs{ \rklamFunk{ \fEl{h}{\gF} }{0} } \nonumber \\[4pt]
        &\leq C \cdot \norm{ \fEl{h}{\gF} } \nonumber \\[4pt]
        &= C
        \cdot 
        \norm{ 
            \rklam{ I - \mkEl{ \cP^{\gF} }{\mu}{\gG}  } 
            h
        }    \nonumber \\[4pt]
        &\leq C
        \cdot 
        \norm{  I - \mkEl{ \cP^{\gF} }{\mu}{\gG}  }
        \cdot
        \norm{h}    \nonumber \\[4pt]
        &\leq C
        \cdot 
        \norm{  I - \mkEl{ \cP^{\gF} }{\mu}{\gG}  }
        \cdot
        \norm{f}    \qPunkt
    \end{align}
    From \fref{3-thm1-17} and \fref{3-thm1-18} we get
    \begin{align*}
        \norm{ \fEl{P}{+} f }
        &\leq 
        \norm{ \fEl{P}{ \fEl{\dC}{\dT} } f }
        + \norm{
            \mkEl{ \cP^{\gF} }{\mu}{\gG}
            \fEl{P}{ \fEl{ \iEl{\gH}{n} }{\mu} }  f
        }
        + C \cdot \norm{ I - \mkEl{ \cP^{\gF} }{\mu}{\gG} }
        \cdot \norm{f} \nonumber \\[4pt]
        &\leq
        \norm{f} 
        +  \norm{ \mkEl{ \cP^{\gF} }{\mu}{\gG} } \cdot \norm{f} 
        + C \cdot \norm{ I - \mkEl{ \cP^{\gF} }{\mu}{\gG} }
        \cdot \norm{f}  \qPunkt
    \end{align*}
    Now considering the limit as $n \rightarrow \infty$
    and taking into account \fref{3-thm1-4}, we see that the
    boundedness of the projection $\mkEl{ \cP^{\gF} }{\mu}{\gG}$
    implies the boundedness of the Riesz projection $\fEl{P}{+}$
    in $\fEl{L^2}{\mu}$.
\end{proof}
\subsection{On the connection of the Riesz projection
    $\fEl{P}{+}$ with the orthogonal projections from
    $\fEl{\gH}{\mu}$ onto the subspaces
    $\mkEl{ \orth{\gH} }{\mu}{\gF}$
    and $\mkEl{ \orth{\gH} }{\mu}{\gG}$}
\label{s4}

Let $\mu \in \rklamFunk{ \fEl{\cM^{1}}{+} }{ \dT }$ be
a Szeg\H{o} measure. As we did earlier, we consider the simple
unitary colligation $\fEl{\Delta}{\mu}$ of type \fref{Nr.2.2} which is
associated with the measure $\mu$. As was previously mentioned, we
then have
\begin{align*}
    \mkEl{ \orth{\gH} }{\mu}{\gF} \neq  \gklam{0}
    \qquad
    \qquad
    \mbox{and}
    \qquad
    \qquad
    \mkEl{ \orth{\gH} }{\mu}{\gG} \neq  \gklam{0} \qPunkt
\end{align*}
We denote by $\fEl{P}{ \mkEl{ \orth{\gH} }{\mu}{\gF} }$
and $\fEl{P}{ \mkEl{ \orth{\gH} }{\mu}{\gG} }$ the
orthogonal projections from $\fEl{\gH}{\mu}$ onto
$\mkEl{ \orth{\gH} }{\mu}{\gF}$ and $\mkEl{ \orth{\gH} }{\mu}{\gG}$,
respectively.

Let $h \in \fEl{ \iEl{\gH}{n} }{\mu}$. Along with the decomposition
\fref{3-thm1-9} we consider the decomposition
\begin{align}       \label{4-1}
    h = \fEl{ \wt{h} }{\gF} + \fEl{ \orth{ \wt{h} } }{\gF}  \qKomma
\end{align}
where
\begin{align*}
    \wt h_\gF\in\gH_{\mu,\gF}^{(n)} 
    \qquad\qquad\mbox{and}\qquad\qquad
    \fEl{ \orth{ \wt{h} } }{\gF}
    \in \fEl{ \iEl{\gH}{n} }{\mu} 
    \ominus 
    \mkEl{ \iEl{\gH}{n} }{\mu}{\gF} \qPunkt 
\end{align*}
>From the shape \fref{3-thm1-5} of the
subspace $\mkEl{ \iEl{\gH}{n} }{\mu}{\gF}$
and the polynomial structure of the orthonormal basis
$\fKlam{\varphi}{n}{0}{\infty}$ of the subspace
$\mkEl{\gH}{\mu}{\gF}$, it follows that
\begin{align*}
    \fEl{ \orth{ \wt{h} } }{\gF}
    = \fEl{P}{ \mkEl{ \orth{\gH} }{\mu}{\gF} }  h \qPunkt
\end{align*}
Since $\fEl{h}{\gF}$ (see \fref{3-thm1-9})
and $\fEl{ \wt{h} }{\gF}$ belong to $\mkEl{ \iEl{\gH}{n} }{\mu}{\gF}$,
we get
\begin{align}       \label{4-2}
    \fEl{ \orth{ \wt{h} } }{\gF}
    = \fEl{P}{ \mkEl{ \orth{\gH} }{\mu}{\gF} }  h
    = \fEl{P}{ \mkEl{ \orth{\gH} }{\mu}{\gF} }  \mkEl{h}{\gG}
    = \fEl{P}{ \mkEl{ \orth{\gH} }{\mu}{\gF} }
    \mkEl{ \cP^{\gF} }{\mu}{\gG} h
    = \mkEl{B}{\mu}{\gF}
    \mkEl{ \cP^{\gF} }{\mu}{\gG} h  \qKomma
\end{align}
where
\begin{align}       \label{4-3}
    B_{\mu,\gF}:=\Rstr_{\gH_{\mu,\gG}}P_{\gH_{\mu,\gF}^\perp}:\gH_{\mu,\gG}\to\gH_{\mu,\gF}^\perp\,,
\end{align}
i.e., we consider $\mkEl{B}{\mu}{\gF}$ as an operator acting between
the spaces $\mkEl{\gH}{\mu}{\gG}$ and $\mkEl{ \orth{\gH} }{\mu}{\gF}$.

\begin{thm}[\zitaa{DFK}{\cthm{4.1}}]\label{4-thm1}
    Let $\mu \in \rklamFunk{ \fEl{\cM^{1}}{+} }{ \dT }$ be
    a Szeg\H{o} measure. Then the projection $\mkEl{ \cP^{\gF} }{\mu}{\gG}$
    is bounded in $\fEl{\gH}{\mu}$ if and only if the operator
    $\mkEl{B}{\mu}{\gF}$ defined in \fref{4-3} is boundedly invertible.
\end{thm}

\begin{proof}
    Suppose first that $\mkEl{B}{\mu}{\gF}$ has a bounded inverse
    $\mkEl{ \Inv{B} }{\mu}{\gF}$. Then for $h \in \fEl{ \iEl{\gH}{n} }{\mu}$,
    in view of \fref{4-1} and \fref{4-2}, we have
    \begin{align*}
        \mkEl{ \cP^{\gF} }{\mu}{\gG}  h
        = \mkEl{ \Inv{B} }{\mu}{\gF}
        \fEl{P}{ \mkEl{ \orth{\gH} }{\mu}{\gF} }  h \qPunkt
    \end{align*}
    If $n \rightarrow \infty$, this gives us the boundedness of
    the projection $\mkEl{ \cP^{\gF} }{\mu}{\gG}$ in $\fEl{\gH}{\mu}$.

    Conversely, suppose that the projection $\mkEl{ \cP^{\gF} }{\mu}{\gG}$
    is bounded in $\fEl{\gH}{\mu}$.
    If $h \in \fEl{ \iEl{\gH}{n} }{\mu} \ominus \mkEl{ \iEl{\gH}{n} }{\mu}{\gF}$,
    then the decomposition \fref{4-1} provides us
    \begin{align*}
        h = \fEl{ \wt{h} }{\gF}^\bot
    \end{align*}
    and identity \fref{4-2} yields
    \begin{align}       \label{4-4}
        h = \mkEl{ B }{\mu}{\gF}
        \mkEl{ \cP^{\gF} }{\mu}{\gG}  h \qPunkt
    \end{align}
    Since $\mkEl{ \cP^{\gF} }{\mu}{\gG}$ is bounded in $\fEl{\gH}{\mu}$,
    Theorem \ref{3-thm1} implies that the Riesz projection $\fEl{P}{+}$
    is bounded in $\fEl{L^2}{\mu}$. Then it follows from condition (iii)
    in Theorem \ref{s0-t1} that
    \begin{align*}
        \mkEl{\gH}{\mu}{\gF}  \cap  \mkEl{\gH}{\mu}{\gG}  = \gklam{0} \qPunkt
    \end{align*}
    Thus, from the shape \fref{4-3} of the operator $\mkEl{B}{\mu}{\gF}$,
    we infer that
    \begin{align*}
        \mathrm{ker} \, \mkEl{B}{\mu}{\gF} = \gklam{0}
        \qquad
        \qquad
        \mbox{and}
        \qquad
        \qquad
        \rklamFunk{ \mkEl{ \cP^{\gF} }{\mu}{\gG} }
        { \fEl{ \orth{\gH} }{\gF} }
        = \fEl{ \gH }{\gG}  \qPunkt
    \end{align*}
    Now equation \fref{4-4} can be rewritten in the form
    \begin{align*}
        \mkEl{ \Inv{B} }{\mu}{\gF}  h 
        = \mkEl{ \cP^{\gF} }{\mu}{\gG}  h \qKomma
        \qquad
        \qquad
        \qquad
        \qquad
        h
        \in \fEl{ \iEl{\gH}{n} }{\mu} 
        \ominus 
        \mkEl{ \iEl{\gH}{n} }{\mu}{\gF} \qPunkt 
    \end{align*}
    The limit $n \rightarrow \infty$, \eqref{3-thm1-2} and \eqref{3-thm1-4}  give us the desired result.
\end{proof}

The combination of Theorem \ref{4-thm1} with Theorem \ref{3-thm1}
leads us to the following result.

\begin{thm}[\zitaa{DFK}{\cthm{4.2}}]\label{4-thm2}
    Let $\mu \in \rklamFunk{ \fEl{\cM^{1}}{+} }{ \dT }$ be
    a Szeg\H{o} measure. Then the Riesz projection $\fEl{P}{+}$
    is bounded in $\fEl{L^2}{\mu}$ if and only if the operator
    $\mkEl{B}{\mu}{\gF}$ defined in \fref{4-3} is boundedly invertible.
\end{thm}

Let $f \in {\mathcal Pol}$ be given by \fref{s0-1}.
Along with the Riesz projection $\fEl{P}{+}$, we consider
the projection $\fEl{P}{-}$, which is defined by:
\begin{align*}
    \rklamFunk{ \rklam{ \fEl{P}{-} } }{t}
    :=  \sum_{ -k \in I \cap \N_0 }
    { \fEl{a}{k}  t^k } \qKomma
    \qquad  \qquad  t \in \dT   \qPunkt
\end{align*}
Obviously,
\begin{align*}
    \fEl{P}{-}  = \overline{
        \fEl{P}{+}  \overline{f}
    } \qKomma
    \qquad
    \qquad
    \mbox{and}
    \qquad
    \qquad
    \fEl{P}{+}  = \overline{
        \fEl{P}{-}  \overline{f}
    } \qPunkt
\end{align*}
Thus, the boundedness of one of the projections $\fEl{P}{+}$
and $\fEl{P}{-}$ in $\fEl{L^2}{\mu}$ implies the boundedness
of the other one. It is readily checked that the change from the
projection $\fEl{P}{+}$ to $\fEl{P}{-}$ is connected
with changing the roles of the spaces $\mkEl{\gH}{\mu}{\gG}$
and $\mkEl{\gH}{\mu}{\gF}$. Thus we obtain the following result,
which is dual to Theorem \ref{4-thm2}.

\begin{thm}[\zitaa{DFK}{\cthm{4.3}}]\label{4-thm3}
    Let $\mu \in \rklamFunk{ \fEl{\cM^{1}}{+} }{ \dT }$ be
    a Szeg\H{o} measure. Then the Riesz projection $\fEl{P}{+}$
    is bounded in $\fEl{L^2}{\mu}$ if and only if the operator
    $\mkEl{B}{\mu}{\gG} : \mkEl{\gH}{\mu}{\gF} \rightarrow \mkEl{ \orth{\gH} }{\mu}{\gG}$ defined by
    \begin{align}\label{4-thm3-1}
        \mkEl{B}{\mu}{\gG}h
        := {\fEl{P}{ \mkEl{ \orth{\gH} }{\mu}{\gG} }}h
    \end{align}
    is boundedly invertible. Here the symbol $\fEl{P}{ \mkEl{ \orth{\gH} }{\mu}{\gG} }$
    stand for the orthogonal projection from $\fEl{\gH}{\mu}$ onto
    $\mkEl{ \orth{\gH} }{\mu}{\gG}$.
\end{thm}

\subsection{Matrix Representation of the Operator $\mkEl{B}{\mu}{\gG}$
    in Terms of the Schur Parameters Associated With the
    Measure $\mu$}
\label{s5}

Let $\mu \in \rklamFunk{ \fEl{\cM^{1}}{+} }{ \dT }$ be
a Szeg\H{o} measure. We consider the simple unitary
colligation $\fEl{\Delta}{\mu}$ of the type \fref{Nr.2.2} which
is associated with the measure $\mu$. In this case we have
(see Section \ref{s2})
\begin{align*}
    \mkEl{ \orth{\gH} }{\mu}{\gF} \neq  \gklam{0}
    \qquad
    \qquad
    \mbox{and}
    \qquad
    \qquad
    \mkEl{ \orth{\gH} }{\mu}{\gG} \neq  \gklam{0}
\end{align*}
The operator $\mkEl{B}{\mu}{\gG}$ acts between the
subspaces $\mkEl{\gH}{\mu}{\gF}$ and $\mkEl{ \orth{\gH} }{\mu}{\gG}$.
According to the matrix description of the operator
$\mkEl{B}{\mu}{\gG}$ we consider particular orthogonal
bases in these subspaces. In the subspace $\mkEl{ \gH }{\mu}{\gF}$ we have already
considered one such basis, namely the basis consisting
of the trigonometric polynomials $\fKlam{\varphi}{n}{1}{\infty}$
(see Theorem \ref{s2-t2}). Regarding the construction of
an orthonormal basis in $\mkEl{ \orth{\gH} }{\mu}{\gG}$,
we first complete the system $\fKlam{\varphi}{n}{1}{\infty}$
to an orthonormal basis in $\fEl{\gH}{\mu}$. This procedure
is described in more detail in Subsection~\ref{sec2.1}.

We consider the orthogonal decomposition
\begin{align}\label{Nr.67}
    & \gH_\mu = \gH_{\mu,\gF} \oplus \gH_{\mu,\gF}^\perp.
\end{align}
Denote by $\tilde{\gL}_0$ the wandering subspace which generates the subspace associated with the unilateral shift 
$V_{T_\mu^\ast}$. Then (see Proposition \ref{3.4})
$\dim \tilde{\gL}_0 = 1$ and, since $V_{T_\mu^\ast}$ is an isometric operator, we have
\begin{align}\label{Nr.3.10}
    & V_{T_\mu^\ast} = {\rm Rstr.}_{\gH_{\mu,\gF}^\perp} (U_\mu^\times)^\ast.
\end{align}
Consequently,
\begin{align}\label{Nr.3.11}
    & \gH_{\mu,\gF}^\perp 
    = \bigoplus\limits_{n=0}^\infty V_{T_\mu^\ast}^n (\tilde{\gL}_0)
    = \bigvee_{n=0}^\infty (T_\mu^\ast)^n (\tilde{\gL}_0)
    = \bigvee_{n=0}^\infty [ (U_\mu^\times)^\ast]^n (\tilde{\gL}_0).
\end{align}
There exists (see Corollary~\ref{cor1.9}) 
a unique unit function $\psi_1\in\tilde{\gL}_0$ which fulfills
\begin{align}\label{Nr.3.12}
    & \bigl( G_{\mu}^\ast (1), \psi_1 \bigr)_{L_\mu^2} > 0.
\end{align}
Because of (\ref{Nr.3.10}), (\ref{Nr.3.11}), and (\ref{Nr.3.12}) it follows that the
sequence $(\psi_k)_{k=1}^\infty$, where
\begin{align}\label{Nr.3.13}
    & \psi_k := [(U_\mu^\times)^\ast]^{k-1} \psi_1,
    \quad k\in\N,
\end{align}
is the unique orthonormal basis of the space $\gH_{\mu,\gF}^\perp$ which satisfies the conditions
\begin{align}\label{Nr.3.14}
    & \bigl( G_{\mu}^\ast (1), \psi_1 \bigr)_{L_\mu^2} > 0,
    \quad \psi_{k+1} = (U_\mu^\times)^\ast \psi_k,
    \quad k\in \N,
\end{align}
or equivalently
\begin{align}\label{Nr.3.15}
    & \bigl( G_{\mu}^\ast (1), \psi_1 \bigr)_{L_\mu^2} > 0,
    \quad \psi_{k+1} (t) = t^k\cdot \psi_1 (t),
    \; t\in\mathbb T, \quad k\in \N.
\end{align}
According to Definition~\ref{D2.1-20230123} 
we introduce the following notion.

\begin{defn}\label{3.5}
    The constructed orthonormal basis
    \begin{align}\label{Nr.3.16}
        & \varphi_0, \varphi_1, \varphi_2,\dotsc ;\;  \psi_1, \psi_2,\dotsc
    \end{align}
    in the space $L_\mu^2$ which satisfies the conditions \eqref{Nr.3.2} and \eqref{Nr.3.14} is called the
    \emph{cano\-ni\-cal orthonormal basis in} $L_\mu^2$.
\end{defn}

Note that the analytic structure of the system $(\psi_k)_{k=1}^\infty$ is described in the paper \cite{MR2805420}.

Obviously, the canonical orthonormal basis \eqref{Nr.3.16} in $L_\mu^2$ is uniquely determined
by the conditions \eqref{Nr.3.2} and \eqref{Nr.3.14}. 
Here the sequence $(\varphi_k)_{k=0}^\infty$ is an orthonormal system of
polynomials (depending on $\ko{t}$). The orthonormal system $(\psi_k)_{k=1}^\infty$ is
built with the aid of the operator $U_\mu^\times$ from the function
$\psi_1$ $($see \eqref{Nr.3.13}$)$ in a similar way as the system
$(\varphi_k)_{k=0}^\infty$ was built from the function $\varphi_0$
$($see \eqref{Nr.3.1} and \eqref{Nr.3.2}$)$. The only difference is that the system
$\left( [(U_\mu^\times)^\ast]^k \psi_1 \right)_{k=0}^\infty$ is
orthonormal, whereas in the general case the system $\left(
(U_\mu^\times)^k \varphi_0\right)_{k=0}^\infty$ is not
orthonormal. In this respect, the sequence $(\psi_k)_{k=1}^\infty$
can be considered as a natural completion of the system of
orthonormal polynomials $(\varphi_k)_{k=0}^\infty$ to an orthonormal
basis in $L_\mu^2$.

\begin{rem}\label{3.6}
    The orthonormal system
    \begin{align}\label{Nr.3.16b}
        & \varphi_1,\varphi_2,\dotsc ;\;  \psi_1, \psi_2,\dotsc
    \end{align}
    is an orthonormal basis in the space $\gH_\mu$. We will call it
    the \emph{canonical orthonormal basis in} $\gH_\mu$. 
\end{rem}

It is well known (see, e.g., Brodskii \cite{B}) that one can consider simultaneously together with the simple unitary
colligation (\ref{Nr.2.2}) the adjoint unitary colligation
\begin{equation}\label{Nr.3.29}
    \tilde{\bigtriangleup}_\mu := 
    (\gH_\mu, \mathbb C, \mathbb C;T_\mu^\ast, G_\mu^\ast, F_\mu^\ast, S_\mu^\ast)
\end{equation}
which is also simple. Its characteristic function $\Theta_{\tilde{\bigtriangleup}_\mu}$ is for
each $z\in\mathbb D$ given by
\begin{equation*}
    \Theta_{\tilde{\bigtriangleup}_\mu} (z) = \Theta_{\bigtriangleup_\mu}^{\ast} (z^{\ast}).
\end{equation*}
We note that the unitary colligation (\ref{Nr.3.29}) is associated with the operator $(U_\mu^\times)^\ast$.
It can be easily checked that the action of $(U_\mu^\times)^\ast$ is given for each
$ f\in L_\mu^2$ by
\[ [(U_\mu^\times)^\ast f] (t) = t\cdot f(t),\quad t\in\mathbb T . \]
If we replace the operator $U_\mu^\times $ by $(U_\mu^\times)^\ast$ in the preceding considerations,
which have lead to the canonical orthonormal basis (\ref{Nr.3.16}), then we obtain an orthonormal basis
of the space $L_\mu^2$ which consists of two sequences
\begin{equation}
    \label{Nr.ONB}
    (\tilde{\varphi}_j)_{j=0}^\infty 
    \quad\mbox{and}\quad
    (\tilde{\psi}_j)_{j=1}^\infty
\end{equation}
of functions. From our treatments above it follows that the orthonormal basis (\ref{Nr.ONB}) is uniquely
determined by the following conditions:
\begin{enumerate}
    \item[(a)]
    The sequence $(\tilde{\varphi}_j)_{k=0}^\infty$ arises from the result of the Gram-Schmidt orthogonalization
    procedure of the sequence $\left( [(U_\mu^\times)^\ast]^n {\bf 1}\right)_{n=0}^\infty$ and additionally
    taking into account the normalization conditions
    \begin{equation*}
        \bigl( [(U_\mu^\times)^\ast]^n {\bf 1}, \tilde{\varphi}_n \bigr)_{L_\mu^2} > 0,
        \quad n\in \N_0.
    \end{equation*}
    \item[(b)]
    The relations
    \[ \bigl(F_{\mu}(1), \tilde{\psi}_1\bigr)_{L_\mu^2} > 0
    \quad\mbox{and}\quad
    \tilde{\psi}_{k+1} = U_\mu^\times \tilde{\psi}_k,\quad k\in \N, \]
    hold. 
\end{enumerate}
It can be easily checked that
\[ \tilde{\varphi}_{k} = \varphi_k^{\ast}, 
\quad k\in \N_0, 
\vspace{-2mm} \]
and \vspace{-2mm}
\[ \tilde{\psi}_{k} = \psi_k^{\ast},
\quad k\in \N.
\]

According to Definition~\ref{de3.1} 
we introduce the following notion.

\begin{defn}\label{D3.14}
    The orthogonal basis
    \begin{equation}\label{Nr.3.31}
        \varphi_0^{\ast}, \varphi_1^{\ast}, \varphi_2^{\ast},\dotsc ; \psi_1^{\ast}, \psi_2^{\ast},\dotsc
    \end{equation}
    is called the conjugate canonical orthonormal basis with respect to the canonical orthonormal basis \eqref{Nr.3.16}.
\end{defn}

We note that $\varphi_0=\varphi_0^*=1$.
Similarly as (\ref{Nr.3.8}) the identity
\begin{equation}\label{Nr.3.38}
    \bigvee_{k=0}^{n} [(U_\mu^\times)^\ast]^k {\bf 1} 
    = \left( \bigvee_{k=0}^{n-1} (T_\mu^\ast)^k G_\mu^\ast (1) \right) \oplus \mathbb C_\mathbb T
\end{equation}
can be verified. Thus,
\begin{equation}\label{Nr.3.39}
    \gH_{\mu,\gF} = \bigvee_{k=1}^\infty \varphi_k,\quad \gH_{\mu,\gG}
    = \bigvee_{k=1}^\infty \varphi_k^{\ast},
\end{equation}
\begin{equation}\label{Nr.3.40}
    \gH_{\mu,\gF}^\perp = \bigvee_{k=1}^\infty \psi_k,\quad
    \gH_{\mu,\gG}^\perp = \bigvee_{k=1}^\infty \psi_k^{\ast}.
\end{equation}

In 
Subsection~\ref{sec3.1} (see \eqref{3.8}) 
the unitary operator $\cU$ was introduced  which
maps the elements of the canonical basis (\ref{Nr.3.16}) onto the corresponding elements of the
conjugate canonical basis (\ref{Nr.3.31}). More precisely, we consider the operator
\begin{equation}\label{5-16}
    \cU_\mu \varphi_n = \varphi_n^{\ast},\quad n\in \N_0,
    \qquad \mbox{and} \qquad 
    \cU_\mu \psi_n = \psi_n^{\ast},\quad n\in \N.
\end{equation}
The operator $\cU_\mu$ is related to the conjugation operator in $L_\mu^2$. Namely, if
$f\in L_\mu^2$ and if
\begin{equation*}
    f =\sum_{k=0}^\infty \alpha_k \varphi_k + \sum_{k=1}^\infty
    \beta_k\psi_k,
\end{equation*}
then
\begin{equation*}
    f^{\ast} = \sum_{k=0}^\infty \alpha_k^{\ast} \varphi_k^{\ast} + \sum_{k=1}^\infty \beta_k^{\ast} \psi_k^{\ast} 
    = \sum_{k=0}^\infty \alpha_k^{\ast} \cU\varphi_k + \sum_{k=1}^\infty \beta_k^{\ast} \cU\psi_k.
\end{equation*}

From \fref{5-16} it follows that
\begin{align*}
    \fEl{\cU}{\mu}  ~:~ \fEl{\gH}{\mu}  \longrightarrow \fEl{\gH}{\mu}  \;, \qquad \cU_\mu(\bf 1)=\bf 1 \,.
\end{align*}
Let
\begin{align}       \label{5-17}
    \fEl{\cU}{ \fEl{\gH}{\mu} } 
    :=   \Rstr_{ \fEl{\gH}{\mu} }{\fEl{\cU}{\mu}}  \qPunkt
\end{align}
Then, obviously,
\begin{align}       \label{5-18}
    \fEl{\cU}{ \fEl{\gH}{\mu} } \fEl{\varphi}{n}  = \Adj{ \fEl{\varphi}{n} }
    \qquad
    \mbox{and}
    \qquad
    \fEl{\cU}{ \fEl{\gH}{\mu} } \fEl{\psi}{n}  = \Adj{ \fEl{\psi}{n} }
    \qKomma
    \qquad
    n \in \N \qPunkt
\end{align}
Clearly, the system $\fKlam{ \Adj{\psi} }{n}{1}{\infty}$
is an orthonormal basis in the space $\mkEl{ \orth{\gH} }{\mu}{\gG}$.
This system will turn out to be the special orthonormal basis of
the space $\mkEl{ \orth{\gH} }{\mu}{\gG}$ mentioned
at the beginning of this section. Thus, the matrix representation
of the operator
\begin{align*}
    \mkEl{B}{\mu}{\gG}  
    ~:~ \mkEl{\gH}{\mu}{\gF}  \longrightarrow \mkEl{ \orth{\gH} }{\mu}{\gG}
\end{align*}
will be considered with respect to the orthonormal bases
\begin{align}       \label{5-19}
    \fKlam{\varphi}{n}{1}{\infty}
    \qquad
    \qquad
    \mbox{and}
    \qquad
    \qquad
    (\psi_n^*)_{n=1}^\infty
\end{align}
of the spaces $\mkEl{\gH}{\mu}{\gF}$ and $\mkEl{ \orth{\gH} }{\mu}{\gG}$,
respectively.
Let (see \eqref{3.9})
\begin{align}       \label{5-20}
    \begin{pmatrix}
        \cR   & \cL   \\[5pt]
        \cP   & \cQ
    \end{pmatrix}
\end{align}
be the matrix representation of the operator $\fEl{\cU}{ \fEl{\gH}{\mu} }$
with respect to the canonical basis \fref{Nr.3.16b} of the space $\fEl{\gH}{\mu}$.
Then, from \fref{5-18} we infer that the columns
\begin{align*}
    \begin{pmatrix}
        \cR   \\[5pt]
        \cP
    \end{pmatrix}
    \qquad
    \qquad
    \mbox{and}
    \qquad
    \qquad
    \begin{pmatrix}
        \cL   \\[5pt]
        \cQ
    \end{pmatrix}
\end{align*}
of the block-matrix \fref{5-20} are the coefficients in the series
developments of $\Adj{ \fEl{\varphi}{n} }$ and $\Adj{ \fEl{\psi}{n} }$
with respect to the canonical basis \fref{Nr.3.16b}. If $h \in \fEl{\gH}{\mu}$
then clearly
\begin{align}       \label{5-21}
    \fEl{P}{ \mkEl{ \orth{\gH} }{\mu}{\gG} }  h
    = \sum_{k = 1}^{\infty}
    { 
        \rklamPaar{h}{ \Adj{ \fEl{\psi}{k} } } 
        \Adj{ \fEl{\psi}{k} }
    } \qPunkt
\end{align}
Thus, the matrix representation of the operator
$\fEl{P}{ \mkEl{ \orth{\gH} }{\mu}{\gG} }$ considered as
an operator acting between $\fEl{\gH}{\mu}$ and
$\mkEl{ \orth{\gH} }{\mu}{\gG}$ equipped with the
orthonormal bases \fref{Nr.3.16b} and $\fKlam{ \Adj{\psi} }{n}{1}{\infty}$
has the form
\begin{align*}
    \rklamPaar{ \Adj{\cL} }{ \Adj{\cQ} }  \qPunkt
\end{align*}
>From this and the shape \fref{4-thm3-1} of the operator
$\mkEl{B}{\mu}{\gG}$, we obtain the following result.

\begin{thm}[\zitaa{DFK}{\cthm{5.4}}]\label{5-thm4}
    Let $\mu \in \rklamFunk{ \fEl{\cM^{1}}{+} }{ \dT }$ be
    a Szeg\H{o} measure. Then the matrix of the operator
    \begin{align*}
        \mkEl{B}{\mu}{\gG}  
        ~:~ \mkEl{\gH}{\mu}{\gF}  \longrightarrow \mkEl{ \orth{\gH} }{\mu}{\gG}
    \end{align*}
    with respect to the orthonormal bases $\fKlam{\varphi}{k}{1}{\infty}$
    and $\fKlam{ \Adj{\psi} }{n}{1}{\infty}$ of the spaces
    $\mkEl{\gH}{\mu}{\gF}$ and $\mkEl{ \orth{\gH} }{\mu}{\gG}$,
    respectively, is given by $\Adj{\cL}$ where $\cL$ is the block of
    the matrix given in \fref{5-20}. 
\end{thm}

Now Theorem \ref{4-thm3} can be reformulated in the following way.

\begin{cor}[\zitaa{DFK}{\ccor{5.5}}]\label{5-cor5}
    Let $\mu \in \rklamFunk{ \fEl{\cM^{1}}{+} }{ \dT }$ be
    a Szeg\H{o} measure. Then the Riesz projection $\fEl{P}{+}$
    is bounded in $\fEl{L^2}{\mu}$ if and only if 
    $\cL^*$ is boundedly invertible in $l_2$ where $\cL$ is the block of the matrix given in 
    \fref{5-20}.
\end{cor}

In 
Subsection~\ref{sec3.2} the matrix $\cL$ was expressed in terms of the Schur parameters associated with the measure $\mu$.
In order to write down this matrix we introduce the necessary notions and terminology used in 
this subsection. 
The matrix
$\cL$ expressed in terms of the corresponding Schur parameter
sequence will the denoted by $\rklamFunk{\cL}{\gamma}$.

Let us recall the class of sequences given in Definition~\ref{de3.5}:
$$
\Gamma l_2
:=\left\{\gamma=(\gamma_j)_{j=0}^\infty\in l_2:\gamma_j\in\D,j\in\N_0\right\}.
$$
Let us mention the following well--known fact (see, for example,
Remark \ref{s0-r1})

\begin{prop}       \label{5-prop5}
    Let $\mu \in \rklamFunk{ \fEl{\cM^{1}}{+} }{ \dT }$. Then
    $\mu$ is a Szeg\H{o} measure if and only if
    $ \gamma$ belongs to $\Gamma \fEl{l}{2} $.
\end{prop}

The following result is important for us (see Corollary~\ref{cor3.6}).

\begin{thm}       \label{5-thm7}
    Let $\mu \in \rklamFunk{ \fEl{\cM^{1}}{+} }{ \dT }$ be
    a Szeg\H{o} measure and let $\gamma \in \Gamma$ be the Schur
    parameter sequence associated with $\mu$. Then
    $ \gamma  \in \Gamma \fEl{l}{2} $
    and the block $\cL$ of the matrix \fref{5-20} has the form
    \begin{align}\label{5-thm7-1}
        &  \rklamFunk{\cL}{\gamma}
        =  \begin{pmatrix}        
            \Pi_1 & 0 & 0 &  \ldots \\
            \Pi_2L_1(W\gamma) & \Pi_2 & 0 & \ldots  \\
            \Pi_3L_2(W\gamma) & \Pi_3L_1(W^2\gamma) & \Pi_3 & \ldots  \\
            \vdots & \vdots & \vdots & \ddots \\
            \Pi_nL_{n-1}(W\gamma) & \Pi_nL_{n-2}(W^2\gamma) & \Pi_nL_{n-3}(W^3\gamma) & \ldots \\
            \vdots & \vdots & \vdots &          
        \end{pmatrix}    \qKomma
    \end{align}
    where products $\fEl{\Pi}{n}$, functions $\rklamFunk{ \fEl{L}{n} }{\gamma}$ and the coshift $W$ are given via the formulas \fref{3.17}, \fref{3.15}, and \fref{3.14}, respectively.
\end{thm}

\begin{rem}[\zitaa{DFK}{\crem{5.8}}]\label{5-rem8}
    It follows from Theorems \ref{5-thm4} and \ref{5-thm7}
    that the matrix representation of the operator
    \begin{align*}
        \mkEl{B}{\mu}{\gG}  
        ~:~ \mkEl{\gH}{\mu}{\gF}  \longrightarrow \mkEl{ \orth{\gH} }{\mu}{\gG}
    \end{align*}
    with respect to the orthonormal bases $\fKlam{\varphi}{k}{1}{\infty}$
    and $\fKlam{\psi^*}{n}{1}{\infty}$ of the spaces $\mkEl{\gH}{\mu}{\gF}$
    and $\mkEl{ \orth{\gH} }{\mu}{\gG}$, respectively, is given by the
    matrix $\rklamFunk{ \Adj{\cL} }{\gamma}$, where $\rklamFunk{\cL}{\gamma}$
    has the form \fref{5-thm7-1}.
\end{rem}

\subsection{Characterization of Helson-Szeg\H o measures in terms of the Schur parameters of the associated Schur function}
\label{s6}

The first criterion which characterizes Helson-Szeg\H o measures in the associated Schur parameter sequence was already 
obtained. It follows by combination of Theorem \ref{s0-t1}, Theorem \ref{4-thm3}, Proposition \ref{5-prop5},
Theorem \ref{5-thm7}, and Remark \ref{5-rem8}. This leads us to the following theorem, which is one of the main results of this section.

\begin{thm}[\zitaa{DFK}{\cthm{6.1}}]\label{s6-t1}
    Let $\mu\in\cM_+^1(\T)$ and let $\gamma\in\Gamma$ be the sequence of Schur parameters associated with $\mu$.
    Then $\mu$ is a Helson-Szeg\H o measure if and only if $\gamma\in\Gamma l_2$ and the operator $\cL^*(\gamma)$,
    which is defined in $l_2$ by the matrix \eqref{5-thm7-1}, is boundedly invertible.
\end{thm}

\begin{cor}[\zitaa{DFK}{\ccor{6.2}}]\label{s6-c1}
    Let $\mu\in\cM_+^1(\T)$ and let $\gamma\in\Gamma$ be the sequence of Schur parameters associated with $\mu$.
    Then $\mu$ is a Helson-Szeg\H o measure if and only if $\gamma\in\Gamma l_2$ and there exists some positive constant $C$
    such that for each $h\in l_2$ the inequality
    \begin{align}\label{s6-c1-1}
        \|\cL^*(\gamma)h\|\ge C\|h\|
    \end{align}
    is satisfied.
\end{cor}

\begin{proof}
    First suppose that $\gamma\in\Gamma l_2$ and that there exists some positive constant $C$ such that for each $h\in l_2$ the inequality \eqref{s6-c1-1} is satisfied. From the shape \eqref{5-thm7-1} of the operator $\cL(\gamma)$ it follows 
    immediately that 
    $\ker\cL(\gamma)=\{0\}$. Thus, $\overline{\Ran\cL^*(\gamma)}=l_2$. From \eqref{s6-c1-1} it follows that
    the operator $\cL^*(\gamma)$ is invertible and that the corresponding inverse operator $\big(\cL^*(\gamma)\big)^{-1}$ is bounded and satisfies
    $$ \big\|\big(\cL^*(\gamma)\big)^{-1}\big\|\le\frac1C $$
    where $C$ is taken from \eqref{s6-c1-1}. Since $\cL^*(\gamma)$ is a bounded linear operator, the operator $[\cL^*(\gamma)]^{-1}$ is closed. Thus $\Ran\cL^*(\gamma)=l_2$ and, consequently, the operator $\cL^*(\gamma)$ is boundedly invertible. 
    Hence, Theorem \ref{s6-t1} yields that $\mu$ is a Helson-Szeg\H o measure. 
    If $\mu$ is a Helson-Szeg\H o measure, then Theorem \ref{s6-t1} yields that $\cL^*(\gamma)$ is boundedly invertible.
    Hence, condition \eqref{s6-c1-1} is trivially satisfied.
\end{proof}

It should be mentioned that a result similar to Theorem \ref{s6-t1} was
proved earlier using a different method in \cite[Definition 4.6, Proposition 4.7 and Theorem 4.8]{GKPY}. More specifically, it was shown that a measure $\mu$ is a
Helson-Szeg\H o measure if and only if some infinite matrix $\cM$ (which is
defined in \cite[formulas (4.1) and (4.2)]{GKPY}) generates a bounded operator
in $\ell^2$.
It was also shown that the boundedness of $\cM$ is equivalent to the boundedness of another operator matrix $\cL$ defined in formula (6.4) of \cite{GKPY}.

In order to derive criteria in another way we need some statements on the operator $\cL(\gamma)$ defined in Corollary~\ref{cor3.6} (see \eqref{3.42}). 
The following result 
plays an important role in the study of the matrix $\cL(\gamma)$.
Namely, Theorem~\ref{thm3.11} describes the multiplicative structure of $\cL(\gamma)$ and indicates connections to the backward shift:

\textit{
    It holds that
    \begin{equation}       \label{s6-2A-1A}
        \cL(\gamma) = \gM(\gamma) \cdot \cL(W \gamma)
    \end{equation}
    where
    \begin{equation}       \label{s6-2A-1B}
        \gM(\gamma) :=
        \rklam{
            \begin{smallmatrix}
                D_{\gamma_1} & 0 & 0 & \cdots & 0 & \cdots\\
                -\gamma_1\overline\gamma_2 & D_{\gamma_2} & 0 & \cdots & 0 & \cdots \\
                -\gamma_1D_{\gamma_2}\overline\gamma_3 & -\gamma_2\overline\gamma_3 & D_{\gamma_3} & \hdots & 0 & \cdots \\
                \vdots & \vdots & \vdots & & \vdots \\
                -\gamma_1\left( \prod_{j=2}^{n-1}D_{\gamma_j}\right)\overline\gamma_n & -\gamma_2\left( \prod_{j=3}^{n-1}D_{\gamma_j}\right)\overline\gamma_n & -\gamma_3\left( \prod_{j=4}^{n-1}D_{\gamma_j}\right)\overline\gamma_n & \cdots & D_{\gamma_n} & \cdots \\
                \vdots & \vdots & \vdots &  & \vdots & 
            \end{smallmatrix}
        }
    \end{equation}
    and $D_{\gamma_j} := \sqrt{ 1 - \lvert \gamma_j \rvert^2 }, \, j \in \N_0$.
    The matrix $\gM(\gamma)$ satisfies
    \begin{equation}       \label{s6-2A-1C}
        I - \gM(\gamma) \gM^*(\gamma) = \eta(\gamma) \eta^*(\gamma)
    \end{equation}
    where
    \begin{equation}       \label{s6-2A-1D}
        \eta(\gamma)
        := {\rm col} \left(
        \overline{ \gamma_1 } ,  \, 
        \overline{ \gamma_2 } D_{\gamma_1} ,  \, 
        \ldots ,
        \overline{ \gamma_n } \prod_{j = 1}^{n - 1}{ D_{\gamma_j} } ,  \, 
        \ldots
        \right)
    \end{equation}
}

The multiplicative structure \eqref{s6-2A-1A} of $\cL(\gamma)$ gives us some hope that the boundedness of the operator $\cL^*(\gamma)$ can be reduced to a constructive condition on the Schur paramters via convergence of some infinite products (series). This is a promising direction for future work on this problem.

Let $\gamma\in\Gamma l_2$.
For each $n\in\N$ we set (see \eqref{5.3})
\begin{equation}\label{frakLn}
    \fL_n(\gamma):=\begin{pmatrix}
        \Pi_1 & 0 & 0 & \hdots & 0 \\
        \Pi_2L_1(W\gamma) & \Pi_2 & 0 & \hdots & 0 \\
        \Pi_3L_2(W\gamma) & \Pi_3L_1(W^2\gamma) & \Pi_3 & \hdots & 0 \\
        \vdots & \vdots & \vdots & & \vdots \\
        \Pi_nL_{n-1}(W\gamma) & \Pi_nL_{n-2}(W^2\gamma) & \Pi_nL_{n-3}(W^3\gamma) & \hdots & \Pi_n
    \end{pmatrix}.
\end{equation}
The matrices introduced in (\ref{frakLn}) will play an important role in our investigations. Now we turn our attention to some properties of the matrices $\fL_n(\gamma)$, $n\in\N $, which will later be of use.
From Corollary~\ref{cor5.2} 
we get the following result.

\begin{lem}\label{lem11a}
    Let $\gamma=(\gamma_j)_{j=0}^\infty\in\Gamma l_2$ and let $n\in\N$. Then the matrix $\fL_n(\gamma)$ defined by (\ref{frakLn}) is contractive.
\end{lem}

We continue with some asymptotical considerations.

\begin{lem}\label{lem12neu}
    Let $\gamma=(\gamma_j)_{j=0}^\infty\in\Gamma l_2$. Then:
    \begin{itemize}
        \item[(a)] $\lim_{k\rightarrow\infty}\Pi_k=1$.
        \item[(b)] For each $j\in\N$,\: $\lim_{m\rightarrow\infty}L_j(W^m\gamma)=0$.
        \item[(c)] For each $n\in\N$,\: $\lim_{m\rightarrow\infty}\fL_n(W^m\gamma)=I_n$.
    \end{itemize}
\end{lem}

\begin{proof}
    The choice of $\gamma$ implies the convergence of the infinite product $\prod_{k=0}^\infty D_{\gamma_k}$. 
    This yields (a). 
    Assertion (b) is an immediate consequence of the definition of the sequence $(L_j(W^m\gamma))_{m=1}^\infty$ (see (\ref{3.14}) and (\ref{3.15})). By inspection of the sequence\\
    $(\fL_n(W^m\gamma))_{m=1}^\infty$ one can immediately see that the combination of (a) and (b) yields the assertion of (c).
\end{proof}

The following result is similar to the multiplicative representation \eqref{s6-2A-1A} (see Lemma~\ref{lm5.3} and Corollary~\ref{cor5.4}):

\textit{
Let $\gamma=(\gamma_j)_{j=0}^\infty\in\Gamma l_2$ and let $n\in\N$. Then
\begin{equation}\label{LnProd}
    \fL_n(\gamma)=\gM_n(\gamma)\cdot\fL_n(W\gamma),
\end{equation}
so that
\begin{equation}\label{LnDecomp}
            \fL_n(\gamma)=\stackrel{\longrightarrow}{\prod_{k=0}^{\infty}}\gM_n(W^k\gamma),
        \end{equation}
where
\begin{multline}\label{frakMn}
    \gM_n(\gamma)\\
    :=
    \begin{psmallmatrix}
        D_{\gamma_1} & 0 & 0 & \hdots & 0 \\
        -\gamma_1\overline\gamma_2 & D_{\gamma_2} & 0 & \hdots & 0 \\
        -\gamma_1D_{\gamma_2}\overline\gamma_3 & -\gamma_2\overline\gamma_3 & D_{\gamma_3} & \hdots & 0 \\
        \vdots & \vdots & \vdots & & \vdots \\
        -\gamma_1\left(\prod_{j=2}^{n-1}D_{\gamma_j}\right)\overline\gamma_n & -\gamma_2\left(\prod_{j=3}^{n-1}D_{\gamma_j}\right)\overline\gamma_n & -\gamma_3\left(\prod_{j=4}^{n-1}D_{\gamma_j}\right)\overline\gamma_n & \hdots & D_{\gamma_n}
    \end{psmallmatrix}.
\end{multline}
Moreover, $\gM_n(\gamma)$ is a nonsingular matrix which fulfills
\begin{equation}\label{MnForm}
    I_n-\gM_n(\gamma)\gM_n^*(\gamma)=\eta_n(\gamma)\eta_n^*(\gamma),
\end{equation}
where
\begin{equation}\label{eta}
    \eta_n(\gamma)
    :=\left(\overline{\gamma_1},\overline{\gamma_2}D_{\gamma_1},\dotsc,\overline{\gamma_n}
    \Bigg(\prod_{j=1}^{n-1}D_{\gamma_j}\Bigg)\right)^T.
\end{equation}
}

Now we state the next main result of this section.

For $h=(z_j)_{j=1}^\infty\in l_2$ and $n\in\N$ we set 
$$ h_n:=(z_1,\dotsc,z_n)^\top\in\C^n .$$

\begin{thm}[\zitaa{DFK}{\cthm{6.8}}]\label{s6-t2}
    Let $\mu\in\cM_+^1(\T)$ and let $\gamma\in\Gamma$ be the sequence of Schur parameters associated with $\mu$.
    Then $\mu$ is a Helson-Szeg\H o measure if and only if $\gamma\in\Gamma l_2$ and there exists some positive constant $C$
    such that for all $h\in l_2$ the inequality
    \begin{align}\label{s6-t2-2}
        \lim\limits_{n\to\infty}\lim\limits_{m\to\infty}\left\|\Bigg(\stackrel\longleftarrow{\prod_{k=0}^m}\gM_n^*(W^k\gamma)\Bigg)h_n\right\|\ge C\|h\|
    \end{align}
    is satisfied.
\end{thm}

\begin{proof}
    In view of \eqref{LnDecomp} and condition (c) in Lemma \ref{lem12neu} the condition \eqref{s6-t2-2}
    is equivalent to the fact that for all $h\in l_2$ the inequality 
    \begin{align}\label{s6-t2-3}
        \lim\limits_{n\to\infty}\|\cL_n^*(\gamma)h_n\|\ge C\|h\|
    \end{align}
    is satisfied. This inequality is equivalent to the inequality \eqref{s6-c1-1}.
\end{proof}

Theorem \ref{s6-t2} leads to an alternate proof of an interesting sufficient condition for a Szeg\H o measure to be a Helson-Szeg\H o measure
(see Theorem \ref{6-thm9}).
To prove this result we use the next inequalities, which follow from Lemma~\ref{lem211}.

\begin{lem}[\zitaa{DFK}{\ccor{6.10}}]\label{s6-c2}
    Let the assumptions of Lemma \ref{lem211} be satisfied. Then for each $h\in\C^n$ the inequalities 
    \begin{align}\label{s6-c2-1}
        \|\gM h\|\ge(1-\|\eta\|^2)^{\frac12}\,\|h\|
    \end{align}
    and
    \begin{align}\label{s6-c2-2}
        \|\gM^*h\|\ge(1-\|\eta\|^2)^{\frac12}\,\|h\|
    \end{align}
    are satisfied.
\end{lem}

\begin{proof}
    Applying \eqref{lem211_38} and \eqref{lem211_44} we get for $h\in\C^n$ the relation
    \begin{align*}
        \|h\|^2-\|\gM h\|^2=\big((I-\gM^*\gM)h,h\big)=\lvert(h,\widetilde\eta)\rvert^2\le\|\widetilde\eta\|^2\,\|h\|^2=\|\eta\|^2\,\|h\|^2.
    \end{align*}
    This implies \eqref{s6-c2-1}. Analogously, \eqref{s6-c2-2} can be verified.
\end{proof}

\begin{cor}[\zitaa{DFK}{\ccor{6.11}}]\label{s6-c3}
    Let $\gamma\in\Gamma l_2$, and let the matrix $\gM_n(\gamma)$ be defined via \eqref{frakMn}. Then for all $h\in\C^n$
    the inequalities
    \begin{align}\label{s6-c3-1}
        \|\gM_n(\gamma)h\|\ge\Bigg(\prod_{j=1}^n D_{\gamma_j}\Bigg)\|h\|
    \end{align}
    and
    \begin{align}\label{s6-c3-2}
        \|\gM_n^*(\gamma)h\|\ge\Bigg(\prod_{j=1}^n D_{\gamma_j}\Bigg)\|h\|
    \end{align}
    are satisfied.
\end{cor}

\begin{proof}
    The matrix $\rklamFunk{ \fEl{\gM}{n} }{\gamma}$
    satisfies the conditions of Lemma \ref{lem211}. Here the
    vector $\eta$ has the form \eqref{eta}. It remains
    only to mention that in this case we have
    \begin{align}\label{s6-c3-3}
        1 - \norm{\eta}^2
        &=	1 - \abs{ \fEl{\gamma}{1} }^2
        - \abs{ \fEl{\gamma}{2} }^2
        \rklam{ 1 - \abs{ \fEl{\gamma}{1} }^2 }
        - \ldots
        - \abs{ \fEl{\gamma}{n} }^2
        \eklam{ 
            \prod_{j = 1}^{n-1}{
                \rklam{1 - \abs{ \fEl{\gamma}{j} }^2}
            } 
        }	\nonumber\\[5pt]
        &=	\prod_{j = 1}^{n}\rklam{
            1 - \abs{ \fEl{\gamma}{j} }^2
        }	\qPunkt
    \end{align}
\end{proof}

The above consideration lead us to an alternate proof for a nice
sufficient criterion for the Helson-Szeg\H{o} property of a measure
$\mu \in \rklamFunk{ \fEl{\cM^1}{+} }{\dT}$
which is expressed in terms of the modules of the associated
Schur parameter sequence.

Regarding the history of Theorem \ref{6-thm9}, it should be mentioned
that, in view of a theorem by B. L. Golinskii and I. A. Ibragimov \cite{5a},
the convergence of the infinite product in \eqref{6-24} is equivalent
to the property that $\mu$ is absolutely continuous with respect to
the Lebesgue measure. The corresponding density is then of the form
$\exp{g}$, where $g$ is a real Besov-class function. A Theorem of V. V. Peller's
\cite{11A} states that every function of this form is a density of a
Helson-Szeg\H{o} measure. This topic was also discussed in detail
in \cite{GKPY}.

\begin{thm}[\zitaa{DFK}{\cthm{6.12}}]\label{6-thm9}
    Let $\mu \in \rklamFunk{ \fEl{\cM^1}{+} }{\dT}$
    and let $\gamma	\in	\Gamma$ be the Schur parameter sequence
    associated with $\mu$. If $\gamma \in \Gamma \fEl{l}{2}$ and
    the infinite product
    \begin{align}		\label{6-24}
        \prod_{k = 1}^{\infty}{ 
            \prod_{j = k}^{\infty}{  
                \rklam{ 1 - \abs{ \fEl{\gamma}{j} }^2 }
            } 
        }
    \end{align}
    converges, then $\mu$ is a Helson-Szeg\H{o} measure.
\end{thm}

\begin{proof}
    Applying successively the estimate \eqref{s6-c3-2}, we
    get for all $m, \, n \in \N$ and all vectors $h \in \dC^n$
    the chain of inequalities
    \begin{align}
        &	\hspace{-18pt}
        \norm{  
            \eklam{
                \overleftarrow{
                    \prod_{k = 0}^{m}
                }
                {
                    \rklamFunk{ \fEl{\gM^*}{n} }{ W^k \gamma}
                }
            } h
        }	\nonumber	\\
        &=	\norm{  
            \rklamFunk{ \fEl{\gM^*}{n} }{ W^m \gamma}
            \eklam{
                \overleftarrow{
                    \prod_{k = 0}^{m-1}
                }
                {
                    \rklamFunk{ \fEl{\gM^*}{n} }{ W^k \gamma}
                }}
            h
        }	\nonumber	\displaybreak[0]	\\[5pt]
        &\geq	\prod_{j = m+1}^{m+n}{ \fEl{D}{ \fEl{\gamma}{j} } }
        \norm{  
            \eklam{
                \overleftarrow{\prod_{k = 0}^{m-1}}
                {
                    \rklamFunk{ \fEl{\gM^*}{n} }{ W^k \gamma}
                }
            } h
        }	
        \nonumber	\\[5pt]
        &\geq		\ldots\nonumber	\\[5pt]
        &\geq	\rklam{
            \prod_{j = m+1}^{m+n}{ \fEl{D}{ \fEl{\gamma}{j} } }
        }
        \cdot
        \rklam{
            \prod_{j = m}^{m+n-1}{ \fEl{D}{ \fEl{\gamma}{j} } }
        }
        \cdot
        \ldots
        \cdot
        \rklam{
            \prod_{j = 1}^{n}{ \fEl{D}{ \fEl{\gamma}{j} } }
        }
        \norm{h}	\nonumber	\displaybreak[0]	\\[5pt]
        &\geq	\rklam{
            \prod_{j = m+1}^{\infty}{ \fEl{D}{ \fEl{\gamma}{j} } }
        }
        \cdot
        \rklam{
            \prod_{j = m}^{\infty}{ \fEl{D}{ \fEl{\gamma}{j} } }
        }
        \cdot
        \ldots
        \cdot
        \rklam{
            \prod_{j = 1}^{\infty}{ \fEl{D}{ \fEl{\gamma}{j} } }
        }
        \norm{h}	\nonumber	\displaybreak[0]	\\[5pt]
        &=	\rklam{
            \prod_{k = 1}^{m+1}{
                \prod_{j = k}^{\infty}{ \fEl{D}{ \fEl{\gamma}{j} } }
            }
        }
        \norm{h}	\nonumber	\\[5pt]
        &\geq	\rklam{
            \prod_{k = 1}^{\infty}{
                \prod_{j = k}^{\infty}{ \fEl{D}{ \fEl{\gamma}{j} } }
            }
        }
        \norm{h}
    \end{align}
    >From this inequality it follows \eqref{s6-t2-2} where
    $$ C=\prod_{k=1}^\infty\prod_{j=k}^\infty D_{\gamma_j} \;.$$
    Thus, the proof is complete.
\end{proof}

Taking into account that the convergence of the infinite product \eqref{6-24}
is equivalent to the strong Szeg\H o condition
\begin{equation*}
    \sum_{k = 1}^{\infty}{ k \cdot \abs{ \fEl{\gamma}{k} }^2 } < \infty ,
\end{equation*}
Theorem \ref{6-thm9} is an immediate consequence of \cite[Theorem 5.3]{GKPY}.
The proof of \cite[Theorem 5.3]{GKPY} is completely different from the above
proof of Theorem \ref{6-thm9}. It is based on a scattering formalism using
CMV matrices (For a comprehensive exposition on CMV matrices, we refer the
reader to Chapter 4 in the monograph Simon \cite{S}.)

The aim of our next considerations is to characterize the Helson-Szeg\H o
property of a measure $\mu \in \rklamFunk{ \cM^{1}_{+} }{\dT}$ in terms
of some infinite series formed from its Schur parameter sequence. The
following result follows from Lemma~\ref{lem23}.

\begin{thm}[\zitaa{DFK}{\cthm{6.13}}]\label{6-thm13}
    Let $\gamma = \fKlam{\gamma}{j}{0}{\infty} \in \Gamma \fEl{\ell}{2}$ and let
    \begin{align}		\label{6-100}
        \rklamFunk{ \cA }{\gamma}
        :=  I - \rklamFunk{ \cL }{\gamma} \rklamFunk{ \Adj{\cL} }{\gamma}
    \end{align}
    where $\rklamFunk{ \cL }{\gamma}$ is given by \eqref{5-thm7-1}.
    Then $\rklamFunk{ \cA }{\gamma}$ satisfies the inequalities
    \begin{align}		\label{6-101}
        0 \leq 
        \rklamFunk{ \cA }{\gamma}
        \leq I
    \end{align}
    and admits the strong convergent series decomposition
    \begin{align}		\label{6-102}
        \rklamFunk{ \cA }{\gamma}
        = \sum_{j = 0}^{\infty}{ 
            \rklamFunk{ \fEl{ \xi }{j} }{ \gamma } 
            \rklamFunk{ \fEl{ \Adj{\xi} }{j} }{ \gamma }
        }
    \end{align}
    where
    \begin{align}		\label{6-103}
        \rklamFunk{ \fEl{ \xi }{0} }{ \gamma } 
        :=  \rklamFunk{ \eta }{ \gamma } ,
        \qquad
        \qquad
        \rklamFunk{ \fEl{ \xi }{j} }{ \gamma } 
        := \eklam{
            \overrightarrow{
                \prod_{k = 0}^{j - 1}
            }
            {
                \rklamFunk{\gM}{ W^k \gamma}
            }
        }
        \rklamFunk{ \eta }{ W^j \gamma } ,
        j \in \N_0,
    \end{align}
    and $\rklamFunk{\gM}{\gamma}$, $\rklamFunk{\eta}{\gamma}$
    and $W$ are given by \eqref{s6-2A-1B}, \eqref{s6-2A-1D} and \eqref{3.14}, respectively.
\end{thm}

The last main result of this section is the following statement, which is an
immediate consequence of Theorem \ref{s6-t1} and Theorem \ref{6-thm13}.

\begin{thm}[\zitaa{DFK}{\cthm{4.16}}]\label{6-thm14}
    Let $\mu \in \rklamFunk{ \cM^{1}_{+} }{\dT}$ and let $\gamma \in \Gamma$
    be the sequence of Schur parameters associated with $\mu$. Then $\mu$ is a
    Helson-Szeg\H o measure if and only if $\gamma \in \Gamma \fEl{\ell}{2}$
    and there exists some positive constant $\varepsilon \in \rklam{0, \, 1}$ such
    that the inequality
    \begin{align}		\label{6-104}
        \sum_{j = 0}^{\infty}{ 
            \rklamFunk{ \fEl{ \xi }{j} }{ \gamma } 
            \rklamFunk{ \fEl{ \Adj{\xi} }{j} }{ \gamma }
        }
        \leq \rklam{ 1 - \varepsilon} I
    \end{align}
    is satisfied, where the vectors
    $\rklamFunk{ \fEl{ \xi }{j} }{ \gamma }, \, j \in \N_0$, are given by
    \eqref{6-103}.
\end{thm}

We note that the inequality \eqref{6-104} can be considered as a rewriting
of condition \eqref{s6-t2-2} in an additive form.


\subsection*{Acknowledgment}
The first author would like to express special thanks to Professor Tatjana Eisner for her generous support.
He also thanks the Max Planck Institute for Human Cognitive and Brain Sciences, in particular 
Professor Christian Doeller, Dr Christina Schroeder, and Dr Sebastian Ziegaus.

The first author was also partially supported by the Volkswagen Foundation grant within the frameworks of the international project ``From Modeling and Analysis to Approximation''.

\end{document}